\documentclass[a4paper]{amsart}

\usepackage{amsmath, amssymb, latexsym, amsfonts, pigpen, enumitem}
\usepackage[matrix,arrow,curve]{xy}

\numberwithin{equation}{section}
\theoremstyle{plain}
\newtheorem{theorem}[equation]{Theorem}
\newtheorem{corollary}[equation]{Corollary}
\newtheorem{lemma}[equation]{Lemma}
\newtheorem{proposition}[equation]{Proposition}
\newtheorem{construction}[equation]{Construction}
\newtheorem{question}[equation]{conjecture}
\newtheorem{conjecture}[equation]{conjecture}
\theoremstyle{definition}
\newtheorem{definition}[equation]{Definition}
\newtheorem{notation}[equation]{Notation}
\theoremstyle{remark}
\newtheorem{remark}[equation]{Remark}

\newcommand{\cal}{\mathcal}

\newcommand{\GW}{{\mathrm{GW}}}

\newcommand{\rank}{\operatorname{rank}}

\newcommand{\coker}{\operatorname{coker}}

\newcommand{\Sm}{\mathrm{Sm}}

\newcommand{\chark}{\operatorname{char}}

\newcommand{\pour}{\ar@{}[ur]|(0.2){\text{\pigpenfont G}}}
\newcommand{\podr}{\ar@{}[dr]|(0.2){\text{\pigpenfont A}}}

\DeclareMathOperator{\Ker}{Ker}
\DeclareMathOperator{\Image}{Im}
\DeclareMathOperator{\Coim}{Coim}
\DeclareMathOperator{\Coker}{Coker}
\DeclareMathOperator{\image}{im}

\DeclareMathOperator{\Supp}{Supp}
\DeclareMathOperator{\Suppred }{Supp_{red}}

\newcommand{\eff}{{\mathrm{eff}}}

\newcommand{\DM}{\mathbf{DM}^{\eff}_{nis}}
\newcommand{\DMGWeff}{\mathbf{DM}^\mathrm{GW}_{\eff}}
\newcommand{\DMGWeffunb}{\mathbf{DM}^\mathrm{GW}_{\eff,\mathrm{unb}}}
\newcommand{\DMGW}{\mathbf{DM}^\mathrm{GW}}
\newcommand{\DMWeff}{\mathbf{DM}^\mathrm{W}_{\eff}}
\newcommand{\DMW}{\mathbf{DM}^\mathrm{W}}
\newcommand{\DaffGW}{\mathbf{D}^\mathrm{GW}_{\affl}}

\newcommand{\Daff}{\mathbf{D}^-_{\affl}}

\newcommand{\mSH}{\mathcal{SH}}

\newcommand{\Dba}{\mathbf{D}^-}
\newcommand{\Dun}{\mathbf{D}}

\newcommand{\ShN}{Sh_{Nis}}
\newcommand{\ShNGW}{\ShN(GWCor)}

\newcommand{\PreGW}{Pre(GWCor)}

\newcommand{\ShNGWs}{Sh^{GW}}

\newcommand{\PreGWs}{Pre^{GW}}

\newcommand{\DPreGWs}{\Dba(\PreGWs)}

\newcommand{\cmod}{\text{-}\mathbf{mod}}

\newcommand{\calP}{{\mathcal P}}
\newcommand{\calO}{{\mathcal O}}
\newcommand{\calN}{{\mathcal N}}
\newcommand{\bbG}{{\mathbb G}}
\newcommand{\bG}{{\mathbb G}}
\newcommand{\bGm}{{\mathbb G_m}}

\newcommand{\pri}[1]{{#1^\prime}}

\newcommand{\ppri}[1]{{#1^{\prime\prime}}}
\newcommand{\pppri}[1]{{#1^{\prime\prime\prime}}}
\newcommand{\prl}{\mathbb{P}^1}
\newcommand{\affl}{\mathbb{A}^1}
\newcommand{\aff}{{\mathbb{A}}}

\newcommand{\GXqGY}{ X\times\bbG_m , Y\times\bbG_m}
\newcommand{\GXGY}{ X\times\bbG_m \times Y\times\bbG_m }
\newcommand{\bbGX}{{ X\times\bbG_m }}
\newcommand{\bbGY}{{ Y\times\bbG_m }}
\newcommand{\bGX}{{ X\times\bbG_m }}
\newcommand{\bGY}{{ Y\times\bbG_m }}

\newcommand{\bGmw}{{\mathbb G_m^{\wedge 1}}}
\newcommand{\bGwX}{X\times\bGmw}
\newcommand{\bGwY}{Y\times\bGmw}
\newcommand{\GwXqGwY}{X\times\bGmw,Y\times\bGmw}

\newcommand{\XGwqYGw}{X\times\bGmw,Y\times\bGmw}

\newcommand{\ipGX}{1_\bGX}
\newcommand{\ipGY}{1_\bGY}

\newcommand{\bZ}{\mathbb Z}
\newcommand{\bZGW}{\bZ_{GW}}

\newcommand{\nisGWCor}{{GWCor}_{nis}}

\newcommand{\Simpo}{\Delta^1}
\newcommand{\alphind}{{\alpha\in \mathfrak A}}

\begin{document}

\title[Cancellation for GW-correspondences and Witt-correspondences]{Cancellation theorem for Grothendieck-Witt-correspondences and Witt-correspondences.}
\author{Andrei Druzhinin}
\address{Chebyshev Laboratory, St. Petersburg State University, 14th Line V.O., 29B, Saint Petersburg 199178 Russia}
\email{andrei.druzh@gmail.com}
\thanks{Research is supported by the Russian Science Foundation grant 14-21-00035}

\begin{abstract}
The cancellation theorem for Grothendieck-Witt--cor\-res\-pon\-den\-ces and Witt-correspon\-dences between smooth varieties over an infinite perfect field $k$, $\chark k \neq 2$, is proved. 

The result implies the
the canonical functor $\Sigma_{\bGmw}^\infty\colon \mathbf{DM}^\mathrm{GW}_\mathrm{eff}(k)\to \mathbf{DM}^\mathrm{GW}(k)$ is fully faithful, 
and for any $X\in Sm_k$ and a homotopy invariant sheaf with GW-transfers $\mathcal F$,

$$
Hom_{\mathbf{DM}^\mathrm{GW}}(M^{GW}(X), \Sigma_{\bGmw}^\infty\mathcal F[i]) \simeq H^i_{Nis}(X,\mathcal F)
,$$
where $\mathbf{DM}^\mathrm{GW}(k)$ is the category of GW-motives obtained by $\bGm$-stabilisation from the category of  
effective Grothendieck-Witt-motives $\mathbf{DM}^\mathrm{GW}_\mathrm{eff}(k)$ 
constructed in \cite{AD_DMGWeff}.
Similarly in the case of Witt-motives $\mathbf{DM}^\mathrm{W}(k)$.

\end{abstract}
\maketitle

\newcommand{\GWCor}{GW^\oplus}
\section{Introduction.}
\subsection{The categories $\mathbf{DM}^\mathrm{GW}(k)$ and $\mathbf{DM}^\mathrm{W}(k)$.} 
This article is devoted to the cancellation theorem in the categories of effective Grothendieck-Witt-motives and Witt-motives.
This is the final article of the series of works (\cite{AD_DMGWeff},\cite{AD_StrHomInv}) devoted to the construction of the categories of GW-motives $\mathbf{DM}^\mathrm{GW}(k)$ and Witt-motives $\mathbf{DM}^\mathrm{W}(k)$ over an infinite perfect field $k$, $\chark k\neq 2$.
The construction follows the Voevodsky-Suslin method
originally used for the construction of the category of the Voevodsky motives $\DM(k)$, see 
\cite{VSF_CTMht_Ctpretr},\cite{VSF_CTMht_DM},\cite{SV_Bloch-Kato}, \cite{MVW_LectMotCohom}.

The construction of $\mathbf{DM}^\mathrm{GW}(k)$ 
starts with 
some additive category of correspondences between smooth varieties $\GWCor$, 
called the GW-cor\-res\-pon\-den\-ces defined in \cite{AD_StrHomInv}. 
Namely,
and for a pair of smooth affine varieties $X$ and $Y$ 
the group $\GWCor(X,Y)$ is defied 
as Grothendieck-Witt-group of quadratic space $(P,q)$,
with $P\in k[Y\times X]\mod$ such that $P$ is finitely generated projective over $k[X]$
and $q\colon P\simeq Hom_{k[X]}(P,k[X])$ being $k[Y\times X]$-linear isomorphism.

The cancellation theorem proved in the article yields the main property of motives of smooth varieties in $\DMGW(k)$ and $\DMW(k)$,
i.e. isomorphism
\begin{equation}\label{eq:intr:MTC}
Hom_{\mathbf{DM}^{*}(k)}(M^*(X), \Sigma^\infty_{\bGm}\mathcal F[i]) = H^i_{Nis}(X,\mathcal F), \, *\in \{GW,W\},\end{equation}
for a homotopy invariant presheaf with GW-transfers (Witt-transfers) $\mathcal F$ and smooth variety $X$,
where $M^{GW}(X)$ is the complex of Nisnevich sheaves $\GWCor_{nis}(\Delta^\bullet\times -,X)$, 
and similarly for the case of Witt motives.

The category $\DMGW(k)$ gives a version of generalised motivic cohomology theory given by $X\mapsto Hom_{\DMGW}(M^{GW}(X), \mathbb Z_{GW}(i)[j]))$, where $\mathbb Z_{GW}(i) = M^{GW}(\mathbb G_m^{\wedge i})[i]$. As written in \cite[example .. ]{FB_EffSpMotCT} this theory is equivalent to the generalised motivic cohomology defined by C`almes and Fasel in \cite{CF_FinChWittCor},
so the result \cite{} by Garkusha implies that $\mathbf{DM}^\mathrm{GW}(k)[1/p]\simeq \widetilde{\mathbf{DM}}(k)[1/p]$, $p=\chark k$.
Consequently the Garkusha's result \cite{GG_RecRatStMot} implies equivalence $\DMGW(k)_{\mathbb Q}\simeq \mSH(k)_{\mathbb Q}$.
Thus the computation of the GW-motives of smooth varieties 
presented here gives a fibrant replacement functor in $\mSH(k)_{\mathbb Q}$. 
Let's note also that the canonical adjunction $\mSH(k)\leftrightharpoons \DMGW(k)$ is not equivalence.  

It is expected that 
the category $\DMW(k)$ is equivalent to the category of Witt-motives $\mathbf{DM}_W(k)$
constructed by Ananievsky, Levine, Panin in \cite{ALP_WittSh-etsinvert} using the category of modules over the Witt-ring sheaf. 
As proven in \cite{ALP_WittSh-etsinvert} the category $\mathbf{DM}_W(k)$ satisfies the Morel conjecture about Witt-motives, i.e. $\mathbf{DM}_W(k)_{\mathbb Q}\simeq \mathcal{SH}^-(k)_{\mathbb Q}$. In \cite{ALP_WittSh-etsinvert}
such category was constructed .
Usefulness of the reconstruction by the Voevodsky-Suslin method as above is that it gives an explicit fibrant replacements for the motives of smooth varieties that can be useful for computations. 

The functor $\mSH(k)\to \DMW(k)$ in some sense is the algebraic version of the real realisation.
To justify the last statement let's say that it is known that $H^0(X_{\mathbb R}) = Spec_{\mathbb Z}\,W(X)$, $H^i_{sing}(X_{\mathbb R}) = H^i_{Zar}(X,\mathcal I^n)$,
where $X_\mathbb{R}$ is the real realisation of an algebraic variety $X$, and
$\mathcal I$ is sheaf of fundamental ideals in the Witt ring and $n>dim_{Krull}(X)$. 
The the case zero cohomologies is the result by L.~Mahe and J.~Houdebine, \cite{Mahe_SignRCon},\cite{HoundMahe_ComConRAlgVar},
the general case is the result by J.~Jacobson \cite{Jacobson_RealChIWitt}


\subsection{The cancellation theorem in $\mathbf{DM}^\mathrm{GW}_{\mathrm{eff}}(k)$ and $\mathbf{DM}^\mathrm{W}_{\mathrm{eff}}(k)$.}
We give the formulation for the case of GW motives, the case of the W motive is similar.
The main result of the article is 
\begin{theorem}[Theorem \ref{th:fullyfaithDMW(k)}]\label{th:intro:fullyfaithDMW(k)}
For an infinite perfect field $k$, $\chark k \neq 2$, the canonical functor
$$\Sigma^\infty_{\bGm}\colon   \DMGWeff(k)  \rightarrow   \DMGW(k),\quad
 A^\bullet\mapsto    (A^\bullet,A^\bullet\otimes \bGmw,\dots, A^\bullet\otimes {\bbG_m^{\wedge i}},\dots)
$$
is a fully faithful embedding.
\end{theorem}
Combining the result of the article with the results of \cite{AD_DMGWeff}
we get the sequence of adjunctions
\begin{equation}\label{eq:intr:ajsq}
\begin{array}{ccccccc}
L_{GW}&\dashv& F_{GW},\quad L_{\affl}&\dashv& R_{\affl},\quad \Sigma^\infty_{\bGm}&\dashv& \Omega^\infty_{\bGm},\\
\Dun(\ShN(k))&{\leftrightharpoons}& \Dun(\ShNGW)& \leftrightharpoons& \DMGWeffunb(k)&\leftrightharpoons& \DMGW(k)
,\\
\bZ(X)&\mapsto & \bZ_{GW,Nis}(X)&\mapsto & M^{GW}_{eff}(X)&\mapsto & M^{GW}(X)
\end{array}
\footnote{
We consider unbounded derived categories that allows to define the infinite loop functor $\Omega_{\bGm}^{\infty}$, but further in the text we'll deal with bounded above derived categories, since the result after stabilisation by $\bGmw$ in both cases is the same,
the GW-motives of smooth varieties and homotopy invariant presheaves with GW-transfers in formula \eqref{eq:intr:MTC} are bounded above. 
}
\end{equation}
where $\bZ_{GW,Nis}(X)=\nisGWCor(-,X)$ is the Nisnevich sheafification,
and $F_{GW}$ is forgetful. 
$L_{\affl}$ is the localisation with respect to $\affl$-equivalences.
The adjoint functor $R_{\affl}$ is equivalent to the full embedding of the subcategory of motivic complexes, 
which are complexes with homotopy invariant cohomology sheaves.
Under the last identification $L_{\affl}$ takes a complex $A^\bullet$ to $\mathcal Hom_{\Dun(\ShNGW)}(\Delta^\bullet,A^\bullet)$.
Finally, the third adjunction is infinite-suspension and infinite-loop functors.
%
The second adjunction in the sequence is a reflection, $R_{\affl}(L_{\affl}(A^\bullet)\simeq A^\bullet$.
Theorem \ref{th:intro:fullyfaithDMW(k)} implies that the third adjunction is a coreflection, $A\simeq \Omega^\infty_{\bGmw}(\Sigma^\infty_{\bGmw}(A))$.

Similar to in the original case of $Cor$-cor\-res\-pon\-den\-ces and other known cases
Theorem \ref{th:intro:fullyfaithDMW(k)} is a consequence of the following cancellation theorem.
\begin{theorem}[Corollary \ref{cor:CancSmMotCompl}]\label{th:intr:Cancellth}
For any $A^\bullet,B^\bullet\in \DMGWeff(k)$, 
there is isomorphism
on hom-groups
$$ Hom_{\DMGWeff(k)}(A^\bullet, B^\bullet)\simeq Hom_{\DMGWeff(k)}(A^\bullet(1), B^\bullet(1)) 
,$$
where $A^\bullet(1) = A^\bullet\otimes \bGmw$, and similarly for $\DMWeff(k)$.
\end{theorem}
Equivalently this states that the adjunction $-\otimes \bGmw\colon \DMGWeff(k) \leftrightharpoons \DMGWeff(k)\colon \mathcal Hom({\bGmw},-)$ is a coreflection.
The main case is the case of $A^\bullet = M^{GW}(X)$, $B^\bullet = M^{GW}(Y)$,
which is equivalent to the isomorphism 
\begin{equation}\label{eq:CancMotCompl}
\nisGWCor(X\times\Delta^\bullet, Y) \simeq\nisGWCor(X\times\bGmw\times\Delta^\bullet, Y\times\bGmw) 
\end{equation}
because of the equality \eqref{eq:intr:MTC} for effective GW-motives 


\subsection{Comparing with other cancellation theorems.}
Let's give short list of other known cancellation theorems proved 
for another categories of correspondences:
\begin{itemize}
\item[1)] Cancellation for 
  $Cor$-cor\-res\-pon\-den\-ces proved by V.~Voevodsky in \cite{Voevodsky_CancellationTh}.
\item[2)] Cancellation for $K_0$-cor\-res\-pon\-den\-ces proved by Suslin in \cite{Suslin-GraysonSpectralSeq}
\item[3)] Cancellation for framed-motives proved by A.~Ananievsky, G.~Garkusha, I.~Panin in \cite{AnGarPanin_CancellationThFrMot}.
\item[4)] Cancellation for Milnor-Witt cor\-res\-pon\-den\-ces in recent work \cite{FaselOstvaer_CancThMilnorWittCor} by J.~Fasel, P.~\O stvaer.
\end{itemize}

\newcommand{\WCor}{WCor}

\label{intro_sketchofproof}
In this subsection we give a brief overview of ingredients and steps of the proof of our cancellation theorem, 
and explain what is the similar and what differs proofs of other cancellation theorems. 
Following to the original scheme in \cite{Voevodsky_CancellationTh}, 
the main two ingredients for theorem \ref{th:intr:Cancellth} are 
\begin{itemize}[leftmargin = 15pt]
\item[1)] the construction of a 'partially defined' maps 
$GWCor(X\times\bGmw,Y\times\bGmw) \dasharrow GWCor(X,Y)$, 
which are inverse up to a canonical $\affl$-homotopy to the homomorphism
$id_{\bGmw}\boxtimes - \colon \GWCor(X,Y)\to \GWCor(X\times\bGmw,Y\times\bGmw)$,
(and similarly for $\WCor$);
\item[2)] the computation of the $\affl$-homotopy class $[T]$ of the permutation morphism $T$ on $\mathbb G_m^{\wedge 2}$ in the category of correspondences. 
\end{itemize}

In the case of $Cor$ and $K_0$, $[T]=[-id_{\mathbb G_m^{\wedge 2}}]$
like as in $\mathcal{SH}^+(k)$.
In the case of $\GWCor$,
$[T]=[-\langle -1\rangle]$ (Prop. \ref{prop:permutGmGm}),
like as for framed correspondences and $\mathcal{SH}(k)$.
For $\WCor$, $[T]=[id_{\mathbb G_m^{\wedge 2}}]$
like as in $\mathcal{SH}^-(k)$.

The construction of the homomorphisms $\rho$ strongly depends on the definition of considered correspondences. 
Let's note that the construction of the homotopy for $T$ of course depends in this too, but $\rho$ can not be obtained directly even if we have a functor between the categories of correspondences. 
Informally, $\rho$ can be thought about as $\cup$-product with the class ofthe  diagonal $\Delta$ in $\bGm\times\bGm$,
or it can be considered by analogy with the trace for linear operators.
\begin{itemize}[leftmargin=10pt]
\item[(1)]
For $Z\in Cor(\GXqGY)$,
$\rho($ is defined by the intersection with a generic representer in the Chow-class of the diagonal $\Delta_\bGm$,
which is transversal with respect to $Z$. 
More precisely, $\rho_n(Z)=(Z\cap Z_{red}(f^+_n))- (Z\cap Z_{red}(f^-_n))$, 
where $f^+_n=t^n-1,f^-_n=t^n-u\in k[\bGm\times\bGm]=k(t,u)$. 
Then $\rho_n$ are defined on an exhausting filtration on the group $Cor(\GXqGY)$. 
\item[(2)] 
For $\Phi=[P]\in {K_0}(X\times\bGm,Y\times\bGm)$, $P\in Coh(Y\times X)$, 
$\rho_n([P]) = [P/f^+_n P] - [P/f^-_n P]\in K_0(X,Y)$.
\item[(3)]
If $\Phi=[(Z,\mathcal V,\phi,g)]\in Fr_n(X\times\bGm,Y\times\bGm)$, 
where $\mathcal V\to \aff^n_X$ is etale neighbourhood of closed subscheme $Z$, 
$\phi=(\phi_1\dots \phi_{dim Y})$, 
$\phi_i\in k[\mathcal V]$, $Z=Z(\phi)$ and $g\colon Z\to Y$, 
then $\rho([\Phi])=[(Z,\mathcal V,\phi,\gamma^*(f^+_n),g)] - [(Z,\mathcal V,\phi,\gamma^*(f^0_n),g)]$ in $Fr_{n+1}(X,Y)$, where $\gamma$ is composition map $\mathcal V\xrightarrow{(pr_{X\times\bGm},g)}X\times\bGm\times Y\times\bGm\to \bGm\times\bGm$.
\item[(4)]
In the case when cor\-res\-pon\-den\-ces are defined by a ring cohomology theory, $\rho$ is defined by cohomological multiplication with $pr^*(\Delta_*(1))$, 
where  $X\times\bGm\times Y\times\bGm\xrightarrow{pr} \bGm\times\bGm\xleftarrow{\Delta} \bGm$.
Such construction is used for Milnor-Witt-correspondences.
As well this can be used for the categories given by $(X,Y)\mapsto {GW}^{dim\,Y}_{fin}(Y\times X,\omega_{\bGm\times Y})$ or$W^{dim\,Y}_{fin}(Y\times X,\omega_{\bGm\times Y})$ 
(where $GW^i$ denotes hermitian K-theory and subscript $fin$ means that we consider cohomology groups on $Y\times X$ with finite supports over $X$). 
\end{itemize}

Since $\GWCor_k$ and $\WCor_k$ are
defined by quadratic spaces,
the closest case in a technical sense is $K_0$-corr.
If $\Phi=[(P,q)]$,
then $\rho(P,q) = (P^+,q^+)]-[(P^-,q^-)]$, where $P^+ = P/f^+_n P$, $P^-=P/f^-_n P$.
So the question is to equip sheaves $P/fP$ with quadratic forms
In the case of affine schemes
$P$ is $k[Y\times X]$-module finitely generated projective over $k[X]$,
and $q\colon P\simeq Hom_{k[\bGX]}(P,k[\bGX])$ is $k[\GXGY]$-linear isomorphism.
So we should construct $k[Y\times X]$-linear isomorphism $q^f\colon P/fP\simeq Hom_{k[X]}(P/fP,k[X])$ 
for such $P$ that $P/fP$ is finitely generated projective over $k[X]$, where $f=f^-_n$ or $f^-_n$.

Firstly we choose some additional data, namely, 
$\mathcal N \in k[\aff^1_X] = k[X][t]$ and $g\in k[\GXGY]$
such that the vanish locus $Z(\mathcal N)\in \bGX$ is finite over $X$,
and contains the image of $\Supp P/fP$ under the projection on $\bGX$ ,
and such that for any $a\in P$, we have $\mathcal N \cdot a = f g \cdot a$. 
Actually, we can put $\mathcal N=det_P(f)$,
Then the construction is in four steps: 
\par 1)
Put $(\pri P,\pri q) = (P,q)\otimes_{k[\bGX]} k[\bGX]/(\mathcal N)$.
\par 2)
Define 
a linear homomorphism $l^{\mathcal N}\colon k[Z(\mathcal N)]\to k[X]$ as the junior term of the Euler trace 
$k[X][t]\to k[X][\mathcal N]$, 
where $k[X][\mathcal N]$ is identified with the subalgebra i $k[X][t]$ generated by $\mathcal N$.
\par 3)
Consider  a quadratic space $(\pri P,\ppri q)$ over $k[X]$, where 
 $\ppri q(a,b) = l^{\mathcal N}(\pri q(a,b))$. 
\par 4)
Now we consider the quadratic form $\pppri q(a,b) = \ppri q(g\cdot a,b)$.
Then $\pppri q$ is not non-degenerate, and its kernel is exactly $fP/\mathcal NP \subset P/\mathcal NP$;
hence it defines a quadratic form $q^\rho$ on $P/fP$ which is the result of the construction. 
I.e. the resulting quadratic space is $(P/fP,q^\rho)$.

By a formula this means
$\rho_f(P,q) = red( g\cdot ( (P,q)\circ \langle\mathcal N\rangle ) )$
(see Def. \ref{def:rho_triple}), where 
$\langle \mathcal N \rangle = (k[Z(\mathcal N)], q^{\mathcal N})\in \GWCor(X,\bGX)$, $q^{\mathcal N}(a,b)=l^{\mathcal N}(ab)$,
$g\cdot (-)$ denotes operation of multiplication of quadratic form by the function $g$,
and $red(-)$ denotes the reduction of degenerate quadratic from to the factor-space.

Since $\rho$ is well defined only if $P/f^{+/-}P$ is fin.gen. projective over $k[X]$ 
precisely we define the set of homomorphisms of presheaves
$
GWCor(X\times\bGmw,Y\times\bGmw) \xleftarrow{j_\alpha} R_\alpha \xrightarrow{\rho^{GW,R}_\alpha} GWCor(X,Y) ,
$
$\rho^{GW,R}_\alpha(id_{\bGmw}\boxtimes -) \stackrel{\affl}{\sim} i_\alpha(-)
$
where $R_\alpha$ is a filtering systems $\varinjlim_{\alpha\in \mathfrak{A}} R_\alpha = GWCor(\GwXqGwY)$,
see Lm \ref{lm:leftinverse} for the case of 'left inverse'.
In distinct to the cases of other cor\-res\-pon\-den\-ces $j_\alpha$ aren't injections and 
$\mathfrak{A}$ is not $\bZ$, since it relates to the choice of the triple $(f,\mathcal N,g)$ instead of a one function $f$. 
In sections \ref{sect:Norm}, \ref{sect:CupProdQf}, \ref{sect:invHom} we show that this construction is correct and satisfies enough functoriality.

The construction above gives the homomorphism $\rho$ for presheaves.
But to prove cancellation in the form of isomorphism \eqref{eq:CancMotCompl}
following the mentioned scheme
we need to construct $\rho$ for sheaves $\GWCor_{nis}(-,Y)$.
In distinct to the cases of $Cor$ and framed correspondences the presheaves $GWCor_k(-,Y)$ aren't sheaves .
That is why
we prove firstly the cancellation theorem for presheaves, i.e. the quasi-isomorphism 
$
GWCor(X\times\Delta^\bullet, Y) \simeq
GWCor(X\times\bGmw\times\Delta^\bullet, Y\times\bGmw).
$
This yields that the adjunction
$- \otimes \bGm\colon  D^-_{\affl}(\PreGW) \leftrightharpoons D^-_{\affl}(\PreGW) \colon \mathcal Hom(\bGm,-) $
is a coreflection,
where $D^-_{\affl}(PreShWGtr)$ is the localisation of $\Dba(\PreGW)$ with respect to $\affl$-equivalences.
Next we show that this pair of functor remains to be a coreflection after the Nisnevich localisation, which is the claim according the paragraph before \eqref{eq:CancMotCompl}.
The presheave $K_0(-,Y)$ are not sheaves too, and we think that our reasoning actually is equivalent,
to the reasoning in \cite{Suslin-GraysonSpectralSeq}.
But the approach used here is more functorial and it allows us
to avoid references to internal arguments inside the proofs of some Voevodsky's lemmas about presheaves with transfers
(used in \cite{Suslin-GraysonSpectralSeq}), and it allows to avoid some technical details in this part of the text.

\subsection{The text review:}
In the section \ref{sect:QCor}
we recall the definition, and prove some properties of 
$QCor$, 
$\GWCor$, $\WCor$, and give the construction that produce a quadratic space from a regular function on a relative affine line. 
In section \ref{sect:DMGWeff} we recall some facts about $\DMGWeff(k)$ and $\DMWeff(k)$.

In sections \ref{sect:Norm}, \ref{sect:CupProdQf} and \ref{sect:invHom}
we prove some technical results on 
endomorphisms of locally free coherent sheaves of finite rank,
and construct the homomorphisms $\rho^{GW,L}_\alpha$, $\rho^{GW,R}_\alpha$,$\rho^{W,L}_\alpha$, $\rho^{W,R}_\alpha$.

Finally, in section \ref{sect:CancellationTh} we prove Cancellation Theorem \ref{th:intro:fullyfaithDMW(k)}, Theorem \ref{th:intro:fullyfaithDMW(k)}, and isomorphism \eqref{eq:intr:MTC}.

\subsection{Acknowledgements:}
Acknowledgement to I.~Panin who encouraged me to work on this project, for helpful discussions.

\subsection{Notation and conventions:}
By default we assume schemes being of finite type and separated over a field $k$, $char\,k\neq 2$;
$Sm_k$ is the category of smooth schemes over $k$; 
$Coh(X)=Coh_X$ denotes the category of coherent sheaves on a scheme $X$.
Denote by $\Gamma_f$ a graph of a regular map $f\colon X\to Y$ and $f^{-1}(Z)=Z\times_Y X$.
$k[X] = \Gamma(X,\mathcal O(X))$, 
and for any $f\in k[X]$,
$Z(f)=f^{-1}(0)$, 
and $Z_{red}(f)$ is its reduced subscheme.
For $P\in Coh(X)$ we denote by $\Supp P$ the closed subscheme in $X$ defined by the sheaf of ideals $\mathcal I(U)=Ann\,P\big|_{U}\subset k[U]$ and by $\Supp_{red} P$ its reduced subscheme.
Finally we denote by ${\mathcal F}_{nis}$ the Nisnevich sheafification of a presheave $\mathcal F$ on $Sm_k$.

\section{Quadratic correspondences}\label{sect:QCor}

\subsection{Categories $QCor$, $\GWCor$, $\WCor$.}
In this subsection we summarise definitions and used properties of GW-correspondences and Witt-correspondences, see \cite{AD_StrHomInv}
, \cite{AD_DMGWeff} for more details and proofs.

\begin{definition}\label{def:catCohfcalP}
For a morphism of schemes $p\colon Y\to X$
let $Coh_{fin}(p)$ (or $Coh_{fin}(Y\to X)$) 
denotes
the full subcategory of the category of coherent sheaves on $Y$ 
  spanned by sheaves $\cal F$ such that $\Supp \cal F$ is finite over $X$;
and let  
$\mathcal P(p)$ (or 
$\mathcal P^Y_X$) 
denotes the full subcategory of $Coh_{fin}(Y)$ 
  spanned by sheaves $\cal F$ such that $p_*(\cal F)$ is a locally free sheaf on $X$. 
%
For two schemes 
$X$ and $Y$ over a base scheme $S$
we denote $
Coh_{fin}^S(X,Y) = Coh_{fin}(X\times_S Y\to X)$, $\mathcal P^S(X,Y) = \mathcal P(X\times_S Y\to X).$
\end{definition}

\begin{remark}
For affine schemes $Y$, $X$,
$\mathcal P(Y\to X)$ is equivalent to the full subcategory in the $k[Y]-Mod$ spanned by finitely generated and projective modules over $k[X]$.
\end{remark}

The functors $k[Y]\cmod\to {k[Y]\cmod}^{op}\colon M\mapsto Hom_{k[X]}(M,k[X])$ for finite morphisms $Y\to X$ of affine schemes,
defines in a canonical way a functor $D_X\colon Coh_{fin}(Y_X)\to Coh_{fin}(Y_X)^{op}$ 
for any morphism of schemes $Y\to X$.
This gives an exact category with duality $(\mathcal P(Y\to X),D_X)$.
For any schemes $X$, $Y$ and $Z$ over a base scheme $S$ 
the tensor product over $Y$
induce a functor of categories with duality
\begin{equation}\label{eq:compQ}-\circ -\colon  (\mathcal P^S(Y,Z),D_Y) \times (\mathcal P^S(X,Y),D_X) \to (\mathcal P^S(X,Z),D_X),\end{equation}
where $(\mathcal P^S(X,Y),D_X) \times (\mathcal P^S(Y,Z),D_Y)=(\mathcal P^S(X,Y)\times \mathcal P^S(Y,Z),D_X \times D_Y)$.
This functor is natural in $X$,$Y$ and $Z$ and satisfies the associativity axiom.

\begin{definition}
Let $(\mathcal C, D)$ be an exact category with duality, then\begin{itemize}[leftmargin=15pt]
\item[1)] a pair $(P,q)$, where $P\in \mathcal C$ and $q\colon P\to D(P)$ is symmetric morphism,
is called a quadratic pre-space,
and let's denote by $preQ(\mathcal C,D)$ the set of isomorphism classes of quadratic pre-spaces;
\item[2)] a quadratic space is a pair $(P,q)$ whenever $q\colon P\to D(P)$ is isomorphism,
and we denote by $Q(\mathcal C,D)$ the set of isomorphism classes of quadratic spaces;
\item[3)] $GW^{\oplus}(\mathcal C,D)$ denotes the Grothendieck-Witt-group of $Q(\mathcal C,D)$ in respect to the direct sums, 
\item[4)] $W(\mathcal C,D)$ denotes the Witt-group of $(\mathcal C,D)$, i.e. factor-group of $GW(\mathcal C,D)$ by classes of metabolic spaces (see \cite{Bal_DerWitt}).
\end{itemize}
Some times we write $Q(\mathcal P(X,Y))$ for $Q(\mathcal P(X,Y),D_X)$ since we always consider the duality $D_X$.
\end{definition}


\begin{definition}\label{def:QGWWCor}
Let $k$ be a field.
The \emph{category} $QCor_k$ is the category such that 
objects of $QCor_k$ are smooth varieties and
morphisms $QCor(X,Y)=Q(\mathcal P(X,Y),D_X)$.
The composition is induced by the functor \eqref{eq:compQ},
and the identity morphism $Id_X$ is defined by the class of quadratic space $(\mathcal O(\Delta),1)$ 
    where $i\colon \Delta\to X\times X$. To shortify notations we write sometimes $Q(X,Y)$ for $QCor(X,Y)$.
The \emph{categories} $\GWCor_k$ and $\WCor_k$ are the additive categories 
with the same objects and such that $\GWCor(X,Y) = GW^{\oplus}(\mathcal P^Y_X,D_X)$, $\WCor(X,Y) = W(\mathcal P^Y_X,D_X)$,
the composition is induced by the one of $QCor$.
A \emph{presheave} with GW-transfers is an additive presheave $F\colon GWCor_k\to Ab$
and similarly for presheaves with Witt-transfers.     
\end{definition}

\begin{notation}\label{not:compQ}
For $X,Y,Z\in Sm_k$ and a pair of quadratic spaces $(P_1,q_1)\in QCor(X,Y)$, $(P_2,q_2)\in QCor(Y,Z)$ we call $(P_2,q_2)\circ (P_1,q_1)$ 
the composition of the quadratic spaces and denote it also sometimes by $(P_2,q_2)\circ (P_1,q_1)$ or by $(P_2 \otimes_Y P_1 ,q_2\otimes_Y q_1)$.
\end{notation}

\begin{definition}
We write $\Phi_1\stackrel{\affl}{\sim}\Phi_2$ for $\Phi_1,\Phi_2 \in \GWCor(X,Y)$
whenever
there is 
$\Theta\in \GWCor(X\times \affl,Y)$ such that  
$\Theta\circ i_0=\Phi_1, \Theta\circ i_1=\Phi_1$, where $i_0,i_1\colon X\to X\times\affl$ denotes zero and unit sections. 
In such a situation we call $\Theta$ by an $\affl$-homotopy joining  $\Phi_1$ and $\Phi_2$. 
Denote by $\overline{\GWCor}_k$ the factor-category of $\GWCor_k$ such that morphisms are classes up to homotopy equivalence.
\end{definition}

\begin{theorem}[see \cite{ChepInjLocHIiWtrPreSh}, theorem 1, and \cite{AD_DMGWeff}]\label{th:HGWCorinj} 
Suppose is a homotopy invariant presheave with GW-transfers (Witt-transfers). Then 
(a) for a local essentially smooth $U$ over a base filed $k$,
the restriction homomorphism 
$\mathcal F(U)\to \mathcal F(\eta)$
is injective, where $\eta\in U$ is generic point.

\end{theorem}
\begin{theorem}[see \cite{AD_DMGWeff}]\label{th:HGWCorinj} 
For $F$ as above and
for any open subschemes $V_1\subset V_2\subset \affl_K$, $k=k(X)$, $X\in Sm_k$,
the restriction homomorphism $\mathcal F(V_2)\to \mathcal F(V_1)$ is injective.
\end{theorem}
\begin{theorem}[see \cite{AD_DMGWeff},  theorem 3.1 and corollary 4.13;
\cite{AD_WtrSh}, theorem 3; and \cite{AD_StrHomInv}]\label{th:NshGWtr} 
For a presheave with GW-transfers (Witt-transfers) $\mathcal F$ over a field $k$,
the Nisnevich sheaf $\mathcal F_{nis}$ and Nisnevich cohomology presheaves $h_{nis}({\mathcal F}_{nis})$ are equipped with the structure of presheaves with GW-transfers (Witt-transfers) in a canonical way.

If $\mathcal F$ is homotopy invariant, then $H_{nis}^i(U)=0$, for any open $U\subset \affl_K$, $K=k(X)$, $X
\in Sm_k$.
If in addition $k$ be infinite, perfect, $char k\neq 2$;
then the associated Nisnevich sheaf ${\mathcal F}_{nis}$ and Nisnevich cohomology presheaves $h_{nis}({\mathcal F}_{nis})$ are homotopy invariant.
\end{theorem}

For any pair of quadratic spaces $(P_1,q_1)\in Q(\mathcal P^{Y_1}_{X_1})$, $(P_2,q_2)\in Q(\mathcal P^{Y_2}_{X_2})$,
the isomorphism $$q_1\otimes_k q_2\colon  P_1\otimes_k P_2\simeq D_{X_1}(P_1)\otimes_k D_{X_2}(P_2)\simeq D_{X_1\times X_2}(P_1\otimes_k P_2)$$
defines a quadratic space $(P_1\otimes_k P_2,q_1\otimes_k q_2)\in Q(\mathcal P^{Y_1\times Y_2}_{X_1\times X_2})$
and if at least one of spaces $(P_1,q_1)$, $(P_2,q_2)$ is metabolic then $(P_1\otimes_k P_2,q_1\otimes_k q_2)$ is metabolic.
So we have the following:
\begin{lemma}\label{lm:boxprodunct}
The tensor product over a base field $k$ induce
a functors
$- \boxtimes -\colon QCor_k\times QCor_k \to QCor_k$, and
$- \boxtimes -\colon \GWCor_k \times \GWCor_k \to \GWCor_k$
(and similarly for $\WCor$),
such that $X\boxtimes Y=X\times Y$.

\end{lemma}
\begin{remark}
We use symbol $\boxtimes$ here, to distinct products on the categories of correspondences 
from the tensor product on the category of presheaves with transfers for which we use the symbol $\otimes$.
\end{remark}
\begin{remark}
Since $\GWCor_k$ is additive category, 
all mentioned operations, functors and presheaves with GW-transfers can be passed to the  Karoubi envelope $Kar(GWCor_k)$. 
Similarly for $\WCor_k$.
\end{remark}

\subsection{Construction of quadratic correspondences by a function on an oriented curve}
\begin{lemma}\label{lm:CurveFinS->F}
Let $X\in Sm_k$ and $f\in k[\bbGX]$ such that $Z(f)$ is finite over $X$ and $Z(f)\neq \emptyset$,
then there is a unique regular map $\overline{f}\colon \prl\times X\to \prl$
such that $f$ is the restriction of $\overline{f}$ to $\bbGX$. 
Moreover for such a map $\overline{f}$ 
the morphism $(\overline{f},pr_X)\colon \prl\times X\to \prl\times X$ is finite and flat
and the scheme theoretical preimage $\overline{f}^{-1}(0)$ is equal to the vanish locus $Z(f)$.
\end{lemma}
\begin{proof}
If the map $\overline f$ exists, then $\Gamma_{\overline{f}} = \overline{\Gamma_f}\subset \prl\times X\times \prl$ .
Hence if $\overline{f}$ exists then it is unique.

Let $l_\infty$, $l_0$ be the minimal integers such that there is $r\in \Gamma(\prl\times X,\calO(l_\infty+l_0))\colon f=r/(t_\infty^{l_\infty} t_0^{l_0})$. \emph{We are going to prove that $Z(f) = Z(r)\subset \prl\times X$.}
Since $Z(f)$ is finite over $X$, it is closed in $\prl_X$. 
On the other hand $Z(f)=Z(r)\cap\bbGX$, so it is open in $Z(r)$. 
Then $Z(f)$ is disjoint component of $Z(r)$, and
let $Z(r)=Z(f)\amalg Z_\infty\amalg Z_0$, where $Z_\infty\subset \infty\times X$ and $Z_0\subset 0\times X$. 
If $Z_\infty\neq\emptyset $, then $dim\,Z_{\infty} \geq dim\,(\bbGX)-1=dim\,(\infty\times X)$;
hence $Z_\infty = \infty\times X$ 
and $r\big|_{\infty\times X}=0$.
This contradicts to the minimality of the choice of $l_\infty$.
So $Z_\infty=\emptyset$, and similarly $Z_0=\emptyset$.
Then $Z(f) = Z(r)\cap \bbGX=Z(r)$.

Define $\overline{f}=[r\colon t_\infty^{l\infty}t_0^{l_0}]\colon \prl\times X\to \prl$.
Then from the above we have $Z(f)= \overline{f}^{-1}(0)$. 
Since $dim\,Z(f)\geq dim\,\prl_X-1=dim\,X$ and since $Z(f)$ is finite over $X$, it follows that projection $Z(f)\to X$ is surjective.
Hence the fibre of $Z(f)$ over each point $x\in X$ is not empty and it is not equal to $\prl_x$. Hence
 the fibre of the map $(\overline{f},pr_X)\colon \prl\times X\to \prl\times X$ over each point $s\in \prl\times X$ is not empty and it is not equal to $\prl_s$,
and so $(\overline{f},pr_X)$ is quasi-finite.
Thus the morphism $(\overline{f},pr_X)$ is quasi-finite and projective, and whence it is finite. 
Finally, since it is finite morphism of equidimensional smooth varieties, it is flat. 
\end{proof}

\begin{definition}\label{def:FP}
For any $U\in Sm_k$ and open subscheme $V\subset \affl\times U$,
we denote by $FP(U,V)$ the set of regular maps $\overline{f}\colon \prl\times U\to \prl$
such that $\overline{f}^{-1}(0)$ is finite over $U$ and $\overline{f}^{-1}(0)\subset V$.
Let $FP(U,V,Z)\subset FP(U,V)$, $f\in FR(U,V,Z)$ iff $\overline{f}^{-1}(0)=Z\amalg Z^\prime$.
\end{definition}
\begin{proposition}\label{prop:constrFP->Q}
There are \begin{itemize}[leftmargin=15pt]
\item[1)] a natural (along base changes) map 
$ FP(U,V) \to Q(\mathcal P(U,V))$
such that any function $\overline{f}\in FP(U,V)$ goes to some quadratic space $\langle \overline{f}\rangle=(\mathcal O(f^{-1}(0)),u)$,
\item[2)] a natural map
$ FP(U,V,Z) \to Q(\mathcal P(U,Z))$
such that any function $\overline{f}\in FP(U,V,Z)$ goes to some quadratic space $\langle \overline{f},Z\rangle=(\mathcal O(Z),u)$, where $u\in k[Z]^\times$.
\end{itemize}\end{proposition}
\begin{proof}
1)
Let $F=(\overline{f},pr_U)$ for $\overline{f} \in FP(U)$.
Let's denote the source of the morphism $F$ as $Y$ and the target by $X$.
Then $X$ and $X$ are isomorphic to two copies of the relative projective lines over $X$ and morphism $F$ is finite,
The Grothendieck duality for a finite morphism of smooth projective varieties leads to the natural isomorphism 
$\omega_X\simeq \mathcal Hom_Y(f_*(\calO(X)),\calO(Y))$,
and functoriality of this isomorphism in respect to endomorphisms of $\calO(X)$ implies that it 
defines isomorphism $\omega_X\simeq D_Y(\mathcal O(X))$ that can be considered as
symmetric quadratic space $(\calO(X),\overline{q})\in Q(\mathcal P(f),D_Y^{f^*(\omega_Y)\otimes\omega_X^{-1}})$, where $\mathcal L=f^*(\omega_Y)\otimes\omega_X^{-1}$, and $D_Y^{\mathcal L}(\mathcal F) = D_Y(\mathcal F)\otimes \mathcal L$. 
Then put
$(\calO(Z),q_f^\prime) = i^*(\calO(X),\overline{q})\in Q(\calP(f),D_X^{j^*(f^*(\omega_Y)\otimes\omega_X^{-1})})$, 
where $j\colon Z\hookrightarrow \affl_U$, $i\colon 0_U \hookrightarrow \prl_U$. 
Now since the standard trivialisation of $\omega(\affl_U)$ given by differential of coordinate 
gives an identifications $i^*(\omega_Y)\simeq \calO(U)$ and $j^*(\omega_X)\simeq \calO(Z)$, 
we get required quadratic space $(\calO(Z),q_f^\prime)\in Q(\calP(f),D_X)$.

2)The claim follows form that 
if $\overline{f}^{-1}(0)=Z\amalg Z^\prime$ then any quadratic space
$(\mathcal O(\overline{f}^{-1}(0)),q)$ splits in a canonical way into the sum 
$(\mathcal O(Z),u)\oplus (\mathcal O(Z^\prime),u^\prime)$. 
\end{proof}
\begin{definition}\label{def:fQclsp}
Let $U\in Sm_k$, $V\subset \affl$ be open subscheme 
and $f\colon V\times U\to \affl$ be regular function. 
Then  
we denote by $\langle f,V\times U\rangle \in Q(\mathcal P(U,V\times U))$ the image under the map from Proposition \ref{prop:constrFP->Q}
of the finite morphism $\overline f\colon \prl\times U\to \prl$ that is a lift of $f$ if it exists
(definition is correct since if such lift exists then it is unique). 
\end{definition}

\begin{lemma}\label{lm:constQuad}
Assume $char\,k\neq 2$, then for any finite scheme $Z$ over a variety $X$ and quadratic space $(\mathcal O(Z\times\affl),q)\in Q(\mathcal Coh_{fin}(Z\times\affl\to X\times\affl))$,
there is an isomorphism of quadratic spaces $(\mathcal O(Z),q_0)\simeq(\mathcal O(Z),q_1)$,
where $q_0 = i_0^*(q)$, $q_1=i_1^*(q)$ and $i_0,i_1\colon X\to X\times\affl$ denote zero and unit sections.
\end{lemma}
\begin{lemma}
For any regular function $q$ on a scheme $Z$ (over the base filed $k$, $\chark k\neq 2$) such that $q\big|_{Z_{red}} =1$ there is a square root $t\in \mathcal O(Z)$: $t^2=q\in \mathcal O(Z)$.
\end{lemma}
\begin{proof}[Proof of Lemma \ref{lm:constQuad}]

Any quadratic form on $\mathcal O(Z\times \affl)$ is defined by invertible regular function $q$ on $Z\times\affl$.
In the case of reduced $Z$ any such function is constant along $\affl$, and hence $q_0=i^*_0(q) = i^*_1(q)=q_1$, that implies the claim.

In general case we have that restriction $q\big|_{Z_{red}\times\affl}$ is constant (along $\affl$).
Hence $i^*_0(q)$ and $i^*_1(q)$ are equal on $Z_{red}$.
To prove the claim let's note that any function on $Z$ that is equal to $1$ on $Z_{red}$ has a square root and 
applying this to fraction $i^*_0(q)/i^*_1(q)$
we get isomorphism of quadratic spaces $i^*_0(\mathcal O(Z\times\affl),q)$ and $ i^*_0(\mathcal O(Z\times\affl)$.
\end{proof}

\subsection{Locally splitting homomorphisms} 
\begin{lemma}\label{lm:diagDualMorphism}
For any endomorphism of coherent shaves $e\in End_{Coh_{fin}(X,Y)}(P)$ for any schemes $X$,$Y$ and $P\in Coh_{fin}(X,Y)$,
there is a unique commutative diagram 
\begin{equation}\label{eq:diagDualMorphisms}\xymatrix{
\Ker D(e)\ar@{^(->}[r] \ar[d]&
D_X(P) \ar@/^/[rr]^{D_X(e)} \ar@{->>}[r] \ar@{=}[d]&
\Image D(e) \ar@{^(->}[r] \ar[d]^{w}&
D_X(P) \ar@{->>}[r] \ar@{=}[d]&
\Coker D(e) \ar[d]
\\
D(\Coker e)\ar@{^(->}[r]&
D_X(P)\ar@/_/[rr]_{D_X(e)} \ar@{->>}[r] &
D_X(\Image(e)) \ar@{^(->}[r] &
D_X(P) \ar@{->>}[r]&
D(\Ker e)
.}\end{equation}
\end{lemma}\begin{proof}
The vertical arrows in the diagram are defined by universal properties of kernels and cokernels.  
\end{proof}

\begin{definition}\label{def:locspl}
A morphism of coherent sheaves $e\colon P_1\to P_2$ on a scheme $X$ is called {\em locally splitting}
whenever both short exact sequences in the diagram 
$$\xymatrix{
ker e \ar@{^(->}[r] & P_1 \ar@/^/[rr]^e \ar@{->>}[r] & Im e \ar@{^(->}[r] & P_2 \ar@{->>}[r] & Coker e
}$$
splits locally on $X$.
A morphism $e\colon P_1\to P_2\in \mathcal Con_{fin}(X,Y)$ for a varieties $X,Y$ is called locally splitting
whenever its direct image along the projection $X\times Y\to X$ is locally splitting on $X$. 
\end{definition}

\begin{lemma}\label{lm:exDlocsp}
1)
For any varieties $X,Y$ the subcategory of locally splitting morphisms is an abelian subcategory of $Coh_{fin}(X,Y)$,
and the functor $D_X$ is exact on the subcategory of locally splitting morphisms in $Coh_{fin}(X,Y)$.

2)
For any locally splitting morphism the vertical arrows in the commutative diagram \eqref{eq:diagDualMorphisms} are isomorphisms.

3)
Any short exact sequence in $\mathcal P(X,Y)$ is locally splitting.
\end{lemma}
\begin{proof}
1)
Since for any locally splitting homomorphism $e$, the homomorphisms $\ker e$ and $\coker e$ are locally splitting
the first statement follows. 

The second claim follows from that base change of the duality $D_X$ along the open immersion $U\hookrightarrow X$ (or along the embedding of local subscheme $X_x\to X$)
is equal to $D_U$ ($D_{X_x}$), and for any scheme $U$ the functor $D_U$ sends splitting sequences to splitting ones. 

2) The claim follows from the previous point, since by definition exact functor preserves kernels, images, and cokernels.

3) The last statement follows form that short exact sequence of locally free coherent sheaves is locally splitting.
\end{proof}

\subsection{Homotopy for permutation on $\bGmw\times\bGmw$}
\begin{lemma}
\label{th:Gmwsqinj}
For a homotopy invariant presheave with GW-transfers $\mathcal F$ the homomorphism
$F(\bGmw\times\bGmw)\to F(\bGmw\times\bGmw\Delta_{\bGmw})$ is injective.
\end{lemma}
\begin{proof}
The prove is based on the same method on the prove of the injectivity on the relative affine line.
Indeed, the method proves the injectivity for a pair of open subsets $U\subset V$ in the relative affine line $\affl_X$ such that 
$U=\affl-(T\amalg D)$, $V=\affl-T$, where $T$ is quasi-finite and $D$ is finite over $X$.
\end{proof}

\begin{definition}
Let $pr^{\bGm}_{pt}\colon\bGm\to pt$, and let $1\colon pt\to \bGm$ denotes the unit section.
Let 
$\bGmw = \Coker(\bGm\to pt)\in Kar(GWCor_k)$ 
and let 
$1_\bGm = 1\circ pr^{\bGm}_{pt}$ 
Finally, let
$pr^\wedge \colon 
\GWCor(\bGm,\bGm)\to \GWCor(\bGmw,\bGmw)$
denote the canonical projection to a direct summand,
and let
$e = 
(id_\bGm-1_\bGm)\circ - \circ (id_\bGm-1_\bGm)\colon 
\GWCor(\bGm,\bGm)\to \GWCor(\bGm,\bGm)$
be the corresponding idempotent.
 The same notation we use for $\WCor$.
\end{definition}

\begin{lemma}\label{prop:permutGmGm}
Let $T$ denote 
the transposition on
$\bGmw\times \bGmw $,
then
$$[T] = [\langle-1\rangle \cdot id_{\bGmw\times \bGmw}]\in \overline{\GWCor}_k(\bGmw\times \bGmw,\bGmw\times \bGmw).$$
\end{lemma}
\begin{proof}
Let $i_Q\colon Q = Z(t^2 -a t +m)\hookrightarrow \affl_{\affl\times\bGm}$, where we use 
$k[\affl\times \bGm]=k[a,m,m^{-1}]$.
Define
$$
\Pi = \langle t^2 -a t +m , Q\rangle\in \GWCor(\affl\times\bGm, Q),\;
C\colon \bGm\times\bGm\to \affl\times\bGm\colon C(x,y)=(x+y,xy)
.$$
Then 
$\Pi\circ C = \langle (t-x)(t-y) , Q\times_{,\bGm\times\affl,C} (\bGm\times\bGm) \rangle \in GWCor(\bGm\times\bGm, Q)$.
The map $\bGm\times\bGm\to \affl\times\affl\times\bGm\colon (x,y)\mapsto (x,x+y,xy)$)
induce isomorphism $\bGm\times \bGm \simeq Q$.
Hence 
\begin{multline}\label{eq:Split}
\Pi\circ C\circ i = 
\langle x-y\rangle\langle t-x, Z(t-x) \rangle + \langle y-x\rangle \langle t-y, Z(t-y) \rangle = \\
\langle x-y\rangle  (id_{\bGm\times\bGm} + \langle-1\rangle \circ T)\in 
GWCor( \bGm\times\bGm-\Delta_{\bGm} , Q ) ,\end{multline}
where 
$i\colon\bGm\times\bGm - \Delta_{\bGm}\hookrightarrow \bGm\times\bGm$.

On the other hand
$$\Pi\circ sect \circ \mu \circ i = \langle 1-xy\rangle [(x,y)\mapsto (1,xy)] + \langle xy-1\rangle [(x,y)\mapsto (xy,1)] ,$$ 
where
$\mu\colon \bGm\times\bGm\to \bGm\colon (x,y)\mapsto xy$ and $sect\colon \bGm\to \affl\times\bGm\colon (a,m)\mapsto (m+1,m)$;
and whence
$e\circ \Pi\circ sect \circ pr_{\bGm} \circ  C\circ i = 0\in GWCor_k(\bGm\times\bGm,\bGmw\times\bGmw).$ 

Now using homotopy $h\colon \affl\times\bGm\times\affl\to \affl\times\bGm\colon (a,m,\lambda)\mapsto (a(1-\lambda)+(1+m)\lambda, m),$ and since $h\circ i_0=id_{\affl\times\bGm}$, $h\circ i_1=sect \circ pr_{\bGm}$, we get
\begin{multline*}
pr^{\wedge} (id_{\bGm\times\bGm} + \langle-1\rangle \circ T) = 
\langle (x-y)^{-1} \rangle\circ [\Pi\circ C\circ i] = \\ = \langle (x-y)^{-1} \rangle\circ [\Pi\circ sect \circ \mu \circ  C\circ i]=0
\in \overline{GWCor_k}(\bGm\times\bGm-\Delta_{\bGm},\bGmw\times\bGmw)
.\end{multline*}
Now the claim follows form that homomorphism $\overline{GWCor_k}(\bGmw\times\bGmw, \bGmw\times\bGmw)\to \overline{GWCor_k}(\bGmw\times\bGmw-\Delta_{\bGmw}, \bGmw\times\bGmw)$ is injective.
(see lemma \ref{th:Gmwsqinj}). 
\end{proof}

\begin{remark}
In the previous proof in the equation
we use base change along closed embeddings for the duality isomorphism for finite morphisms $f$
and we assume that choice in the duality theorem is made in a way that duality isomorphism $\xi_f$ for the identity morphism of schemes $f\colon X\to X$, is equal to the isomorphism defined by canonical isomorphisms $\mathcal O(X)\simeq f^*(\mathcal O(X)$ and $\omega(X)\simeq f^*(\omega(X))$.

Let's note also that for this computation it is enough to use flat base changes, and actually base changes along isomorphisms.
To do this we need to apply base change for the square defined by two copies of morphism $Q\to \affl\times \bGm$, identity isomorphism on $\affl\times\bGm$ and isomorphism $T$ on $Q$ (using isomorphism $Q\simeq \bGm\times\bGm$). Then we get that $\langle (t-x_1)(t-x_2) \rangle_{Z(t-x_1)}=\langle \lambda \rangle$ and $\langle (t-x_1)(t-x_2) \rangle_{Z(t-x_2)}=\langle -\lambda \rangle$ for some invertible function $\lambda\in k[\bGm^2]^\times$. So the required equality $[id_{\bGm^2\Delta}]=[\langle -1\rangle id_{\bGm^2\Delta}]\in \overline{GWCor}(\bG_m^2-\Delta,\bG_m^2-\Delta)$ follows after the multiplication of quadratic forms in \eqref{eq:Split} and in other equalities form the proof above by $\lambda^{-1}$. 
\end{remark}

\begin{corollary}\label{cor:permut}
1)
$[T] = [id_{\bG_m^{\wedge 2}}]  $ 
in $\overline{WCor_k}$,
2)
$[C_3] = [id_{\bG_m^{\wedge 3}}] $
in $\overline{GWCor_k}$,
where $C_3$ denotes the 3-cycle permutation. 
\end{corollary}
\begin{proof}
1) The claim follows form that $[\langle 1\rangle] + [\langle -1 \rangle] = 0 \in WCor_k(pt,pt)$, since it is class of a metabolic space.
2) The claim follows from that a 3-cycle is a composition of two transpositions and $[\langle -1 \rangle]^2 = id_{pt} \in GWCor_k(pt,pt)$.
\end{proof}

\section{Categories of GW-motives and Witt-motives}\label{sect:DMGWeff}

Starting form this section we assume that the base field $k$ is infinite, perfect, and $char\,k\neq 2$.
Here we recall the results of \cite{AD_DMGWeff} on the categories of effective GW-motives and Witt-motives used in the article. 
We consider the case of GW-motives,
and the case of Witt-motives is similar.

Let $Pre$ and $\ShN$ denote categories of presheaves and Nisnevich sheaves on $Sm_k$,
$\PreGW$ (or $\PreGWs$) the category of presheaves with GW-transfers, $\ShNGW$ (or $\ShNGWs$) the category of Nisnevich sheaves with GW-transfers.
It is proven in \cite{AD_DMGWeff} that$\ShNGW$ is abelian.


\begin{definition}
The \emph{category of effective GW-motives} $\DMGWeff(k)$
is the full subcategory in $\Dba(\ShNGW)$
spanned by the complexes with homotopy invariant cohomology sheaves.
The functor $M^{GW}\colon Sm_k\to \DMGWeff(k)$, which sends $X\to\Sm_k$ to its \emph{effective GW-motive}, is defined as
$M^{GW}_{eff}(X) = \mathcal Hom_{\Dba(\PreGWs)}(\Delta^\bullet, \nisGWCor(-,X) )=\nisGWCor(-\times \Delta^\bullet,X),$
where $\Delta^i$ denotes affine simplexes. 

\end{definition}
\begin{theorem}\label{th:DMWeff}
Suppose $k$ is an infinite perfect filed, $char\,k\neq 2$.

 There is an adjunction $L_{\affl}^{pre}\colon \Dba(\PreGWs_k) \leftrightharpoons \DaffGW(k) \colon R_{\affl}$
such that 
$R_{\affl}^{pre}$ is equivalent to the full embedding functor of the subcategory spanned by complexes with homotopy invariant cohomology presheaves,
$L_{\affl}^{pre}$ is equal to the localisation with respect to morphisms of the form $\mathbb Z_{GW}(X)\times \mathbb Z_{GW}(\affl\to X)$, for all $X\in Sm_k$,
and $L_{\affl}\circ R_{\affl}(\mathbb Z_{GW}(X)) = \GWCor(\Delta^\bullet\times -, X)$.

Moreover the functors $L_{\affl}^{pre}$ and $R_{\affl}^{pre}$ are exact with respect to the Nisnevich quasi isomorphisms.
The localisation of $\DaffGW(k)$ with respect to Nisnevich quasi isomorphisms is equivalent to $\DMGWeff(k)$.
The adjunction $L_{\affl}^{pre}\dashv R_{\affl}$ induces the adjunction $L_{\affl}\colon \Dba(\ShNGWs_k) \leftrightharpoons \DMGWeff(k) \colon R_{\affl}$ such that $L_{\affl}\circ R_{\affl}(\mathbb Z_{GW,nis}(X)) = M^{GW}(X)$.


\end{theorem}

\begin{corollary}\label{cor:MPMeff}
For an infinite perfect field $k$, $\chark k \neq 2$, $X\in Sm_k$ and 
a motivic complex $A^\bullet\in \DMGWeff(k)$ there is a natural isomorphism
$Hom_{\DMGWeff(k)}(M^{GW}(X),\Sigma^\infty_{\bGmw}A^\bullet[i]) \simeq H^i_{Nis}(X,A^\bullet),$
\end{corollary}


The following lemma is used in section \ref{sect:CancellationTh} to deduce the sheaf cancellation theorem form the presheaf one.

\begin{lemma}\label{lm:coh((subAff)_loc)}
Let $U$ is local essential smooth $k$-scheme, and $V\subset \affl_k$ is open over an infinite perfect base filed $k$, $char\,k\neq 2$.
Let $\mathcal F\in Pre^{GW}_k$ (or $Pre^W$) be homotopy invariant.
Then
$${\mathcal F}_{nis}(U\times V) \simeq \mathcal F(U\times V), H^i_{nis}(U\times V, {\mathcal F}_{nis})\simeq 0, \text{ for }i>0.$$
\end{lemma}
\begin{proof}
Consider the case of $\GWCor$.
Let $\eta\in U$ denote the generic point.
By Th. \ref{th:NshGWtr} $U\mapsto H^i_{nis}(U\times V, {\mathcal F}_{nis})$
are homotopy invariant presheaves with GW-transfers. 
Hence Th. \ref{th:HGWCorinj} yields the injection
$H^i_{nis}(U\times V,{\mathcal F}_{nis})\to H^i_{nis}(\eta\times V,{\mathcal F}_{nis})$ for all $i$,
and by Th. \ref{th:NshGWtr}
${\mathcal F}_{nis}(\eta\times V)\simeq \mathcal F(\eta\times V)$,
$H^i_{nis}(\eta\times V,{\mathcal F}_{nis}) = 0$, for $i>0$. 

Thus since
the injection $\mathcal F(U\times V)\hookrightarrow \mathcal F(\eta\times V)$
factors as
$\mathcal F(U\times V)\to {\mathcal F}_{nis}(U\times V)\hookrightarrow {\mathcal F}_{nis}(\eta\times V)\simeq \mathcal F(\eta\times V)$,
it follows that $\nu\colon \mathcal F(U\times V)\hookrightarrow {\mathcal F}_{nis}(U\times V)$ is injective. 
On other side, applying Th \ref{th:HGWCorinj} to the hom. inv. presheaves $\Coker(\nu),\Coker(\nu)(-\times V)\in Pre^{GW}$
$\Coker(\nu)(U\times V)=\Coker(\nu)(\eta\times V)\Coker(\nu)(\eta^\prime) =0$, where $\eta^\prime$ is generic point of $\eta\times V$.
Thus $\nu$ is an isomorphism.
\end{proof}


As shown in \cite{AD_DMGWeff} there is a tensor structure on the category $\DMGWeff(k)$ such that $M^{GW}(X)\otimes M^{GW}(Y)\simeq M^{GW}(X\times Y)$. 
(Note that this tensor structure 
doesn't relates to the tensor structure $\boxtimes$ on the category of correspondences defined in Lemma \ref{lm:boxprodunct}.)
Let $\mathbb G_m^{\wedge k} = (\bGmw)^{\otimes k}\in \DMGWeff(k)$, where $\bGmw = Cone(pt\stackrel{1}{\hookrightarrow} \bGm)$,
and for any $A^\bullet\in \DMGWeff(k)$, we denote $A^\bullet(n) = \bGm^{\wedge k}\otimes A^\bullet$.

\begin{definition}
Since $\DMGWeff(k)$ is the full subcategory in $\Dba(\ShNGWs_k)$, it is dg-category and it is equipped with the injective model structure.  
Define the category of (non-effective) motives $\DMGW(k)$ 
as the category of $\bGmw$-spectra with respect to $\DMGWeff(k)$. 
So objects of $\DMGWeff(k)$ are 
sequences
$E = (E_0^\bullet,s_0,E_1^\bullet,s_1,\dots,E_n^\bullet,s_n,\dots,),$ 
$E_i^\bullet\in \DMGWeff(k),$
$s_i\in Hom_{\DMGWeff(k)}(E_i^\bullet(1),E_{i+1}^\bullet)$ 
and for two such sequences $E= (E_i)$ and $F = (F_i)$ the homomorphism group is defined as
$Hom_{\DMGW(k)}(E,F) = \varinjlim_i Hom_{\DMGWeff(k)}(E_i,F_i)$.

A functor 
$\Sigma^\infty_{\bGm}\colon \DMGWeff(k) \to \DMGW(k)$ takes a motivic complex
$E^\bullet$  to the spectrum  $(E^\bullet,E^\bullet(1),\dots E^\bullet(n)\dots )$ with $s_i$ being identity morphisms $id_{E^\bullet(i)}$.
\end{definition}


\section{Norm of the multiplication endomorphism $m^P_f$. }
\label{sect:Norm}

\begin{definition}\label{def:NormP}
Let $p\colon V\to S$ a morphism of schemes, $P\in \calP(V\to S)$,
and $f\in k[V]$.
Then the multiplication by $f$ defined the \emph{endomorphism} $m^{p_*(P)}_f\in End(p_*(P))$,
and denote
$det_P(f)=\bigwedge\limits^{r}m^{p_*(P)}_f\in End(\bigwedge\limits^{r} p_*(P))=k[S]$, where $r = \rank_{\mathcal O(S)}\, p_*(P)$.

\end{definition}

%

\begin{lemma}\label{lm:redSupN_P}
For any morphism of schemes $p\colon V\to S$, $P\in \mathcal P(V\to S)$, $P\neq 0$ and $f\in k[V]$, 
we have $Z_{red}(\mathcal N_P(f)) = p(\Suppred P/fP)$, where $P/fP=\Coker(m^P_f)$.
\end{lemma}
\begin{proof}
One can easily check that for any point $z\in S$, $\mathcal N_P(f)=det_{P_z}(f)$, where $P_z$ is a fibre of $P$ over $z$.
So $z\not\in Z_{red}(\mathcal N_P(f)) \Leftrightarrow \text{ the fibre }(m^f_P)_z\text{ is invertible }\Leftrightarrow P/fP=0$.\end{proof}

\begin{definition}\label{def:univ-inj}
For any $p\colon V\to S$, 
a homomorphism $h\colon P\to Q\in Coh(V)$ is called {\em universally injective} over $S$ 
if  $\forall F\in Coh(S)$ the homomorphism $h\otimes p^*(F)$ is injective.
\end{definition}
\begin{lemma}\label{lm:univ-inj}
Suppose $V\xrightarrow{\alpha} U\xrightarrow{p} S$ are morphisms of schemes, $\alpha$ is affine,
and $h\colon P\to Q\in Coh(V)$; 
then $h$ is universally injective over $S$, iff $\alpha_*(h)$ is universally injective on $S$.

2) Suppose $p\colon V\to S$ is affine and flat, and $h\colon P\to Q\in Coh(V)$; 
then the following conditions are equivalent:
(1) $h$ is universally injective over $S$,
(2) $p_*(Coker(h))=Coker(p_*(h))\in Coh(S)$ is flat,
(3) the fibre $h_x=i_x^*(h)$ is injective for any $x\in S$, where $i_x\colon x\to S$ is the injection.

\end{lemma}\begin{proof}
1) The claim follows form that the direct image functor with respect to an affine morphism is exact.
2) (1) $\Leftrightarrow$ (2):
By the first point of the lemma 
(1) is equivalent that $p_*(h)$ is universally injective. 
Now we use the reasoning used by Suslin in the case of $K_0$-corr.
Since $p$ is flat, and $P$  and $Q $ are  flat objects in $Coh(V)$, 
we see that $p_*(P)$ and $p_*(Q)$ are flat on $S$ too. 
Since $p_*(h)$ is universally injective, computing $Tor_i(p_*(Coker\, h))$ using the flat resolvent given by $0\hookrightarrow p_*(P)\rightarrow
p_*(Q)\twoheadrightarrow p_*(Coker\, h)$, we see that $Tor_i(p_*(Coker\,h,G))=0$ $\forall G\in Coh(S),$ $i>0$. 
The equivalence (2) $\Leftrightarrow$ (3) follows form that
$F\in Coh(S)$ is flat iff $Tor_i(F,k(x))=0$ for all points $x\in S$.
\end{proof}

\begin{lemma}\label{lm:FinGenSup}
Let $X$, $Y$ be schemes over the base field $k$,
let $P\in \mathcal P(\bbG_m\times X,\bbG_m\times Y)$,
$f\in k[\bbG_m\times X\times \bbG_m \times Y]$,
and $pr_X\colon \GXGY\to X$ denote the canonical projection;
then ${pr_X}_*(P/fP)$ is coherent on $X$ if and only if $\Supp P/fP$ is finite over $X$.
\end{lemma}
\begin{proof}
The claim follows from the general fact that direct image of a coherent $\mathcal F$ sheaf along a affine morphism $p\colon T\to S$ is coherent iff $\Supp \mathcal F$ is finite over $S$. 
\end{proof}

\begin{lemma}\label{lm:Nfg}
Let $P\in \mathcal P(\bbG_m\times X,\bbG_m\times Y)$, 
$f\in k[\bbG_m\times X\times \bbG_m \times Y]$,
$\calN=det_P(f)\in k[\bbGX]$. 
(1) Then $m_{\calN,P} = m_{f,P}\circ m_{g,P}$ for some $g\in k[\bbG_m\times X\times \bbG_m \times Y]$, (1') and if $f\in k[\bGm\times\bGm]$, there is a such a function $g$ in $k[X\times \bGm\times Y\times\bGm ]$.

If the endomorphism $m^P_f\in End(P)$ is injective, then
(2) 
$m^P_{\mathcal N}$, $m^P_g$ are injective,
(2') 
and moreover $\Ker m_g^{\pri P} = f\pri P$, $g\pri P\simeq \Coim m_g^{\pri P} \simeq \Image m_g^{\pri P} \simeq \pri P/f\pri P $, $\Coker m_g^{\pri P} \simeq \pri P/g\pri P$,
where 
$\pri P=P/\calN P$.

If $\Supp P/fP$ is finite over $X$, then
(3)
$Z(\mathcal N)$ is finite over $X$, 
(3') $m^P_f$, $m^P_{\mathcal N}$, and $m^P_{g}$ are injective over each point of $X$;
(3'')
the short exact sequences 
$f\pri P \hookrightarrow \pri P \twoheadrightarrow g\pri P$, $\pri P/f\pri P\hookrightarrow \pri P\twoheadrightarrow \pri P/g\pri P$
(given by point (2'))
splits locally (see Def. \ref{def:locspl}), 
and 
$f\pri P, g\pri P, \pri P/f\pri P, \pri P/g\pri P $ are locally free of a finite rank.
%
\end{lemma}
\begin{proof}
1) Consider $p(\lambda)=det_P(\lambda - m_f)$, $p(\lambda)=\lambda^{\rank P} + a_1 \lambda^{\rank P-1} +\dots + a_{\rank P} $, $a_i\in k[\bG_m\times X]$, where $m_f=m_f^P$, 
Since $det_P(f-m_f)=0$, it follows that 
$p(\lambda)= (\lambda -  f)\tilde g(\lambda)$ for some $g(\lambda)\in k[\bG_m\times X\times\bG_m\times Y][\lambda]$. 
Hence $\mathcal N = p(0)= f g$, where $g=\tilde g(0)$.
Point (1') 
follows from that if $f\in k[\bG_m^2]$, then $a_i,f\in k[X\times \bGm\times\bGm]$, and
$g$ is contained in the subalgebra  in $k[\bG_m\times X\times\bG_m\times Y]$ generated $a_i$ and $f$.

2) Suppose $m^P_f$ is injective.
Since $\mathcal N=fg$,  $m^P_g$, and $m^P_f$ are injective.
Since $\Supp P$ is finite over $\bGm$, 
for any $h\in k[X\times\bGm\times \bbGY]$ 
 $m^{P}_h$ is injective iff
 $pr_{X \times\bGm}(m^{P}_h)$  injective.
By the definition of $\calP(X\times\bGm,\bbGY)$ 
the sheaf $pr_{x\times\bGm}(m^{P_x}_h)$ is coherent and locally free. 
Hence
$pr_{X\times\bGm}(m^{P}_h)$ is injective iff
the homomorphism $\nu^*(pr_{x\times\bGm}(m^{P_x}_h))$ induced over generic point $\nu\hookrightarrow X\times\bGm$ is injective.
Then $\nu^*(pr_{x\times\bGm}(m^{P_x}_h))$ is injective iff $det(\nu^*(pr_{x\times\bGm}(m^{P_x}_h)))\in k(\nu)$ is invertible,
since $\dim_{k(\nu)} \nu^*(pr_{x\times\bGm}(m^{P_x}_h))<\infty$. 
Now since $det(\nu^*(pr_{x\times\bGm}(m^{P_x}_h)))=\calN(\nu)$, we see that $m^P_{\mathcal N}$ is injective. 

2')
Since $m_g$ is injective by (2), and since $m_g(fP) = \mathcal N P\subset P$,
it follows that
$m_g^{-1}(\calN P) = f P$, and hence $Ker(\pri m_g)=f\pri P$.
Other qualities are by the definition.

4)
The surjection $\Suppred (P/fP)\twoheadrightarrow Z(\calN)_{red}$ implies that $Z(\calN)_{red}$ is finite over $X$.
Let $l_\infty$, $l_0$ be minimal integers such that 
$\exists r\in \Gamma(\prl\times X,\calO(l_\infty+l_0))$, $\calN=r/(t_\infty^{l_\infty} t_0^{l_0})$.

We are going to show that $Z(r)\cap \{0,\infty\}\times X=\emptyset$.
Since $Z(\mathcal N)_{red}$ is finite over $X$, $Z(\mathcal N)_{red}$ is closed in $Z(r)_{red}$; since $Z(\calN)_{red}$ is equal to $Z(r)_{red}\cap\bbGX$, $Z(\mathcal N)_{red}$ is open in $Z(r)_{red}$. 
Hence $Z(r)_{red}=Z(\mathcal N)_{red}\amalg Z_\infty\amalg Z_0$, $Z_\infty\subset \infty\times X$, $Z_0\subset 0\times X$. 
Suppose that $Z_\infty\neq\emptyset $. Then $dim\,Z_{\infty} = dim\,(\bbGX)-1=dim\,(\infty\times X)$, and $Z_\infty = \infty\times X$. So $r\big|_{\infty\times X}=0$. But this contradicts to the minimality of $l_\infty$. Thus $Z_\infty=\emptyset$, and similarly $Z_0=\emptyset$.

So we get $Z(r) = Z(r)\cap \bbGX=Z(\calN)$.
Hence $Z(\calN)$ is projective over $X$.  Since as was shown $Z(\calN)_{red}$ is finite over $X$, it follows that $Z(\calN)$ is quasi-finite over $X$. Whence $Z(\calN)$ is finite over $X$.

3') Let $x\in X$.
By the same reasoning as in (2') applied to the fibre $x^*(P)$ over a point $x\in X$ we see that
if $(m_f^P)_x$ is injective, then $(m^P_{\mathcal N})_x$, $(m^P_g)_x$ are injective. And moreover,
the $(m_f^P)_x$ is injective iff $m^{\nu^*(P_x)}_f$ is invertible,
where $\nu$ is generic point of $ x\times\bGm$.
If $\Supp (P/fP)$ is finite over $X$, then $\mu^*(P/fP)=0$, so $\coker(m^{\nu^*(P_x)}_f)=0$, and hence $m^{\nu^*(P_x)}_f$ is invertible, since $\nu^*(P_x)$ is finite dimensional vector space over $k(\nu)$. 

3'')
By the point (3) $Z(\mathcal N)$ is finite over $X$.
Hence ${pr_X}_*(P/\mathcal NP)\in coh(X)$ and consequently ${pr_X}_*(P/f P), {pr_X}_*(P/g P)\in coh(X)$.  
Now by (3') $m^P_f$, $m^P_{\mathcal N}$ and $m^P_g$ are injective over each point in $X$.
By Lemma \ref{lm:univ-inj} this implies that ${pr_X}_*(P/fP)$ , ${pr_X}_*(P/\mathcal NP)$ and ${pr_X}_*(P/gP)$ are flat over $X$.
The claim follows, 
since any finitely generated flat coherent sheaf is locally free, 
and any short exact sequence with locally free cokernels locally splits.
\end{proof}

\section{$\cup$-product of quadratic spaces with a function.}\label{sect:CupProdQf} {
We construct some operation $\rho$ which
takes $(P,q)\in Q(X\times\bGm,Y\times\bGm)$
and a suitable $f\in k[\bGm\times\bGm]$
to some $(P/fP, q^\prime)\in Q(X,Y)$ (see Def. \ref{def:rho_triple}).
Informally it can be considered as a $\cup$-product of a quadratic space with a regular function.
The function $f$ is applicable to $(P,q)$ 
only if
$Z(f)$ is finite over $Y\times\bGm$ and transversal to $\Supp P$, but it isn't enough.
Moreover, by the construction $\rho$ depends on a triple of functions $(f,\mathcal N,g)$, where 
$\mathcal N\in k[X\times\bGm]$, $g\in k[X\times\bGm\times Y\times\bGm]$ such that $\mathcal N = f g$, $Z(\mathcal N)$ is finite over $X$ (see Def. \ref{def:ftripl}).
A set of 'f'-triples that we can use in our construction with a given quadratic space $(P,q)$ we call as $(P,q)$-applicable 'f'-triples.

So to apply our construction 
we should firstly prove the existence of 
an applicable triple (Lemma \ref{lm:appltripl-fNg_min}) and a partly ndependence form the choice of $\cal N$ and $g$
Also 
we show that the operation is well defined on Grothendieck-Witt and Witt groups. 

In this and the next section we assume that $char\,k\neq 2$, where $k$ is the base field.

\begin{definition}\label{def:Qe}
Let $X$ be a scheme, 
$P\in coh(X)$, 
$q$ a symmetric quadratic form $q\colon P\to D_X(P)$,
and 
$e\in End(P)$, $q\circ e = D(e)\circ q\in Hom(P,D_X(P))$;
then 
denote by $e\cdot (P,q)$ the pair $(P,q^\prime)$,
where
$q^\prime = q\circ e\in Hom(P,D_X(P))$, which symmetric quadratic form, since $e$ is self-adjoint.  
\end{definition}

\begin{definition}\label{def:Qred}
For any $(P,q)\in preQ(coh_{fin}(X,Y))$,
we denote by $red(P,q)$ (or $(P,q)_{red}$) 
the pair $(Im q, q_{red}) \in preQ(\mathcal Coh_{fin}(X,Y))$, 
where $q_{red}\colon \Im q\to D(\Im q)$ is a unique homomorphism such that 
$q= D(\image q)\circ q_{red}\circ \image q$ 
%
defined by universal property of the kernels and cokernels.
\end{definition}

\begin{lemma}\label{lm:eQred}
Let $(P,q)\in Q(\mathcal P(X,Y))$ and $e\colon P\to P$ be a locally splitting (see Def. \ref{def:locspl}) self-adjoint endomorphism;
then $red(e\cdot (P,q))\in Q(\mathcal P(X,Y))$.
\end{lemma}
\begin{proof}

By definition $(e\cdot q)_{red}$ is equal to composition $Im e\xrightarrow{u} Im D_X(e) \xrightarrow{w} D_X(Im e)$,
where 
$u$ is defined by universal property of the image.
So there is the commutative diagram
$$\xymatrix{
\Ker e \ar@{^(->}[r] & 
P\ar@/^/[rr]^e \ar[d]_{q} \ar@{->>}[r] 
& \Image(e)  \ar[d]^u \ar@/^2pc/[dd]^<<<<<<<{q_{red}} \ar@{^(->}[r] 
& P \ar[d]_{q}^{\simeq} \ar@{->>}[r]&
\Coker e\\
\Ker D_X(e)\ar@{^(->}[r] \ar[d]&
D_X(P) \ar@{->>}[r] \ar@{=}[d]&
\Image D_X(e) \ar@{^(->}[r] \ar[d]^{w}&
D_X(P) \ar@{->>}[r] \ar@{=}[d]&
\Coker D_X(e) \ar[d]
\\
D_X(\Coker e)\ar@{^(->}[r]&
D_X(P)\ar@/_/[rr]_{D(e)} \ar@{->>}[r] &
D_X(\Image(e)) \ar@{^(->}[r] &
D_X(P) \ar@{->>}[r]&
D_X(\Ker e)
.
}$$
Since $(P,q)\in Q(\mathcal P(X,Y))$, we see that $q$ is isomorphism and hence $u$ is isomorphism. 
Since $e$ is locally splitting, if follows from Lm \ref{lm:exDlocsp} that $w$ is isomorphism.
The claim follows.
%
%
%
%
\end{proof}

\begin{remark}\label{rem:appltripl} 
For any $P$ as above 
$\Supp P/fP$ is finite over $X$ iff ${pr_{X}}_*(P/fP)\in coh(X)$.
\end{remark}

\begin{definition}\label{def:ftripl} 
For a pair of varieties $X, Y$
we call by an \emph{'f'-tripe} a 
triple $(f,\mathcal N,g)$, $f,g\in k[\GXGY]$, $\calN\in k[\bGX]$ such that $Z(\mathcal N)$ is finite over $X$.

An 'f'-tripe $(f,\mathcal N,g)$,
is called {\em applicable} with respect to $(P,q)$ or just a $(P,q)$-triple
iff $\Supp  P/fP$ is finite over $X$
and $m_{\mathcal N,P} = m_{f,P} \circ m_{g,P}\in End(P)$.

A $(P,q)$-triple $(f,\mathcal N,g)$ is called \emph{normal} $(P,q)$-triple, 
whenever there are regular functions $\mathcal N_{ad}$ on $\bGX$ and $g_{min}$ on $\GXGY$,
such that
$\mathcal N = \mathcal N_P(f) \mathcal N_{ad}$ and $g = g_{min} \mathcal N_{ad}$.
\end{definition}

\begin{definition}[\bf  The map $\rho$]\label{def:rho_triple}
For an $(P,q)$-triple $(f,\mathcal N,g)$ (see Definition \ref{def:ftripl})
put
$$\rho_{(f,\mathcal N,g)}(P,q) = red( m_{g,P}\cdot ( (P,q) \circ \langle \mathcal N , \bGX \rangle ) ) \in preQ(\mathcal P(X,Y))$$
(see Def. \ref{def:fQclsp}, \ref{def:Qe}, \ref{def:Qred}) 
\end{definition}

\begin{lemma}\label{lm:rho_triple}
For a $(P,q)$-triple $(f,\mathcal N,g)$,
if $(P,q)\in Q(\mathcal P(\GXqGY))$ 
then $\rho_{(f,\mathcal N,g)}(P,q)\in Q(\mathcal P(X,Y))$
\end{lemma}
\begin{proof}
The claim follows from the point (3'') of Lm \ref{lm:Nfg} and Lem \ref{lm:eQred}.
\end{proof}
\begin{lemma}\label{lm:e1e2Q1redQ2}
Let $(P_1,q_1)\in preQ(\mathcal P(X,Y))$, and $(P_2,q_2)\in Q(\mathcal P(Y,Z)$, then 

\noindent 1) 
for any locally splitting  $e\in End(P)$(see Def. \ref{def:locspl}),
and any $(P_2,q_2)\in Q(\mathcal P(Y,Z)$, one has
$$red( e\cdot ((P_2,q_2)\circ (P_1,q_1)) ) = (P_2,q_2)\circ red(e\cdot(P_1,q_1))\in Q(\mathcal P(X,Z))$$

\noindent 2) 
for a commuting self-adjoint endomorphisms $e_1,e_2\in End(P)$, one has
$$(e_2\circ e_1)\cdot (P,q) = e_2\cdot (e_1\cdot (P,q) ),\;
red ( (e_2\circ e_1)\cdot (P,q) ) = red ( e_2\cdot red(e_1\cdot (P,q) ),$$
where $e_2$ at the right side denotes the restriction  $e_2\big|_{\Image e_1}$, which is well defined since $e_1e_2=e_2e_1$.

\end{lemma}
\begin{proof}
The point (1) is equivalent to that $\Ker q_1\otimes P_2 = ( \Ker (q_1\otimes_Y q_2) )$ 
(see Lm \ref{not:compQ} for $q_1\otimes_Y q_2$).
Straightforward verification shows that $\Ker q_1\otimes P_2\subset ( \Ker (q_1\otimes_Y q_2) )$. On other hand by Lm. \ref{lm:eQred} $red(e \cdot (P_1,q_1))\in Q(\mathcal P(X,Y))$ and then by the discussion before Def. \ref{def:QGWWCor} $(P_2,q_2) \circ red(e\cdot (P_1,q_1))\in Q(\mathcal P(X,Z))$; hence $\Ker q_1\otimes P_2= ( \Ker (q_1\otimes_Y q_2) )$. 
%
Point (2) easily follows from the definitions. 
\end{proof}
\begin{lemma}\label{lm:rho_props}

Let $(P,q)\in Q(\mathcal P(\GXqGY))$ 
and $(f,\mathcal N,g)$ be $(P,q)$-triple, then

\noindent 1)
$\mathcal N = \mathcal N_{b}\,\mathcal N_{ad}$ and $g = g_{b}\,\mathcal N_{ad}$, we have
$
\rho_{(f,\mathcal N,g)} = red( g_{b}\cdot( red(\mathcal N_{ad}\cdot\langle\mathcal N,\bGX\rangle) \circ (P,q) ) )
.$

\noindent 2)
$id_\bGm\boxtimes \rho_{(f,\mathcal N,g)}(P,q)\simeq \rho_{(f,\mathcal N,g)}(id_\bGm \boxtimes P)$.

\noindent 3)
For any metabolic $(P,q)\in Q(\mathcal P(\GXqGY)$ and $(P,q)$-triple $(f,\mathcal N,g)$, $\rho_{(f,\mathcal N,g)}(P,q)$ is  metabolic.

\end{lemma}
\begin{proof}
1)
Using Lm \ref{lm:e1e2Q1redQ2} and Def \ref{def:rho_triple} we see
$
\rho_{(f,\mathcal N,g)} = 
red( m_{g,P}\cdot (\langle \mathcal N , \bGX \rangle \circ (P,q) ) ) =
red( m_{g_{b},P}\cdot m_{\mathcal N_{ad},P}\cdot (\langle \mathcal N , \bGX \rangle \circ (P,q) ) ) =
red( m_{g_{b},P}\cdot (m_{\mathcal N_{ad},P}\cdot \langle \mathcal N , \bGX \rangle \circ (P,q) ) ) =
red( g_{b}\cdot( red(\mathcal N_{ad}\cdot\langle\mathcal N,\bGX\rangle) \circ (P,q) ) )
$

2) The claim follows, since operations used in the def. \ref{def:rho_triple} 
commute with $id_\bGm\boxtimes -$.

3) 
We can assume that $X$ is affine.
By Def \ref{def:rho_triple} and Lm \ref{lm:Nfg}.(3'') $fg=\mathcal N$, and $\rho_{(f,\mathcal N,g)}(P,q) = red(\pri P, g \pri q)=(\ppri P, \ppri q)$, and $fg=\mathcal N$, where $(\pri P,\pri q)=(P,q) \circ \langle \mathcal N, \bGX \rangle$ and $\ppri P=\pri P/f \pri P=P/fP$. Moreover Let $L$ is sublagrangian in $( P, q)$, then  $\pri L= P/\mathcal N P$ is sublagrangian subspace of $(\pri P,\pri q)$, 
and $\ppri L=L/fL$ is sublagrangian in $(\ppri P,\ppri q)$.
Since diagram \eqref{eq:diagDualMorphisms} is self-dual, if follows that $D(\ppri P/\ppri L)$ is sublagrangian in $(D(\ppri P),\ppri q^{-1})$, and so $\ppri L$ suplagrangian subspaces in $(\ppri P,\ppri q)$.
\end{proof}

\begin{lemma}\label{lm:rho_etripl} Let $(P,q)\in Q(\mathcal P(\GXqGY))$, and $(f,\pri{ \mathcal N} , g)$ and $(f,\mathcal N,\pri g)$ are
$(P,q)$-triples.

\noindent 1) Then if $\pri{\mathcal N}=\mathcal N$, then $\rho_{(f,\mathcal N,g)}(P,q)\simeq \rho_{(f,\mathcal N,\pri g)}(P,q)$.

\noindent 2)
The following conditions holds:\begin{itemize}[leftmargin=15pt]
\item[-]
$\mathcal N$ and $\mathcal N^\prime$ are of the same degree in $t$ and $t^{-1}$ under identification $k[\bGm]=k(t)$,
and the leading and the last coefficients of $\mathcal N$ and $\mathcal N^\prime$ are equal, 
\item[-]
there is $(P,q)$-triple $(f,\mathcal N_{min},g_{min})$ and $\mathcal N_{ad}, \mathcal N_{ad}^\prime\in k[\bGX]$, 
$\mathcal N=\mathcal N_{min}\,\mathcal N_{ad}$ $\mathcal N^\prime = \mathcal N_{min}\,\mathcal N_{ad}^\prime$.
\end{itemize}
Then $\rho_{(f,\mathcal N,g)}(P,q) \simeq \rho_{(f,\pri {\mathcal N},\pri g)}(P,q).$
\end{lemma}
\begin{proof}
From the definition of the map $\rho$ (Def. \ref{def:rho_triple}) it follows that
$$
\rho_{(f,\mathcal N,g)}(P,q) = red( m_{g,P}\cdot (\langle \mathcal N , \bGX \rangle \circ (P,q) ) ),
\;
\rho_{(f,\mathcal N,\pri g)}(P,q) = red( m_{\pri g,P}\cdot (\langle \mathcal N , \bGX \rangle \circ (P,q) ) )
$$
So it is enough to show that $m_{g,P} = m_{\pri g, P}$.
Since both 'f'-triples are applicable for $(P,q)$, we see that
$m_{g,P} \circ m_{f,P} = m_{\mathcal N,P} = m_{\pri g,P} \circ m_{f,P}$, i.e.
$f g \big|_{\Supp P} = \mathcal N\big|_{\Supp P} = f\pri g\big|_{\Supp P}$.
Thus if we show that multiplication by $f$ on $\mathcal O(\Supp P)$ is injective, then the claim follows.

Since $Z(\mathcal N)$ is finite over $X$, it follows that multiplication by $\mathcal N$ is injective on $\mathcal O(X\times\bGm)$.
Denote by $i\colon \Supp P\to X\times\bGm\times Y\times\bGm$ the canonical injection,
and by $pr\colon   X\times\bGm\times Y\times\bGm\to X\times\bGm$ the projection to the first two multiplicands.
Let $\pri P = i^*(P)$, then by Definition \ref{def:catCohfcalP} the direct image $\ppri P = pr_*(i_*(\pri P))$ is locally free coherent sheaf of finite rank on $X\times \bGm$.
Hence $m_{\mathcal N,\ppri P}$ is injective.
The morphism $pr\circ i$ is finite and consequently it is affine, hence direct image functor $pr_*\circ i_*$ is exact an faithful.
So $m_{\mathcal N,\pri P}$ is injective.
Thus $m_{f,P}$ is injective.
\end{proof}

\begin{proof}
By Lemma \ref{lm:rho_etripl} 
since $m^P_g\circ m^P_f = m^P_{\mathcal N} = m^P_{\mathcal N_{ad}}\circ m^P_{\mathcal N_{min}} = m^P_{\mathcal N_{ad}}\circ m^P_{g_{min}}\circ m^P_f$, 
we can assume that 
$g = g_{min} \mathcal N_{ad}$,
and by the same reason
$g^\prime = g^\prime_{min} \mathcal N^\prime_{ad}$.
Then by Lemma \ref{lm:rho_props}
\begin{gather*}
\rho_{(f,{\mathcal N} ,g)}(P, q) = 
red(  g_{min} \cdot 
  ( red( {\mathcal N}_{ad} \cdot \langle {\mathcal N} , \bGX \rangle ) \circ (P,q) )), \\ 
\rho_{(f,\pri {\mathcal N} ,\pri g)}(P, q) = 
red(  \pri g_{min} \cdot 
  ( red( \pri {\mathcal N}_{ad} \cdot \langle \pri {\mathcal N} , \bGX \rangle ) \circ (P,q) ))  
.\end{gather*}
So to prove the claim it is enough to prove that
$red( {\mathcal N}_{ad} \cdot \langle {\mathcal N} , \bGX \rangle )
\simeq 
red( \pri {\mathcal N}_{ad} \cdot \langle \pri {\mathcal N} , \bGX \rangle )$.

By assumption 
$\mathcal N = a_d t^d +  \dots + a_c t^{c}\in k[X](t)=k[\bGX]$ and $\mathcal N^\prime = a_d t^d +  \dots + a_c t^{c}\in k[X](t)=k[\bGX]$, for some integers $d\geq c$ and regular functions $a_d,a_c\in k[X]$.
Since $Z(\mathcal N)$ is finite over $X$, it follows that $a_d$ and $a_c$ are invertible.
Consider affine homotopy between these quadratic correspondences
defined by the 'f'-triple 
$(f,\tilde{\mathcal N},\tilde g)\colon$,
$\tilde{\mathcal N}_{ad}= \mathcal N_{ad} (1-\lambda) + \mathcal N_{ad}^\prime \lambda$,
$\tilde{\mathcal N}= \tilde{\mathcal N}_{ad} \mathcal N_{min}$,
$\tilde g = \tilde{\mathcal N}_{ad} g_{min}$,
which is the quadratic space 
$red( \tilde {\mathcal N}_{ad} \cdot \langle \tilde{\mathcal N} , \bGX\times\affl \rangle )$.
Then $\tilde{\mathcal N_{ad}}\mathcal O(Z(\tilde{\mathcal N} )) \simeq \mathcal O(\bGX\times\affl)/\mathcal N_{min} = \mathcal O( Z(\mathcal N_{min})\times \affl )$,
and $red( \tilde {\mathcal N}_{ad} \cdot \langle \tilde{\mathcal N} , \bGX\times\affl \rangle ) = (\mathcal O( Z(\mathcal N_{min})\times \affl ), u),$
for some invertible regular function $u$ on $Z(\mathcal N_P(f))\times \affl$.
Thus 
$red( {\mathcal N}_{ad} \cdot \langle {\mathcal N} , \bGX \rangle ) = i_0^*(\mathcal O( Z(\mathcal N_{min})\times \affl ), u)$,
$red( \pri {\mathcal N}_{ad} \cdot \langle \pri {\mathcal N} , \bGX \rangle ) = i_1^*(\mathcal O( Z(\mathcal N_{min})\times \affl ), u),$
where $i_0$, $i_1$ denotes zero and unit sections of $X\times\affl$ as usual.
The claim follow since the quadratic spaces above are isomorphic by Lm \ref{lm:constQuad}.
\end{proof}



\begin{lemma}\label{lm:appltripl-fNg_min}
For any $(P,q)\in Q(\mathcal P(\GXqGY))$ and $f\in k[\GXGY]$ such that $\Supp  P/fP$ is finite over $X$,
there is an  $(P,q)$-triple of the form $(f,\mathcal N_P(f),g)$.
\end{lemma}
\begin{proof}
By Lemma \ref{lm:redSupN_P} 
$Z_{red}(\mathcal N_P(f)) = p(\Suppred P/fP)$, where $p\colon \Supp P\to \bGX$. Hence $Z_{red}(\mathcal N_P(f))$ is finite over $X$. 
Whence closure of of $Z(\mathcal N_P(f))$ in $\prl_X$ is equal to $Z(\mathcal N_P(f))$. 
Then $Z(\mathcal N_P(f))$ is projective over $X$ and hence it is finite.
By Lm \ref{lm:Nfg}.(1) there is $g\in k[\GXGY]$, such that $m_{\mathcal N_P(f),P} = m_{f,P}\circ m_{g,P}$.
Thus $(f,\mathcal N_P(f),g)$ is $(P,q)$-triple.
\end{proof}

\begin{lemma}\label{lm:GWCortoAffl}
%

1)
Suppose $n,m$ are positive integers such that $m>n$,
$X\in Sm_k$, $g,g^\prime\in k[ X\times\affl]=k[X][t]$ 
are regular functions that are monic polynomials of degree $m-n$ 
and such that $Z(g\big|_{X\times\bGm})$ and $Z(g^\prime\big|_{X\times\bGm})$ are finite over $X$,
and suppose $\tau\in k[X\times\affl]$ is a function such that $Z(\tau)$ is subscheme of $X\times\bGm$ that is finite over $X$;
then
$$
pr^{X\times\affl}_X\circ [red(g\cdot \langle g \tau,\bGm\times X\rangle)] = 
pr^{X\times\affl}_X\circ [red(g^\prime\cdot \langle g^\prime \tau, \bGm\times X \rangle)]
\in GWCor(X,X)
,$$
where $pr_X\colon X\times\affl\to X$.

2)
Suppose $n,m\in \mathbb Z$, $m>n>0$;
then there are an invertible element $\beta_{n,m}\in k^*$
and some $h_{n,m}\in GWCor(\affl,pt)$
such that for any $X$ and $g=t^{m-n} +a_{m-n-1} t^{m-n-1} + \dots +a_{1}t+1\in k[\affl\times X]=k[X][t]$, 
we have
$$\begin{array}{l}
(id_X\boxtimes h)\circ i_0 = 
pr^{X\times\bGm}_X\circ ([red(g\cdot \langle g (t^n-1),\bGX\rangle)] - 
[red(g\cdot \langle g (t^n-t), \bGX \rangle)]),\\
(id_X\boxtimes h)\circ i_1 = \langle \beta_{n,m}\rangle\boxtimes id_X \in GWCor_k(X,X),
\end{array}$$
where $i_0,i_1\colon X \to X\times\affl$ denotes the zero and unit sections, and
$pr^X\colon X\to pt$, and $pr^{X\times\bGm}_X\colon X\times\bGm\to X$ denote the canonical projections. 

\end{lemma}
\begin{proof}

1)
Consider a quadratic space
$\tilde Q = red( \tilde g \cdot \langle \tilde g \tau , X\times\affl\rangle)\in QCor(X,X\times\affl),$
where $\tilde g = g(1-\lambda)+g^\prime\lambda\in k[X\times\affl\times\affl]\subset k[X\times\bGm\times\affl].$
Since $\mathcal O(Z(\tilde g\tau))/\tilde g \mathcal O(Z(\tilde g\tau))\simeq \mathcal O(Z(\tau))$,
then $\tilde Q = (\mathcal O(Z(\tau)\times\affl),\tilde q)$ and Lemma \ref{lm:constQuad} yields the claim.

2)
Firstly consider the case of $X=pt$ and $g = (t^{m-n}+1)$.
Define 
$$\tilde h_{n,m} = pr^{\affl}\circ [red( (t^{m-n}+1)\cdot (\langle (t^{m-n}+1) ((t^n-1)(1-\lambda)+(t^n-t)\lambda)] , \affl \rangle) ) \in GWCor(pt,pt). $$ 
Then
since $Z( (t^{m-n}+1) (t^n-1) )\subset \bGm$, we get 
$
\tilde h_{n,m}\circ i_0 = pr^{\affl}\circ [red( (t^{m-n}+1)\cdot \langle (t^{m-n}+1) (t^n-1),\affl\rangle )]=
pr^{\bGm}\circ [\langle (t^{m-n}+1) (t^n-1),\bGm\rangle]
.$ 
On other side,
$Z( (t^{m-n}+1) (t^n-t) )= Z( (t^{m-n}+1) (t^{n-1}-1) )\amalg \{0\}$, and the 
module (sheaf) of the quadratic space $red( (t^{m-n}+1)\cdot \langle (t^{m-n}+1) (t^n-t)\rangle )$ is isomorphic to $k[t]/(t^n-t)\simeq k[t]/(t^{n-1}-1)\oplus k[t]/t$. Then 
$(t^{m-n}+1)\cdot \langle (t^{m-n}+1) (t^n-t),\{0\}\rangle$ is rank one quadratic spaces over $k$.
Hence
\begin{multline*}
\tilde h_{n,m}\circ i_1 = pr^{\affl}\circ [red( (t^{m-n}+1)\cdot \langle (t^{m-n}+1)(t^n-t), \affl \rangle )]=\\
pr^{\bGm}\circ [red( (t^{m-n}+1)\cdot \langle (t^{m-n}+1) (t^n-t),\bGm\rangle )]\oplus
\langle \beta_{n,m}\rangle
,\end{multline*}
for some invertible $\beta_{n,m}\in k^*$.
Thus we can put $$h_{n,m} = \tilde h_{n,m} - pr^{\bGm}\circ [red( (t^{m-n}+1)\cdot \langle (t^{m-n}+1) (t^n-t),\bGm\rangle )]\circ pr^{\bGm\times\affl}_{\bGm} ,$$
where $pr^{\bGm\times\affl}_{\bGm}\colon {\bGm\times\affl}\to{\bGm}$.

Now consider any $X\in Sm_k$ and a function $g$ as in lemma. Form the first point it follows that 
\begin{multline*}
[red(g\cdot \langle g (t^n-1),\bGX\rangle)] - [red(g\cdot \langle g (t^n-t), \bGX \rangle)] 
=\\
[red((t^{m-n}+1)\cdot \langle (t^{m-n}+1) (t^n-1),\bGX\rangle)] - \\- [red((t^{m-n}+1)\cdot \langle (t^{m-n}+1)(t^n-t), \bGX \rangle)]
,\end{multline*}
and the claim follows from that constructions from Definition \ref{def:Qe}, Definition \ref{def:Qred} and Proposition \ref{prop:constrFP->Q} respects base change.

%

\end{proof}

\section{An inverse homomorphism for $-\boxtimes id_{\bGmw}$}\label{sect:invHom} 
\subsection{Filtering system of functions}
In this subsection we define two systems of functions $f^{+/-}_n$ on $\bGm\times\bGm$ indexed by positive integers, define of some special 'f'-triples relating to the functions $f^{+/-}_n$, and prove properties of such 'f'-triples needed in the construction of the  homomorphisms that are left and right inverse for the homomorphism $-\boxtimes id_{\bGmw}$.

\begin{definition}\label{def:n,m-appltripl}
Using identification $k[\bGm\times\bGm]=k(t)(u)$ let's define two regular function $f_n^+ = t^n-1$, $f_n^- = t^n-u$.
Denote by the same symbols inverse images of these functions on $\GXGY$.
A pair of 'f'-triples (Def. \ref{def:ftripl}) $\tau =((f^+_n,\mathcal N^+,g^+),(f^-_n,\mathcal N^-,g^-))$
is called an \textbf{'f'-bi-triple of degree $(n,m)$} or by $(n,m)$-bi-triple  
whenever $\mathcal N^+, \mathcal N^-  \in k[X][t]\subset [\bGX]$ and
$\mathcal N^+ = t^m + a_{m-1}^+ t^{m-1} +\dots+ a_1^+ t     - 1$,
$\mathcal N^-  = t^m + a_{m-1}^-  t^{m-1} +\dots+ a_2^-  t^2 - t$.

For any quadratic space $Q\in QCor(\GXqGY)$,
$(n,m)$-bi-triple $\tau$ is called {\em normal} and {\em applicable} for $Q$ iff such are both of 'f'-triples in $\tau$ (see Def. \ref{def:ftripl}),
and we denote
$\rho^+_\tau(Q)=\rho_{(f^+_n,\mathcal N^+, g^+)}(Q)$,
$\rho^-_\tau(Q)=\rho_{(f^-_n,\mathcal N^-, g^-)}(Q)$.
\end{definition}

\begin{lemma}\label{lm:rho:indbitrip}
Suppose $Q=(P,q)\in QCor(\GXqGY)$
and $\tau_1$, $\tau_2$ are two normal $Q$-applicable $(n,m)$-bi-triples; 
then $\rho_{\tau_1}^{+/-}(P,q) \simeq \rho_{\tau_2}^{+/-}(P,q).$
\end{lemma}\begin{proof}
The claim follows form definitions, Lemma \ref{lm:rho_props} and Lemma \ref{lm:rho_etripl}.
\end{proof}

\begin{definition}\label{def:filtsetA} 
Define ordered set
$\mathfrak{A} = \{(r,n,m)\in \mathbb Z^3| m<rn\}$
with order $(r_2,n_2,m_2)>(r_1,n_1,m_1)$ iff $r_2>r_1$, $n_2>n_1$
and $m_2 - r_2 n_2 > m_2 - r_2 n_2$. 
\end{definition}
\begin{lemma}\label{lm:filteredsyst}
Ordered set $\mathfrak A$ is filtering, i.e. for any $\alpha_1,\alpha_2\in \mathfrak A$ there is $\alpha_3>\alpha_1,\alpha_2$.
\end{lemma}\begin{proof} For two triples $(r_1,n_1,m_1)$ and $(r_2,n_2,m_2)$, the triple $(r_3,n_3,m_3)>(r_1,n_1,m_1),(r_2,n_2,m_2)$ where $r_3=(max(r_1,r_2), n_3=max(n_1,n_2),m_3=max(m_2 - r_2 n_2)$. \end{proof}

\begin{definition}\label{def:N,M_P}
For any $Q=(P,q)\in QCor(X\times\bGm,Y\times\bGm)$ 
we put $-M_P = deg_t\, \mathcal N_{P}(m_{P,u})$
(under the identification $k[\bGm\times\bGm]=k(t,u)$). 
\end{definition}
\begin{lemma}\label{lm:N,M_P}
For any $X,Y\in Sm_k$ and $(P,q)\in Q(\mathcal P(\GXqGY))$ there is an integer $N_P$, 
such that for all $n>N_P$:\begin{itemize}[leftmargin=15pt]
\item[1)]$\Supp P/f^{+}_n P$ is finite over $X$, 
$\mathcal N_P(f^+_n) = (t^n-1)^r\in k[X](t)\subset k[\bGX]$,
\item[2)]$\Supp P/f^-_n P$ is finite over $X$,
$\mathcal N_P(f^-_n) = t^{-M_P}(t^m+a_{m-1}t^{m-1}+\dots +a_0)\in k[X][t]\subset k[\bGX]$, 
where $m=r n + M$ and $r=\rank_{\mathcal O(\bGX)} P$,
and $a_0$ is some invertible regular function on $X$.
\end{itemize}
\end{lemma}
\begin{proof}

By Definition \ref{def:catCohfcalP} for any $P\in \mathcal P(\GXqGY)$,
$p_*(P)$ is locally free coherent sheave on $\bGX$ of finite rank $r$.
Hence for any $f\in k[\bGX]$,
$\Supp P/f P$ is finite over $Z(f)$ that is finite over $X$,
and $N_P(f)= f^r$.
In particular $\Supp P/f^+_n P$ is finite over $Z(t^n-1)\subset X\times\bGm$ which is finite over $X$, 
and $N_P(f^+_n)= (t^n-1)^r$.
This proves the first point.

The proof of the second point
is contained in the proof of proposition 4.1 in \cite{Suslin-GraysonSpectralSeq}.
Let's briefly repeat it.
If $N_P(x-u) = x^n + b_{n-1} x^{n-1} + \dots b_1 x + b_0$, be characteristic polynom of operator $m_P(u)$ on $p_*(P)$,
then $b_i\in k[\bGX]=k[X](t)$ and $b_0 = \mathcal N_P(u)$ is invertible, since $m_P(u)$ is invertible operator.
Let $b_i = c_{i,d_i} t^{d_i} + c_{i,d_i-1} t^{d_i-1} + \dots + c_{i,e_i+1} t^{e_i+1} + c_{i,e_i} t^{e_i}$,
and then $b_0 = t^{e_0} c_{0,e_0}$, $c_{0,e_0}\in k[X]^*$.
For $x=t^n$ we get 
$N_P(t^n-u) = t^{rn} + b_{n-1} t^{(r-1)n} + \dots b_1 t^n + b_0$, 
and 
and for enough big $n$, namely for $n>N_P=\max\limits_i(-e_i+d_{i-1})$, 
the function $t^{-e_0} N_P(t^n-u)$ is monic polynomial in $t$ with coefficients in $k[X]$ and its zero term is equal to $c_{0,e_0}$ and so it is invertible. Now let's note that by definition $-e_0 = M_P$. So we get that for $n>N_P$, $\mathcal N_P(f^-_n) = t^{-M_P} (t^m+a_{m-1}t^{m-1}+\dots +c_{0,e_0})$,
and hence $Z(\mathcal N_P(f^-_n)$ is finite over $X$ and hence $\Supp P$ is finite over $X$.
\end{proof}
\begin{definition}\label{def:alpha_Q}
For any quadratic space $Q=(P,q)\in QCor(X\times\bGm,Y\times\bGm)$ 
put $\alpha_Q = (\rank_{\mathcal O(\bGX)}\, P, N_{P}, max(r N_p, r N_P +M_{P}))\in \mathfrak A.$
where $N_P$ is defined as in the proof of lemma \ref{lm:N,M_P}.\end{definition}

\begin{lemma}\label{lm:existApplBitri-alphaQ} 
For any $Q=(P,q)\in QCor(\GXqGY)$ 
for all $\alpha=(r,n,m)>\alpha_Q\in\mathfrak A$
there is a normal $Q$-applicable 
$(n,m)$-bi-triple (see Def. \ref{def:normappltripl} and def \ref{def:n,m-appltripl}).
\end{lemma}
\begin{proof}
Lemmas \ref{lm:N,M_P} and \ref{lm:appltripl-fNg_min} 
yields that for $n>N$ 
there is a pair of applicable 'f'-triples 
$(f^{+}_n, \mathcal N_P(f^{+}_n), g^{+})$, $(f^{-}_n, \mathcal N_P(f^{-}_n), g^{-})$
that is $(n,rn-M)$-triples.
Denote $\delta_m = m- rn+M$,
then for any $\mathcal N_{ad} = t^\delta_m + \dots +1 \in k[X][t]\subset k[\bGX]$
pair 
$(f^+_n, \mathcal N_P(f^+_n) \mathcal N_{ad}, g^+\mathcal N_{ad})$,
$(f^-_n, \mathcal N_P(f^-_n) \mathcal N_{ad}, g^-\mathcal N_{ad})$
 is applicable $(n,m)$-triple.
\end{proof}

\begin{lemma}\label{lm:DirImPrComd}
For some quadratic space $Q=(P,q)\in QCor(X\times\bGm,Y\times\bGm)$, 
for any morphism of schemes $v\colon Y\to \pri Y$,
we have $\alpha_{\pri Q} = \alpha_Q$,
where $\pri Q=v\circ Q$.
\end{lemma}
\begin{proof}
Let $Q^\prime=(P^\prime,q^\prime)$ the claim follows form that direct images of sheaves $P$ and $P^\prime$ on $X\times\bGm\times\bGm$ are isomorphic.
\end{proof}

\subsection{Sets with action of two commuting idempotents}
In this subsection we fix notations relating to the 
compositions of quadratic spaces $QCor(X\times\bGm,Y\times\bGm)$
with maps $1_{\bGX}\colon \bGX\to X\xrightarrow{1} \bGX$ and $1_{\bGY}\colon \bGY\to Y\xrightarrow{1} \bGY$.

\newcommand{\ip}{\mathcal Id^2}
\newcommand{\ipc}[1]{ \sGcl{#1}{\ip} }
\newcommand{\ips}[1]{ \sGst{#1}{\ip} }

\newcommand{\sGcl}[2]{ CCl_{#2}(#1) }
\newcommand{\sGst}[2]{ St_{#2}(#1) }

\begin{definition}\label{def:sGcl}
Let $P$ be a semi-group and $\mathcal Q$ be a $P$-set (i.e set with action of $P$), 
and let $\mathcal F\subset \mathcal Q$ be any subset.
Then denote by $\sGcl{\mathcal F}{P}$
the maximal subset of $F$ closed under action of $P$, 
and denote by $\sGst{\mathcal Q}{P}$ the subset of $P$-stable elements, i.e.
$$\sGcl{\mathcal F}{P} =\{s\in \mathcal F\colon P\cdot s\in \mathcal F\},\;
\sGst{\mathcal Q}{P} = \{s\in \mathcal Q\colon g\cdot s = s,\forall g\in P\}.$$

If $\mathcal Q(-,-)\to Sm_k\times Sm_k\to Set$ is a bi-functor with $P$-action (i.e. functor $Sm_k\times Sm_k\to P-Set$), 
and $\mathcal F\subset \mathcal Q$ is a $Set$-sub-bi-functor,
we get $P$-Set bi-functor inclusions
$\sGcl{\mathcal F}{P}\subset \mathcal Q$, $\sGst{\mathcal Q}{P}\subset \mathcal Q$. 

\end{definition}

\begin{definition}\label{def:IP}
Let $\ip$ denote semi-group with four elements $\{1,p_1,p_2,p_{1,2}\}$
such that $p_1^2=p^1$, $p_2^2=p_2$, $p_1 p_2 = p_{1,2}$.
Then if $\ip$-Set $A$ is abelian group then it is exactly abelian group with two idempotents $p_1,p_2\in End(A)$.
Let's denote $$dot^A = (id_A - p_1)(id_A- p_2)\in End(A),$$
and then $\ips{A}=\Image(dot^A)$, and we get homomorphisms $ p \colon A \leftrightharpoons \ips{A} \colon i $, $p \circ i=id_{\ips{A}}$.

We define the action of $\ip$ on the bi-functor $(X,Y)\mapsto Q(\mathcal P(\GXqGY))$ as follows:
$
p_1\cdot (P,q) = \ipGX\circ (P,q), 
p_2\cdot (P,q) = (P,q)\circ \ipGY,
p_{1,2}\cdot (P,q) = \ipGX\circ (P,q)\circ \ipGY
.$
\end{definition}

\subsection{Domains of left and right inverse homomorphisms}
In this subsection we define 
a filtering system of
additive bi-functors $L^*_\alpha,N^*_\alpha,R^*_\alpha\colon Sm_k\times Sm_k^{op}\to Ab$, $*\in\{GW,W\}$
with homomorphisms
$L^{GW}_\alpha\rightarrow GWCor(\XGwqYGw)$, 
$N^{GW}_\alpha\rightarrow GWCor(X,Y)$,
$R^{GW}_\alpha\rightarrow GWCor(\XGwqYGw)$
(and similarly for $WCor$).
Bi-functors $R^*_\alpha$ and $L^*_\alpha$ 
play a role of domains of partly defined 
homomorphisms 
$\rho^{*,L}_\alpha\colon L^*_\alpha\to N^*_\alpha$ and
$\rho^{*,R}_\alpha\colon R^*_\alpha\to GWCor(X,Y)$,
that are
left and right inverse homomorphisms to the homomorphism $-\boxtimes id_\bGmw$,
and $N^*_\alpha$ plays the role of codomain for $\rho^{*,L}$. 

\begin{definition}\label{def:GW,W(gruppoid)}
Let $\mathcal Q\subset QCor(X,Y)$ be any subset for smooth varieties $X,Y\in Sm_k$,
then let's denote 
by $Met(\mathcal Q)\subset \mathcal Q$ the subset of the metabolic spaces;
by $\Simpo(\mathcal Q,\oplus )$ the set of triples $Q_1,Q_2,Q_3\in \mathcal Q\colon Q_3= Q_2\oplus Q_1$; 
by $GW(\mathcal Q)$ Grothendieck-Witt group of groupoid $\mathcal Q$ in respect to direct sums,
i.e. $GW(\mathcal Q)=\Coker(\bZ(\Simpo(\mathcal Q))\to \bZ(\mathcal Q))$ where homomorphism sends triple $(Q_1,Q_2,Q_3)$ to a formal sum $Q_3-Q_2-Q_1$;
and denote $W(\mathcal Q)=\Coker(Met(\mathcal Q)\to GW(\mathcal Q))$. 
Note that any action of $\mathcal Q$ on a set, that commutes with direct sums and sends metabolic subspaces to metabolic, 
induce the group action on $GW(\mathcal Q)$ and $W(\mathcal Q)$
And note also that these definitions are functorial, i.e. the same definitions work for any sub-bi-functor $\mathcal Q\subset QCor(-,-)$.
\end{definition}

\begin{definition}\label{def:L^QGWW}
For any any $\alpha=(r,n,m)\in \mathfrak{A}$ let's define bi-functors 
\begin{gather*}
\tilde L^Q_\alpha(-,-), L^Q_\alpha
\colon Sm_k^{op}\times Sm_k\to Set,\;
\tilde L^{GW}_\alpha, L^{GW}_\alpha,
\tilde L^{W}_\alpha, L^{W}_\alpha
\colon Sm_k^{op}\times Sm_k\to Ab
:\\
\tilde L^Q_\alpha(X,Y) = \{Q \in Q(\mathcal P(\GXqGY))| 
\alpha>\alpha_{Q}\},
L^{Q}_\alpha = \ipc{( \tilde L^Q_\alpha )},\;
\tilde L^{GW}_\alpha =  GW(L^Q_\alpha),\;\\
\tilde L^{W} = W(L^Q_\alpha),\;
L^{GW}_\alpha = \ips{( \tilde L^{GW}_\alpha )},\;
L^{W} = \ips{( \tilde L^{W}_\alpha )}
,\end{gather*}
(see Def. \ref{def:N,M_P} and Lemma \ref{lm:DirImPrComd} for the second row;
see Def. \ref{def:sGcl}, Def. \ref{def:IP} and Def. \ref{def:GW,W(gruppoid)} for the third row).
In other words $L^{GW}_\alpha=Im(dot^L_\alpha)$, 
where $dot^L_\alpha\in End(\tilde L^{GW}_\alpha )$ 
is the idempotent defined by the composition $(id_{\bGX} -1_{\bGX})\circ - \circ (id_{\bGY} -1_{\bGY})$.
(see Def. \ref{def:IP}). Set in addition 
$
L^Q_\infty = QCor(\GXqGY),
L^{GW}_\infty = GWCor(\XGwqYGw),
L^{W}_\infty = WCor(\XGwqYGw)
.$
\end{definition}

\begin{definition}\label{def:N^QGWW}
For any $\alpha=(r,n,m)\in \mathfrak A$, define the bi-functor $N^Q_\alpha\colon Sm_k^{op}\times Sm_k\to Set$
as the preimage of $L^Q_\alpha$ under $-\boxtimes id_{bGm}$, i.e.
$$
N^Q_\alpha(X,Y)= (-\boxtimes id_{\bGm})^{-1}(L^Q_\alpha(X,Y))= \{(P,q)\in Q(\mathcal P(X,Y))| (P,q)\boxtimes id_{\bGm}\in L^Q_\alpha\};
$$
and define bi-functors
$
\tilde N^{GW}_\alpha, N^{GW}_\alpha,
\tilde N^{W}_\alpha, N^{W}_\alpha
\colon Sm_k^{op}\times Sm_k\to Ab
:
$
$
N^{Q}_\alpha = \ipc{( \tilde N^Q_\alpha )},\;
\tilde N^{GW}_\alpha =  GW(N^Q_\alpha),\;
\tilde N^{W} = W(N^Q_\alpha),\;
N^{GW}_\alpha = \ips{( \tilde N^{GW}_\alpha )},\;
N^{W} = \ips{( \tilde N^{W}_\alpha )}
$
(see Def. \ref{def:sGcl}, Def. \ref{def:IP} and Def. \ref{def:GW,W(gruppoid)}).
Set in addition 
$
N^Q_\infty = QCor(X,Y),
N^{GW}_\infty = GWCor(X,Y),
N^{W}_\infty = WCor(X,Y)
.$
\end{definition}

\newcommand{\XGGqYGG}{X\times\bGm\times\bGm,Y\times\bGm\times\bGm}
\newcommand{\XGGAqYGG}{X\times\bGm\times\bGm\times\affl,Y\times\bGm\times\bGm}
\begin{lemma}\label{lm:permhomQCor}
There are maps
$$h^{+/-}_{per}\colon QCor(\XGGqYGG)\to QCor(\XGGAqYGG)$$
such that $h^{+/-}$ are natural in $X$ and $Y$ and preserve direct sums and metabolic spaces,
and maps (not necessarily natural)
$$h^{*}_{per,b}, h^{*,+/-}_{per,d}\colon QCor(\XGGqYGG)\to QCor(\XGGqYGG), *\in \{0,1\}$$ 
such that for any $Q\in QCor(\XGGqYGG)$
the spaces $h^{*,+/-}_{per,d}(Q)$ split into the sum of spaces that supports are sent to the unit under the projection to at least one of multiplicands $\bGm$, and such that
\begin{equation}\label{eq:h_pre-cond}\begin{array}{cccc}
i_0^*(h_{per}^+(Q)) \oplus h^0_{per,b}(Q)\oplus h^{0,+}_{per,d}(Q)&\simeq& Q &\oplus i_0^*(h_{per}^-(Q)) \oplus h^0_{per,b}(Q)\oplus h^{0,-}_{per,d}(Q)\\
i_1^*(h_{per}^+(Q)) \oplus h^1_{per.b}(Q)\oplus h^{1,+}_{per,d}(Q)&\simeq& T_{Y}\circ Q \circ T_{X} &\oplus i_1^*(h_{per}^-(Q)) \oplus h^1_{per,b}(Q)\oplus h^{1,-}_{per,d}(Q)
,\end{array}\end{equation}
where 
$T_{X} = id_X\boxtimes T$, $T_Y=id_Y\boxtimes T$ are transpositions on $X\times \bGm\times\bGm$ and $Y\times \bGm\times\bGm$, and
$i_0,i_1$ denote zero and unit sections.
\end{lemma}\begin{proof}
Lemma \ref{prop:permutGmGm} implies that there are
\begin{equation}\label{eq:hpreQSp}\begin{aligned}
h^+, h^-\in QCor(\bGm\times\bGm\times\affl,\bGm\times\bGm),\,
h^0, h^1\in QCor(\bGm\times\bGm,\bGm\times\bGm)\colon\\
i^*_0(h^+)\oplus h^0 = (k[\Delta],1)_\oplus h^0,\;
i^*_1(h^+)\oplus h^1 = (k[\Gamma_T],1)_\oplus h^1,\\
\end{aligned}\end{equation}
and such that spaces $h_0$ and $h_1$ splits into the sum of spaces that supports are 
contained in $pr_i^{-1}(1)$,
where $pr_i \colon \bGm\times\bGm\times\bGm\times\bGm\to \bGm$ for $1\leq i\leq 4$ denotes projections to multiplicands.
The operation of composition  at the left with the class $[h^+]-[h^-]$ defines a homotopy on $GWCor(X\times\bGmw\times\bGmw,Y\times\bGmw\times\bGmw)$
that permutes multiplicands in the domain.
The
operation of composition at the right with $[h^+]-[h^-]$ defines a homotopy that permutes multiplicands in the codomain.
Now composing this two homotopies we get the claim. 

Indeed, 
let $i_\Delta\colon U\times\affl \rightarrow U\times\affl\times\affl $ denote the diagonal embedding for $U\in Sm_k$ and let
$$h_1\star h_2(Q) = i_\Delta^*( ( h_2 \circ (Q\circ h_1) \boxtimes \affl) )\in QCor(U\times\affl\times\affl,U),$$
for any two quadratic spaces $h_1\in QCor(U\times \affl,U)$,$h_2\in QCor(V\times \affl,V)$ 
and $Q\in QCor(U,V)$.
Then we can define the required maps $h^+_{per}$ and $h^-_{per}$ as
\begin{equation*}\begin{array}{ccc}
h^+_{per}\colon Q&\mapsto& (id_X\boxtimes h^+)\star (Y\boxtimes h^+)(Q)\oplus (X\boxtimes h^-)\star (Y\boxtimes h^-)(Q), \\
h^-_{per}\colon Q&\mapsto& (X\boxtimes h^+)\star (Y\boxtimes h^-)(Q)\oplus (X\boxtimes h^-)\star (Y\boxtimes h^+)(Q).
\end{array}\end{equation*}
The equations \eqref{eq:hpreQSp} implies that
$$\begin{array}{ccc}
i^*_0([h^+_{per}(Q)]-[h^-_{per}(Q)])=&[Q]&\in GWCor(\bGm\times\bGm,\bGm\times\bGm),\\
i^*_1([h^+_{per}(Q)]-[h^-_{per}(Q)])=&[T_Y\circ Q\circ T_X]&\in GWCor(\bGm\times\bGm,\bGm\times\bGm),
\end{array}$$
and this implies existence of the maps $h^0_{per,b}$, $h^1_{per,b}$, $h^{0,+}_{per,d}$, $h^{0,-}_{per,d}$, $h^{1,+}_{per,d}$, $h^{1,-}_{per,d}$.

\end{proof}

\begin{definition}\label{def:Ralpha}
For any any $\alpha=(r,n,m)\in \mathfrak{A}$ let's define the bi-functor $R^Q_\alpha\colon Sm_k^{op}\times Sm_k\to Set$
\begin{multline*}
R^Q_\alpha(X,Y)= 
\{Q=(P,q)\in L^Q_\alpha(X,Y)\cap N^Q_\alpha(X\times\bGm,Y\times\bGm) | \\
  T_Y\circ (Q\boxtimes id_\bGm) T_X,
  h_{per,b}^{*}(Q), h_{per,d}^{*,\star}(Q) \in L^Q_\alpha(\bGX,\bGY), *\in \{0,1\},\\
  h_{per}^{\star}(Q) \in L^Q_\alpha(\bGX\times\affl,\bGY), \star\in \{+,-\}
\} 
,\end{multline*}
where $h^{\star}_{per}$, $h_{per,b}^{*}(Q)$, $h_{per,b}^{*,\star}(Q)$ denote maps from Lemma \ref{lm:permhomQCor},
and define bi-functors
\begin{gather*}
\tilde R^{GW}_\alpha, R^{GW}_\alpha,
\tilde R^{W}_\alpha, R^{W}_\alpha
\colon Sm_k^{op}\times Sm_k\to Ab
\\
R^{Q}_\alpha = \ipc{( \tilde R^Q_\alpha )},
\tilde N^{GW}_\alpha =  GW(R^Q_\alpha),
\tilde N^{W} = W(R^Q_\alpha),\\
R^{GW}_\alpha = \ips{( \tilde R^{GW}_\alpha )},
R^{W} = \ips{( \tilde R^{W}_\alpha )}
\end{gather*}
(see Def. \ref{def:sGcl}, Def. \ref{def:IP} and Def. \ref{def:GW,W(gruppoid)}).
Set in addition
$
R^Q_\infty = QCor(\GXqGY),
$
$R^{GW}_\infty = GWCor(\XGwqYGw),$
$R^{W}_\infty = WCor(\XGwqYGw)
.$
\end{definition}

\begin{lemma}\label{lm:additivityfunctors}
1) For any $Y\in Sm_k$, functors 
$L^*_\alpha(-,Y), N^*_\alpha(-,Y),R^*_\alpha(-,Y)$, $*\in \{Q,\GW,W\}$
are additive, i.e. sends disjoint unions to products in corresponding category.

2) For any $X\in Sm_k$ and $
\alpha\in\mathfrak{A}$, the functors
$L^*_\alpha(X,-), N^*_\alpha(X,-), R^*_\alpha(X,-)$, $*\in \{\GW,W\}$
are additive, i.e. $L^*_\alpha$, $R^*_\alpha$ send disjoint unions in $Sm_k$ to the direct sums in corresponding additive category.
\end{lemma}
\begin{proof}
1)
Firstly note that $\tilde L_\alpha^Q(X_1\amalg X_2,Y)=L^Q_\alpha(X_1,Y)\times L^Q_\alpha(X_2,Y)$, since  
required conditions on rank can be checked on disjoint components,
and pair of 'f'-triples on disjoint union is normal applicable $(n,m)$-bi-triple if it is so at each disjoint component.
The claim for other functors defined in Def. \ref{def:L^QGWW}, Def. \ref{def:Ralpha}
follows from that additivity preserves under used in definitions operations 
and from that direct sums of quadratic spaces and property of metabolic spaces are compatible with disjoint union of schemes.

2)
To show additivity in the second argument consider the pair of homomorphisms
$$\begin{array}{ccc}
\mathcal F^\oplus(X,Y_1) \oplus \mathcal F^\oplus(X,Y_2) &\leftrightharpoons& \mathcal F^\oplus(X,Y_1\amalg Y_2)\colon \\
 ( [(P_1,q_1)] , [(P_2,q_2)] )&\mapsto& [(P_1,q_1)]+[(P_2,q_2)],\\
 ( [(P,q)\big|_{X\times Y_1}],[(P,q)\big|_{X\times Y_2}] )&\gets& [ (P,q) ]
\end{array}$$
where $\mathcal F \in \{ L^Q,N^Q,R^Q \}$
and in the definition of the second map we use that support of any $(P,q)\in Q(\mathcal P(X,Y_1\amalg Y_2))$ splits into disjoint union and hence
quadratic space splits into direct sum $(P,q)\big|_{X\times Y_1}\oplus (P,q)\big|_{X\times Y_2}$.
Checking that composition are identity we get
isomorphism $(\mathcal F)^\oplus(X,Y_1\amalg Y_2) \simeq (\mathcal F)^\oplus(X,Y_1) \oplus (\mathcal F)^\oplus(X,Y_2)$
and 
checking that action of $\ip$ is compatible with this decomposition 
we get decomposition for $\ips{(\mathcal F^\oplus)}$.
\end{proof}

\begin{definition}
Lemma \ref{lm:additivityfunctors} implies that
the presheaves of abelian groups 
$L^*_\alpha$, $N^*_\alpha$, $R^*_\alpha$, $*\in \{GW,W\}$
can be uniquely lifted to a bi-additive bi-functors on Karoubi closure of the additivisation of $Sm_k$,
and in particular these presheaves are correctly defined on $X\times\bGmw$ for $X\in Sm_k$.
We shell denote such lifts of these functors by the same symbols.

Precisely $L^*(\bGwX,\bGwY)$ 
 is defined as 
\begin{multline*}
L^*(\bGwX,\bGwY) = \\ =\Image \Bigl( [(P,q)] \mapsto [(P,q)] - [{\ipGX}^*(P,q)]  - [{\ipGY}_*(P,q)] + [{\ipGX}^*({\ipGY}_*(P,q))] \Bigr)
\end{multline*} 
\end{definition}

\begin{lemma}\label{lm:surlimFNL}
For any varieties $X,Y$ and a quadratic space $(P,q)\in Q(\mathcal P(\GXqGY))$ 
there is $a\in \mathfrak{A}$ such that for all $\alpha>a$, 
$(P,q)\in L^Q_\alpha(X,Y)$ 
and 
$(P,q)\in R^Q_\alpha(X,Y)$.
\end{lemma}
\begin{proof}
Indeed by Definition \ref{def:L^QGWW} for any $Q = (P,q)$ and $\alpha> \alpha_Q$, 
$(P,q)\in \tilde L^Q_\alpha$.
Then since definitions of $L^Q_\alpha$ (and $R^Q_\alpha$) requires that some finite set of quadratic spaces lays in $\tilde L^Q_\alpha$ (and $L^Q_\alpha$ respectively), and since $\mathfrak A$ is filtering, the claim follows.
\end{proof}

\begin{lemma}\label{lm:alphaarrows}
For any $Y\in Sm_k$
and any $\alpha_1 , \alpha_2 \in \mathfrak{A}\cup \{\infty\}$
such that $\alpha_1 < \alpha_2$, 
there is inclusions of presheaves
$L^Q_{\alpha_1}\subset L^Q_{\alpha_2}$,
$N^Q_{\alpha_1}\subset N^Q_{\alpha_2}$,
$R^Q_{\alpha_1}\subset R^Q_{\alpha_2}$,
and this induce homomorphisms of presheaves
$\gamma^{*,L}_{\alpha_1,\alpha_2}\colon L^*_{\alpha_1} \to L^*_{\alpha_2}$, 
$\gamma^{*,N}_{\alpha_1,\alpha_2}\colon N^*_{\alpha_1}\to N^*_{\alpha_2}$
 and $\gamma^{*,R}_{\alpha_1,\alpha_2} \colon R^*_{\alpha_1} \to R^*_{\alpha_2}$
 (for $*\in \{GW,W\}$).
Let's denote 
$j^{GW,L}_\alpha = \gamma^{GW,L}_{\alpha,\infty}\colon L^{GW}_\alpha(-,-)\to GWCor(X\times\bGmw,Y\times\bGmw)$,
and similarly for 
$j^{GW,N}_\alpha$, $j^{GW,R}_\alpha$,
and $j^{W,L}_\alpha$, $j^{W,N}_\alpha$, $j^{W,R}_\alpha$.
\end{lemma}
\begin{proof}
By Definition \ref{def:L^QGWW} the same statement holds for the bi-functor $\tilde L^Q$,
and the claim follows by induction in the definition of the bi-functors $L^Q$, $N^Q$, $R^Q$.
\end{proof}

\begin{lemma}\label{lm:limFNL}
For each $*\in \{Q,GW,W\}$ there are isomorphisms
$$
\varinjlim\limits_{\alpha\in \mathfrak A} L^*_\alpha(X,Y) \simeq L^*_\infty,\;
\varinjlim\limits_{\alpha\in \mathfrak A} N^*_\alpha(X,Y) \simeq N^*_\infty,\;
\varinjlim\limits_{\alpha\in \mathfrak A} R^*_\alpha(X,Y) \simeq R^*_\infty
.$$
\end{lemma}
\begin{proof}
Lemma \ref{lm:surlimFNL} yields that
$
\bigcup\limits_{\alpha\in \mathfrak A} L^Q_\alpha(X,Y) = L^Q_\infty,\;
\bigcup\limits_{\alpha\in \mathfrak A} \Simpo(L^Q_\alpha(X,Y),\oplus) = \Simpo(L^Q_\infty,\oplus)
,$
and hence 
\begin{multline*}
\varinjlim\limits_\alphind (L^Q_\alpha(X,Y))^\oplus = 
\varinjlim\limits_\alphind \bZ(L^Q_\alpha(X,Y))/\bZ(\Simpo(L^Q_\alpha(X,Y),\oplus)) = \\
\bZ(L^Q_\infty(X,Y))/\bZ(\Simpo(L^Q_\infty(X,Y),\oplus)) = (L^Q_\infty(X,Y))^\oplus = 
GWCor(\GXqGY)
.\end{multline*}
 
Then 
since action of the idempotents $dot_\alpha$ (see Def. \ref{def:IP}) 
commutes with homomorphisms in the direct limit, we get
$$
\varinjlim\limits_\alphind L^{GW}_\alpha(X,Y) = 
\varinjlim\limits_\alphind Im(dot^L_\alpha) = 
Im(dot^L_\infty) = 
GWCor(\GwXqGwY).$$
Finally, since 
$\bigcup\limits Met(L^Q_\alpha) = Met(L^Q_\infty)$ (by Lemma \ref{lm:surlimFNL}),
we have
$\varinjlim\limits_\alphind L^{W}_\alpha(X,Y) = WCor(\GwXqGwY)$.

The proof of the equalities for $N^*_\alpha$ and $R^*_\alpha$ is similar.

\end{proof}

\subsection{Left and right inverse homomorphisms.}
Here we prove that homomorphisms $-\times\bGmw$ on $GWCor$ 
has a left and right inverse homomorphisms defined on presheaves $F^{GW}_\alpha$ and $L^{GW}_\alpha$,
and similarly for $WCor$.
We prove precisely that $\rho^{GW,L}_\alpha$ and $\rho^{GW,R}_\alpha$ defined in previous subsection becomes such inverse homomorphisms after the twist by some invertible function.

\begin{lemma}\label{lm:rho_alpha}
For any $\alpha=(r,n,m)\in \mathfrak A$, and any choice of sign $\star \in \{+.-\}$ 
there are unique natural transformations $\tilde\rho^*_{\alpha,\star}$, $*\in \{Q,GW,W\}$
such that diagram
$$\xymatrix{
QCor(-,Y)\ar[r] & GWCor(-,Y) \ar[r] & WCor(-,Y)\\
\tilde L^{Q}_{\alpha}(-,Y) \ar@{^(->}[d] \ar[r] \ar[u]^{\tilde\rho^Q_{\alpha,\star}}& 
\tilde L^{GW}_{\alpha}(-,Y)\ar[d] \ar[r] \ar[u]^{\tilde\rho^{GW}_{\alpha,\star}} & 
\tilde L^{W}_\alpha(-,Y)\ar[d] \ar[u]^{\tilde \rho^{GW}_{\alpha,\star}}\\
Q(\mathcal P(-\times\bGm, Y\times\bGm))\ar[r] & 
GWCor(\bGX, \bGY) \ar[r]  & 
WCor(\bGX, \bGY)
}
$$
is  commutative, and such that for any
$(P,q)\in F^Q_{r,n,m}(X,Y)$
and applicable $(n,m)$-bi-triple $\tau$, 
$$\tilde\rho^Q_{\alpha,\star}(P,q) =\tilde\rho_{\tau}^\star(P,q)\boxtimes \langle \beta_{n,m}^{-1}\rangle,$$
where $\beta_{n,m}$ are defined in Lemma \ref{lm:GWCortoAffl}.
\end{lemma}
\begin{proof}
Lemma \ref{lm:existApplBitri-alphaQ} implies that for any $Q=(P,q)\in L^Q_\alpha$ there is normal applicable $(n,m)$-bi-triple.
Let $\tau$ be such a bi-triple,
then we can put $\tilde\rho_{\alpha,\star}^Q(Q) = \rho_{\tau}^\star(Q)\boxtimes \langle \beta_{n,m}^{-1}\rangle$,
and by Lemma \ref{lm:rho:indbitrip} the result is independence from the choice of $\tau$.


Then to show that $\rho^Q_{\alpha,\star}$ induce homomorphisms $\rho^{GW}_{\alpha,\star}$ and $\rho^{W}_{\alpha,\star}$ is enough to check that it preserves direct sums and sends metabolic spaces to metabolic.
The second one follows immediate from Lemma \ref{lm:rho(metabolic)}.
And to show that $\rho^Q_{\alpha,\star}((P_1,q_1)\oplus (P_2,q_2))\simeq \rho^Q_{\alpha,\star}(P_1,q_1)\oplus \rho^Q_{n,m}(P_2,q_2)$ we can choose any applicable 'f'-triple $(f^{+/-}_n,\mathcal N^{+/-}, g^{+/-})$ for $(P_1,q_1)\oplus (P_2,q_2)$ and apply it for all three quadratic spaces. 

\end{proof}

\begin{definition}\label{def:rhoGWalpha}
Define 
a natural transformations $$\rho^{*}_\alpha
 = \rho^{*}_{\alpha,+}-\rho^{*}_{\alpha,-}, \;
\rho^{*}_{\alpha,+/-}(P,q) = 
 \tilde\rho^{*}_{\alpha,+.-}((id_\bGY-\ipGY)\circ(P,q)\circ(id_{\bGX}-\ipGX)),$$
for $*\in\{GW,W\}$ and $\alpha=(r,n,m)\in \mathfrak A$. 
Then we get commutative diagram
$$\xymatrix{
GWCor(X\times\bGmw, Y\times\bGmw)\ar[d] & L^{GW}_\alpha(X,Y)\ar[r]^{\rho^{GW}_\alpha}\ar[l]\ar[d] & GWCor(X,Y)\ar[d]\\
WCor(X\times\bGmw, Y\times\bGmw) & L^{W}_\alpha(X,Y)\ar[r]^{\rho^{W}_\alpha}\ar[l] & WCor(X,Y)
.}$$
\end{definition}
\begin{remark}
Note that homomorphisms $\rho^{*}_\alpha$ are not injective and we don't state and use that these homomorphisms are natural in $\alpha$. 
\end{remark}

\begin{definition}\label{def:thetaalpha}
Let's denote by $\theta^{*}_\alpha\colon N^{GW}_\alpha(-,-)\to L^{GW}_\alpha(-,-)$
the natural transformations of presheaves induced by $-\boxtimes id_{\bGm}\colon QCor(X,Y)\to QCor(\GXqGY)$ 
(in particular $\theta_\infty = -\boxtimes id_{\bGm}\colon GWCor(X,Y)\to GWCor(\GXqGY)$). 
\end{definition}

\begin{lemma}\label{lm:alpharho-diag}
For any $\alpha=(r,n,m)\in \mathfrak A$ and a quadratic space $Q\in \tilde{L}_\alpha^{Q}(\GXqGY)$,
\begin{multline*}
\tilde{\rho}^{GW,L}_\alpha([Q]) = 
[Q\circ pr_{X}] \circ ([red(g^+\cdot\langle g^+ (t^n-1),\bGX \rangle)] - \\-
[Q\circ pr_{X} \circ red(g^-\cdot\langle g^- (t^n-t),\bGX \rangle)])
\boxtimes \langle \beta_{n,m}^{-1}\rangle
,\end{multline*}
for some functions $g^{+/-}\in k[\bGX]\simeq k[X](t)$ with the leading term $t^{m-r}$,
where $\beta_{n,m}$ are defined as in Lemma \ref{lm:GWCortoAffl}.
\end{lemma}\begin{proof}
By definition 
\begin{multline*}
\tilde{\rho}^{GW,L}_{\alpha}([Q]) = 
[(Q\boxtimes id_{\bGm})\circ red({\pri g}^+\cdot\langle \mathcal N^+,\bGX \rangle)] - \\-
[(Q\boxtimes id_{\bGm})\circ red({\pri g}^-\cdot\langle \mathcal N^-,\bGX \rangle)]
\end{multline*}
where $(f^+_n,\mathcal N^+,{\pri g}^+)$, $(f^-_n,\mathcal N^-,{\pri g}^-)$
is applicable $(n,m)$-bi-triple for $(Q\boxtimes id_{\bGm})$.
Denote $(P,q) = Q\boxtimes id_{\bGm}$,
then $\Supp P = X\times Y\times\Delta_{\bGm}$,
and let's denote $g^+ = i_{\Delta}^*({\pri g}^+)$, $g^- = i_{\Delta}^*({\pri g}^-)$,
where $i_\Delta\colon X\times\bGm\to X\times\bGm\times Y\times\bGm$.
Then,
since $f^+_n = t^n-1$, $f^-_n\big|_{\Delta_{\bGm}} = t^n-t$,
$m_{\mathcal N^+,P} = m_{t^n-1,P}\circ m_{g^+,P}$, $m_{\mathcal N^-,P} = m_{t^n-t,P}\circ m_{g^-,P}$,
and then
$(f^+_n, (t^n-1)g^+ ,g^+)$, $(f^-_n, (t^n-t)g^- ,g^-)$
is applicable $(n,m)$-bi-triple for $(P,q)$.
The claim follows by Lemma \ref{lm:rho:indbitrip}.

\end{proof}
\begin{lemma}\label{lm:alpharho-unitsect}
For any $\alpha=(r,n,m)\in \mathfrak A$ and a quadratic space $Q=(P,q)\in N^Q_\alpha(X,Y)$,
$$\tilde{\rho}^{GW,L}_\alpha([i_Y \circ Q\circ pr_X]) = 0
,$$
where $pr_X\colon \bGX\to X$ and $i_Y\colon Y\to \bGY$ is unit section.
\end{lemma}\begin{proof}
Since $\Supp(i_Y\circ Q\circ pr_X)\subset \bGX\times Y\times 1$, then 
$f^+_n\big|_{\Supp Q} = t^n-1$.
Let $\mathcal N_{ad}^+ = t^{m-n}+\dots +(-1)^{r-1}$ and $\mathcal N_{ad}^- = t^{m-n}+\dots +(-1)^{r-1}t$ be regular functions in $k[\bGm]=k(t)$,
and $\mathcal N^+ = \mathcal N_{ad} (t^n-1)^r$, $\mathcal N^- = \mathcal N_{ad} (t^n-1)^r$,
$g^+ = \mathcal N_{ad} (t^n-1)^{r-1}$, $g^- = \mathcal N_{ad} (t^n-1)^{r-1}$.
Then pair 
$(f^+_n, \mathcal N^+,g^+)$, $(f^-_n, \mathcal N^-,g^-)$ is $(n,m)$-bi-triple applicable to $i_Y\circ Q\circ pr_X$,
then using Lemma \ref{lm:rho:indbitrip} 
and the second point of Lemma \ref{lm:GWCortoAffl}
\begin{multline*}
\tilde{\rho}^{GW,L}_\alpha([i_Y\circ Q\circ pr_X])=
[\rho_\alpha^+(i_Y\circ Q\circ pr_X)] - [\rho_\alpha^+(i_Y\circ Q\circ pr_X)]=\\
[i_Y\circ Q\circ pr_X \circ red(g^+ \langle \mathcal N^+,\bGX\rangle)] -
[i_Y\circ Q\circ pr_X \circ red(g^- \langle \mathcal N^-,\bGX\rangle)]=\\
[i_Y \circ Q]\circ 
[pr_X \circ red(g^+ \cdot\langle g^+ (t^n-1),\bGX\rangle) -
pr_X \circ red(g^- \cdot\langle g^- (t^n-1),\bGX\rangle)] =
0
\end{multline*}
\end{proof}

\begin{lemma}[the left inverse]\label{lm:leftinverse}
Diagrams
$$
\xymatrix{
N^{GW}_\alpha(X,Y)\ar[r]^{\theta^{GW}_\alpha}\ar[d]^{j^{N}_\alpha}& 
L^{GW}_\alpha(X,Y) \ar@{-->}[ld]|{\rho^{GW}_\alpha}\ar[d]^{j^L_\alpha}\\ 
GWCor(X,Y)\ar[r]_<<<<<{\theta^{GW}_\infty} & GWCor(\GwXqGwY)
}
$$
are natural in $X$ and $Y$, 
and \begin{itemize}[leftmargin=15pt]
\item[1)] the square in these diagrams is commutative and it is natural in $\alpha$, 
i.e. it commutes with morphisms $N^*_{\pri \alpha}\to N^*_\alpha$ and $F^*_{\pri \alpha}\to F^*_\alpha$;
\item[2)] left-up triangle is commutative for each $\alpha$ up to canonical $\affl$-homotopy, i.e.
 there are homomorphisms of presheaves $ h^{GW,l}_{\alpha}\colon N^{GW}_\alpha(-,Y)\to GWCor(-\times\affl, Y)\colon
$ such that for any class $a\in N^*_\alpha(X,Y)$ and $\alpha\in \mathfrak A$ there are equalities
\begin{equation}\label{eq:lefthomot}
 i_0^*(h^{GW,l}_\alpha(a)) = \rho^{GW,l}_{\alpha}( a \boxtimes id_{\bGm} ), 
 i_1^*(h^{GW,l}_\alpha(a))= a,
\end{equation}
where $i_0^*$ and $i_1^*$ denote homomorphism 
induced by zero and unit sections of $\affl$,
and see Def. \ref{def:rhoGWalpha} and Def. \ref{def:thetaalpha} for $\rho^{GW}_\alpha$ and $\theta_\alpha$. 
\end{itemize}
And the same holds for $WCor$.
(Note: we don't state that homomorphisms $\rho^*_\alpha$ and $h^l_\alpha$ are natural in $\alpha$.)
\end{lemma}\begin{proof}
We consider the case of $GWCor$, the case of $WCor$ is absolutely similar.

The first point follows from definitions.
To prove the second we 
use second point of Lemma \ref{lm:GWCortoAffl} and find homotopy $h\in GWCor(\affl\to pt)$,
then define $h_X\in GWCor(X\times\affl,X)\colon h_X =id_X\boxtimes h$, and 
$$
h^{GW,L}_\alpha(a) = j_\alpha^{GW,N}(a) \circ h_X\boxtimes\langle \beta_{n,m}^{-1}\rangle\in GWCor(X\times\affl,Y)
,$$
for any $a\in N_\alpha^{GW}(X,Y)$.

Let's check required conditions.
Let $a\in N^{GW}_\alpha(X,Y)$, $a=[Q]$ for $Q\in N^Q_\alpha$.
And let's denote $Q^\prime = Q\boxtimes id_\bGm$
Then $\pri Q \circ 1_{\bGX} \simeq 1_{\bGY}\circ \pri Q \simeq 1_{\bGY}\circ \pri Q \circ 1_{\bGX}\simeq i_Y\circ Q\circ pr_X$,
where $i_Y\colon Y\to \bGY$ is unit section and $pr_X\colon \bGX\to X$ is projection.
Then
$$\rho_{\alpha}^{GW,L}(\theta_\alpha(a))=
\rho_{\alpha}^{GW,L}([\pri Q])=
\tilde{\rho}_{\alpha}^{GW,L}([\pri Q] - [i_Y\circ Q\circ pr_X]).$$

Now  by Lemma \ref{lm:alpharho-diag} and by the second point of Lemma \ref{lm:GWCortoAffl}
\begin{multline*}
\tilde{\rho}_{\alpha}^{GW,L}([\pri Q]) = \\
[Q\circ pr_X]\circ [red(g^+\cdot\langle g^+ (t^n-1),\bGX \rangle)-red(g^-\cdot\langle g^- (t^n-t),\bGX \rangle)]\boxtimes\langle \beta_{n,m}^{-1}\rangle =\\
[Q \circ h_X\circ i_0\boxtimes\langle \beta_{n,m}^{-1}\rangle] = i_0^*(h^{GW,L}_\alpha([Q]))
,\end{multline*}%
and by Lemma \ref{lm:alpharho-unitsect} 
$\tilde{\rho}_{\alpha}^{GW,L}([i_Y\circ Q\circ pr_X]) = 0$.
And on other side (by Lemma \ref{lm:GWCortoAffl})
$$
[Q] =
[Q \circ h_X\circ i_1\boxtimes\langle \beta_{n,m}^{-1}\rangle] = i_1^*(h^{GW,L}_\alpha([Q]))
,$$
so the claim follows.
\end{proof}

\begin{lemma}[the right inverse]\label{lm:rightinverse}
For any $Y\in Sm_k$ and $\alpha\in\mathfrak A$, 
there is homomorphism of presheaves
\begin{multline*}
h^{GW,r}_\alpha\colon R^{GW}_\alpha(-,Y)\to GWCor(-\times\bGmw\times\affl, Y\times\bGmw)\colon \\
i_0^*(h^{GW,r}_\alpha(-)) = \rho^{GW,R}_{\alpha}(-)\boxtimes id_{\bGmw} ,\;
i_1^*(h^{GW,r}_\alpha) = j^{GW,R}_\alpha
,\end{multline*}
where $i_0^*$ and $i_1^*$ denotes natural transformations induced by zero and unit sections.

And there is $h^{W,r}_\alpha$ with the same property in respect to $R^W_\alpha$ and $WCor$.
\end{lemma}
\begin{proof}
We consider the case of $GWCor$, the case of $WCor$ is absolutely similar.

Since $R^Q_\alpha(X,Y)\subset N^Q_\alpha(\GXqGY)$ (see Def. \ref{def:Ralpha}),
it follows that there is a natural transformation 
\begin{equation}\label{eq:hRl}
h^l_\alpha\colon R^{GW}_\alpha(-,Y)\to GWCor(-\times\affl,Y)\colon\,
i_0^*(h^{l}_\alpha(a)) = \rho^{GW,l}_{\alpha}( \theta_\alpha^{GW}(a) ), \,
i_1^*(h^{l}_\alpha(a))= a
.\end{equation}

\newcommand{\XGAqYG}{X\times\bGm\times\affl,Y\times\bGm}
\newcommand{\XGwAqYGw}{X\times\bGmw\times\affl,Y\times\bGmw}
Next define natural maps
$$
h^{Q,+/-,per}_\alpha\colon R^Q_\alpha(X,Y)\to L^Q_\alpha(\XGAqYG)\colon
h^{Q,+/-,per}_\alpha(Q) = h^{+/-}_{per}(\theta_\alpha^Q(Q))
,$$
where $h^{+/-}_{per}$ are maps from Lemma \ref{lm:permhomQCor}. 
Denote by $\tilde{h}^{GW,+}_\alpha$, $\tilde{h}^{GW,-}_\alpha$ the induced homomorphisms of Grothendieck-Witt groups, 
and set
\begin{gather*}
\tilde{h}^{GW,per}_\alpha = \tilde{h}^{GW,+,per}_\alpha -\tilde{h}^{GW,-,per}_\alpha\colon 
\tilde{R}^{GW}_\alpha(X,Y)\to \tilde{L}^{GW}_\alpha(\XGAqYG),\\
h^{GW,per}_\alpha =  p\circ p^{GW,L}_\alpha\circ \tilde{h}^{GW,per} \circ i^{GW,R}_\alpha\colon 
{R}^{GW}_\alpha(X,Y)\to {L}^{GW}_\alpha(\XGwAqYGw),
\end{gather*}
where $p$ denotes projection 
$L^{GW}_\alpha(\XGAqYG)\to L^{GW}_\alpha(\XGwAqYGw)$
Then equations \eqref{eq:h_pre-cond} implies that for any $a\in R^{GW}_\alpha(X,Y)$,
\begin{equation}\label{eq:hGWperbord}
i_0^*( h^{GW,pre-r}_\alpha(a) ) = \theta^{GW}_\alpha(a),\; 
i_1^*( h^{GW,pre-r}_\alpha(a) ) = T_Y \circ \theta_\alpha^{GW}(a)\circ T_X
,\end{equation}
since for any $Q\in QCor(\XGGqYGG)$, we have $p(p^{GW,L}([h^0(Q)]))=0$ 
(With ore details: 
since for any $Q\in QCor(X\times\bGm,Y\times\bGm)$ such that $\Supp Q\subset X\times\bGm\times 1\times \bGm$ we have $p^{GW,L}([Q])=0$,
and for any $Q\in QCor(X\times \bGm\times \bGm, Y\times\bGm\times\bGm)$ such that $\Supp Q\subset X\times 1\times \bGm\times Y\times\bGm\times\bGm$ we have $p([Q])=0$ (and similarly for $Y$ instead of $X$)). 


Now define required homotopy as composition of homotopies $h^l_{\alpha}$ and $\rho^{GW}_\alpha \circ h^{GW,pre}_\alpha$, i.e. 
$$h^{GW,r}_\alpha = \rho^{GW}_\alpha \circ  h^{GW,pre}_\alpha - h^{l}_\alpha +id_{L^{GW}_\alpha(-,Y)}.$$
Then form \eqref{eq:hRl} and \eqref{eq:hGWperbord} we get 
$$
i_0^* (h^{GW,r}_\alpha(a)) = i_\alpha
,
i_1^* (h^{GW,r}_\alpha(a)) = \theta_\alpha( \rho^{GW,R}_\alpha(a) )
$$
where in the second equality we use that  
$$\rho_\alpha^{+/-}(T_Y \circ(Q\boxtimes\bGm)\circ T_X) = 
\rho_\alpha^{+/-}(( \bGm\boxtimes Q)) =
\bGm\boxtimes \rho_\alpha^{+/-}(Q)=
\rho_\alpha^{+/-}(Q)\boxtimes \bGm
,$$
for any $Q\in R^Q_\alpha(X,Y)$ (see Lemma \ref{lm:rho:Gmbasechange} for the middle equality).

\end{proof}

\section{Cancellation theorem.}\label{sect:CancellationTh}

Firstly we prove cancellation in the category $\DaffGW(k)=D^-_{\affl}(\PreGW)$ (see sect. \ref{sect:DMGWeff}). 
Let's recall that there is a canonical functor $Sm_k\to \DaffGW(k)$ which we denote 
$C^*(V)=C^*(GWCor(-,V)) = \mathcal Hom_{\Dba(\PreGWs)}(\Delta^\bullet, \mathbb Z_{GW}(V))$.
In the section we consider the case of GW-correspondneces, 
All statements and proofs in the section holds for $\WCor$ as well.


\begin{notation}
Denote by 
$h^i_{\affl}(\mathcal F) = h^i(C^*(\mathcal F)) = h^i(\mathcal Hom(\Delta^\bullet, \mathcal F))=H^i(\mathcal F(-\times\Delta^\bullet))$
the presheaves of cohomologies
of $\mathcal F\in \PreGWs$.
And for any $s\colon \mathcal F\to \mathcal G\in \PreGWs$,
$h^i_{\affl}(s)\colon h^i_{\affl}(\mathcal F)\to h^i_{\affl}(\mathcal G)$ is the induced homomorphism.
\end{notation}

\begin{lemma}\label{lm:limH(F,L)}
For any $Y\in Sm_k$,
$$\varinjlim_{\alpha\in \mathfrak A} h^i_{\affl}(L^{GW}_{\alpha}(-,Y))\simeq \varinjlim_{\alpha\in \mathfrak A} h^i_{\affl}(R^{GW}_{\alpha}(-,Y)) \simeq h^i_{\affl}(\GWCor(- \times\bGm,\bGY))$$
\end{lemma}
\begin{proof}
Lm. \ref{lm:limFNL} yields that 
$\varinjlim_{\alpha\in \mathfrak A} R^Q_{\alpha}(X,Y) =\varinjlim_{\alpha\in \mathfrak A} L^Q_{\alpha}(X,Y) = Q(\GXqGY).$
By Lm. \ref{lm:limFNL} again for any 
$(P_1,q_1),(P_2,q_2)\in Q(\mathcal P(\GXqGY))$ there is a triple of integers $(r,n,m)$ such that 
$(P_1,q_1),$ $(P_2,q_2),$ $(P_1,q_1)\oplus(P_2,q_2)\in R^Q_{r,n,m}(X,Y)\subset L^Q_{r,n,m}$.
Hence
$\varinjlim_{\alpha\in \mathfrak A} R^{GW}_{\alpha}(X,Y) =\varinjlim_{\alpha} L^{GW}_{\alpha}(X,Y) = \GWCor(\GXqGY).$
Since the direct limits of abelian groups along filtering systems are exact, the claim follows
\end{proof}

\begin{lemma}\label{lm:Haffinv}
For any $\alpha\in \mathfrak A$ and integer $i$ 
\begin{itemize}[leftmargin=15pt]
\item[1)]
$h^i_{\affl}(\rho^{GW,L}_{\alpha})\circ h^i_{\affl}( \theta^{GW}_\alpha )=h^i_{\affl}(j^{GW,N}_\alpha)
\colon h^i_{\affl}(N^{GW}_\alpha(-,Y))\to h^i_{\affl}(GWCor(-,Y))$,
\item[2)]
$h^i_{\affl}( \theta^{GW}_\infty )\circ h^i_{\affl}(\rho^{GW,R}_{\alpha})=h^i_{\affl}(j^{GW,R}_\alpha)
\colon h^i_{\affl}(R^{GW}_\alpha(-,Y))\to h^i_{\affl}(\GWCor(-\times\bGmw,Y\times\bGmw))$,
\end{itemize}
(let's remember that $\theta_\infty =-\boxtimes id_\bGmw\colon \GWCor(X,Y)\to \GWCor(\GXqGY)$
and see def \ref{def:thetaalpha} for $\theta_\alpha$, $\alpha\in \mathfrak A$).
\end{lemma}
\begin{proof}
The cylinder triangulation
and Lm. \ref{lm:leftinverse}, Lm. \ref{lm:rightinverse}
gives homotopies of the chain complexes
\begin{multline*}
h_l\colon C^*(\GWCor(-,Y))\to C^*(\GWCor(-,Y))[-1],\\
 h_r\colon C^*(L^{GW}(-,Y))\to C^*(\GWCor(-\times\bGmw,Y\times\bGmw))[-1]\end{multline*}
such that
$d(h_l(a)) = a - \rho^{GW,L}_\alpha(a\boxtimes id_{\bGmw})$,
for $a\in \GWCor(X\times \Delta^i ,Y) )$,
and 
$d(h_r(a)) = id(a) - \rho^{GW,R}(a)\boxtimes id_{\bGmw}$,
for $a\in L^{GW}(X\times \Delta^i,Y) )
.$
The claim follows.
\end{proof}

\begin{theorem}\label{th:CancGWCor} For any $X,Y\in Sm_k$ the functor $-\boxtimes id_{\bGmw}$ induces the natural quasi-isomorphism
\begin{equation*}
\GWCor(X\times\Delta^\bullet,Y)\simeq \GWCor(X\times\bGmw\times\Delta^\bullet,Y\times\bGmw)
\end{equation*}

\end{theorem}
\begin{proof}

First we shell prove that $\Coker h^i_{\affl}(\theta^{GW}_\infty) =0$.
Lm \ref{lm:Haffinv}.(2) gives us 
the commutative diagram
$$\xymatrix{ 
h^i_{\aff}( R_{\alpha}(-,Y) ) \ar[r]^{\rho^{GW,R}_{\alpha}} \ar@{^(->}[rd]_{j^{GW,R}_\alpha} & 
h^i_{\aff}( GWCor(-,Y) )\ar[d]^{ \theta^{GW}_\infty } & \\
& h^i_{\aff}( R_{\infty}(-,Y) )  \ar@{->>}[d]\ar@{->>}[rd] & \\
& \Coker( \theta^{GW}_\infty) &
\Coker(j^{GW,R}_\alpha)\ar@{->>}[l]_{c_\alpha} 
}$$
Since the category of factor objects of $R_\infty$ is an ordered set, 
for any $\alpha<\pri \alpha \in \mathfrak A$, 
$c_\alpha = c_{\pri \alpha}\circ \zeta_{\alpha,\pri \alpha}$,
where $\zeta_{\alpha,\pri \alpha}\colon  \Coker(j^{GW,R}_\alpha)\to \Coker(j^{GW,R}_{\pri \alpha}) $.
Hence there is a surjection 
$\varinjlim_{\alpha\in \mathfrak{A}}\Coker(R_{\alpha}\to R_{\infty}) 
\twoheadrightarrow
\Coker(\theta^{GW}_\infty),$
In the same time by Lemma \ref{lm:limH(F,L)} 
$\varinjlim_{\alpha\in \mathfrak{A}}\Coker(R_{\alpha}\to R_{\infty}) = \Coker(\varinjlim_{\alpha\in \mathfrak{A}} R_{\alpha}=0$
Thus $\Coker(\theta^{GW}_\infty)=0$.


Now we prove injectivity. 
For any integer triples $\alpha=(r,n,m)$, $\pri \alpha = (\pri r,\pri n,\pri m)$
homomorphisms $\theta^{GW}_\alpha$ and $\theta^{GW}_{\pri \alpha}$ 
commute with homomorphisms 
$\gamma^{GW,N}_{\alpha,\alpha^\prime}$ and $\gamma^{GW,L}_{\alpha,\alpha^\prime}$ (see Lemma \ref{lm:alphaarrows}).
Therefore $
\varinjlim_\alpha [\theta^{GW}_\alpha
\colon N^{GW}_\alpha(-,Y)\to L^{GW}_\alpha(-,Y)
]
 = 
[ \theta^{GW}_\infty
\colon \GWCor(-,Y)\to L^{GW}_\infty(-,Y)]
.$
Since injective limits of abelian groups is exact
it follows that 
$\Ker(\theta^{GW}_\infty) = \Ker(\varinjlim_{\alpha}(\theta^{GW}_\alpha) = \varinjlim_{\alpha} (\Ker(\theta^{GW}_\alpha))=0
.$
\end{proof}

\begin{corollary}\label{cor:CancSmMotCompl}
For all $A^\bullet,B^\bullet\in \DaffGW(k)$, there is a natural isomorphism
$$ 
Hom_{\DaffGW(k)}(A^\bullet, B^\bullet)\simeq Hom_{\DaffGW(k)}(A^\bullet(1), B^\bullet(1)) 
$$
\end{corollary}
\begin{proof}
Since $M^{GW}(X)(1)=M^{GW}(X)\otimes M^{GW}(\bGmw)\simeq M^{GW}(X\times \bGmw)$,
and since the objects $M^{GW}(X)$ generate the category $\DMGWeff(k)$,
it is enough to prove isomorphism
$Hom_{\DaffGW}(X,Y[i])\to 
Hom_{\DaffGW}( X\times \bGmw, Y\times \bGmw[i] )$
for all $X,Y\in Sm_k$ and $i \in \mathbb Z$.

The adjunction $\Dba(\PreGW)\leftrightharpoons \DaffGW(k)=\Daff(\PreGW)$ (see Th \ref{th:nisadjaff})
yields
$Hom_{\DaffGW(k)}(\bZGW(X), \bZGW(Y)[i])\simeq  Hom_{\DPreGWs}(\bZGW(X) , C^*(\bZGW(Y))[i]) = H^i(\GWCor(X\times\Delta^\bullet,Y).$
So the claim follows from Th \ref{th:CancGWCor}. 
\end{proof}

Consider the functor $\mathcal Hom_{\PreGWs}(\bGmw,-)\colon \PreGWs\to \PreGWs\colon F\mapsto F(-\times\bGm)$.
It is exact on 
$\PreGWs$, and inducesa  functor $\mathcal Hom_{\DPreGWs}(\bGmw,-)$ on $\DPreGWs$.
Now since the last functor commutes with the functor $C^*=\mathcal Hom_{\DPreGWs}(\Delta^\bullet,-)$, it follows that $\mathcal Hom_{\DPreGWs}(\bGmw,-)$ is exact in respect to $\affl$-quasi-isomorphisms.
Thus it induces a functor $\mathcal Hom_{\DaffGW}(\bGmw,-)$ on $\DaffGW$.

Since the functor $-\otimes \bGmw$ on $\DPreGWs$ preserves the class of morphisms $X\times\affl\to X$ it is $\affl$-exact too.
So the adjunction $-\otimes \bGmw\dashv \mathcal Hom(\bGmw,-)$ on $\DPreGWs$ yields an adjunction
\begin{equation}\label{eq:adjunctDWGaff}
- \otimes \bGmw\colon  \DaffGW(k) \leftrightharpoons \DaffGW(k) \colon \mathcal Hom(\bGmw,-) \end{equation} 
The last adjunction
is a coreflection by Corollary \ref{cor:CancSmMotCompl},
i.e. $A^\bullet \simeq \mathcal Hom(\bGmw,A^\bullet\otimes\bGmw).$
We'll show that the same holds for derived functors in $\DMGW(k)$ that is considered as localisation of $\DaffGW(k)$ in respect to Nisnevich-quasi-isomorphisms. 

\begin{proposition}\label{prop:GmshNisex}
The functors $-\otimes\bGmw$ and $\mathcal Hom(\bGmw,-)$ on the category $\DaffGW(k)$ are exact in respect to Nisnevich quasi-iso\-mor\-phisms.
\end{proposition}
\begin{proof}
The functor $-\otimes\bGmw$ is exact in respect to Nisnevich quasi-isomorphisms, since it preserves Nisnevich squares (see \cite{AD_DMGWeff} for detailed discussion).

To prove the second claim it is enough to show that for a 
locally trivial homotopy invariant presheave with GW-transfers $F$
the complex of presheaves $\mathcal Hom(\bGm,F)$ is Nisnevich acyclic.
In the same time
$\mathcal Hom(\bGm,F)=F(\Delta^\bullet\times\bGm\times -)\simeq F(\bGm\times -)$, and Lm \ref{lm:coh((subAff)_loc)} yields the claim.
\end{proof}

%
%

\begin{theorem}\label{th:fullyfaithDMW(k)}
For an infinite perfect field $k$, $\chark k \neq 2$ the canonical functors
$\DMGWeff(k)\rightarrow \DMGW(k)$
are fully faithful embeddings.
\end{theorem}
\begin{proof}
Since the functors $-\otimes_{\DaffGW(k)}\bGmw$ and $\mathcal Hom_{\DaffGW(k)}(\bGmw,-)$ on the category $\DaffGW(k)$ are exact in respect to Nisnevich quasi-isomorphisms by Proposition \ref{prop:GmshNisex}. Hence the adjunction \eqref{eq:adjunctDWGaff} yields
the adjunction  $ - \otimes \bGmw\colon  \DMGW(k) \leftrightharpoons \DMGW(k) \colon \mathcal Hom(\bGmw,-),$
and this adjunction is a coreflection too, that is equal to the claim.
\end{proof}
\begin{corollary}\label{cor:MTC}
For an infinite perfect field $k$, $\chark k \neq 2$, $X\in Sm_k$ and 
a motivic complex $A^\bullet\in \DMGWeff(k)$ there is a natural isomorphism
$Hom_{\DMGW(k)}(M^{GW}(X),\Sigma^\infty_{\bGmw}A^\bullet[i]) \simeq H^i_{Nis}(X,A^\bullet),$
\end{corollary}\begin{proof}
The claim follows immediate from Corollary \ref{cor:MPMeff} and Th \ref{th:fullyfaithDMW(k)}.
\end{proof}

\section{Appendix: the functor from the frame-correspondences to $GWCor$.}

\begin{definition}
Suppose $X$ and $Y$ are pair of schemes, then 
a \emph{frame-correspondence} on the rank $n$ between $X$ and $Y$ is 
a set $(\mathcal V,Z,\phi,g)$, where $Z$ is a closed subset in $\aff^n_X$, $v\mathcal V\to \aff^n_X$ is an etale morphism such that $v^{-1}(Z)\simeq Z$ (in other words $(\mathcal V,Z)\to (\aff^n_X,Z)$ is a Nisnevich neighbourhood), $\phi\colon \mathcal V\to \aff^n$, and $g\colon \mathcal V\to Y$ are regular maps.

Denote by $Fr_n(X,Y)$ the set of isomorphism classes of frame-correspondences on the rank $n$ between $X$ and $Y$ up to a shrinking of the Nisnevich neighbourhood $(\mathcal V,Z)\to (\aff^n_X,Z)$, 
and let $Fr_*(X,Y)=\coprod_n Fr_n(X,Y)$; 
denote by $Fr_*$ the category with objects smooth schemes and morphisms $Fr_*(X,Y)$ and the composition as defined in \cite{GarkushaPanin_FrMot};
denote by $\mathbb ZFr_*$ an additive category with objects smooth schemes and morphisms $\mathbb ZFr_*(X,Y)/([(\mathcal V,Z_1\amalg Z_2,\phi,g)]-[(\mathcal V,Z_1,\phi,g)]-[(\mathcal V,Z_2,\phi,g)])$,
and denote by $Fr_*^{aff}$ and $\mathbb ZFr_*^{aff}$ the full subcategories spanned by affine schemes.

\end{definition}

\begin{definition}
For any integer $n$ and integers $d_1,\dots d_n$ 
denote $Fr_{n,d_1,\dots,d_n}(X,Y)$ a subset  in $Fr_n(X,Y)$
consisting of frame-correspondences $(\mathcal V,Z,\phi_1,\dots,\phi_n,g)$
such that for any $i$ 
there is a section $s_i\in \Gamma(\mathbb P^n_X,\mathcal O(d_i))$: 
$s_i\big|_{\mathbb P^{n-1}_X}=x_i^{d_i}$, $\phi_i=v^*(s_i/x_0^{d_i})$,
where $(x_0,x_1,\dots x_n)$ are coordinates on $\mathbb P^n_X$, $Z(x_0)=\mathbb P^{n-1}_X=\mathbb P^n_X\setminus \mathbb A^n_X$.
\end{definition}

\begin{definition}
For any frame correspondence $\Phi = (\mathcal V,Z,\phi_1,\dots\Phi_n,g)\in Fr_n(X,Y)$
denote by $S(\Phi)$ the subset in $Fr_n(\Phi)$
consisting of elements $\tilde\Phi=(\mathcal V^\prime,Z,\phi_1^\prime,\dots,\phi_n^\prime,g)$ 
such that $(\phi_i^\prime - \phi_i) \big|_{Spec k[\aff^n]/I(Z)^2}=0 $ for all $i$.

Let's write $\Phi\stackrel{1}{\sim}\Phi^\prime$ iff $\Phi^\prime\in S(\Phi)$. This defines equivalence relation on frame-correspondences.
\end{definition}

\begin{lemma}

For any $\Phi=(\mathcal V,Z,\phi,g),\Phi^\prime=(\mathcal V^\prime,Z,\phi^\prime,g)\in Fr_n(X,Y)$ if $\Phi\stackrel{1}{\sim}\Phi^\prime$ then there is an affine homotopy connecting $\Phi$ and $\Phi^\prime$, i.e. $\Phi\stackrel{\aff^1}{\sim}\Phi^\prime$.
\end{lemma}
\begin{proof}
The required homotopy is given by $(\mathcal V\times_{\aff^n_X}\mathcal V^\prime,Z, \alpha \phi+(1-\alpha)\phi^\prime,g)\in Fr_n(X\times\affl,Y)$.
\end{proof}
\begin{lemma}\label{lm:sectexist}
Suppose $X$ is affine.
Then for any $\Phi\in Fr_n(X,Y)$
there is integer $d$, such that for any integers $d_i>d$, $i=1,\dots n$,
there is $\tilde{\Phi}\in Fr_{n,d_i,\dots d_n}(X,Y)\cap S(\Phi)$
\end{lemma}
\begin{proof}
The claim follows from that for an affine scheme $X$ the sheave $\mathcal O(1)$ on $\mathbb P^n_X$ is ample.
\end{proof}

\begin{construction}\label{constr:Fr-GW}
For any $\Phi\in Fr_{n,d_1,\dots d_n}(X,Y)$ we construct a quadratic space $Q(\Phi)\in QCor(X,Y)$ in the following way:

Consider map $f\colon \aff^n_X\to \aff^n_X$, and denote $A\to B$ corresponding homomorphism of sheaves of algebras over $X$.
Since $\phi_i$ is polynom with leading coefficient $x_i^{d_i}$ 
then $f$ is finite morphism of smooth varieties and $B$ is finite flat over $A$.
The Grothendieck duality theorem gives us 
isomorphism $$Hom_{B}(B,A)\simeq \omega_B\otimes\omega_A^{-1}.$$
Next using trivialisation of the canonical classes $\omega_B$ and $\omega_A$ defined by coordinate functions on relative affine space
we get isomorphism $$Hom_A(B,A)\simeq B.$$
Now base change along the embedding by zero section $X\to \aff^n_X$
gives us isomorphism $$Hom_{\mathcal O(X)}(\mathcal O(Z),\mathcal O(X))\simeq \mathcal O(Z).$$
\end{construction}


\begin{proposition}\label{prop:secthomot}
\begin{itemize}
\item[1)]
for any $\Phi_1,\Phi_2\in Fr_{n,d_1,\dots,d_n}(X,Y)$, $\Phi_2\in S(\Phi_2)$,
$Q(\Phi_1)\simeq Q(\Phi_2)$;
\item[2)]
for any $\Phi_1\in Fr_{n,d_1,\dots,d_n}(X,Y)$, $\Phi_2\in Fr_{n,d_1+1,\dots,d_n}(X,Y)$, $\Phi_2\in S(\Phi_2)$,
$Q(\Phi_1)\simeq Q(\Phi_2)$;

\end{itemize}
\end{proposition}
\begin{proof}
1)
Consider the frame correspondence
$\Theta=(Z\times\affl,\mathbb A^n-Z^\prime, \varphi,g)\in Fr(X\times\affl,Y)$,
where $\varphi = \lambda\varphi_1+(1-\lambda)\varphi_2$, $Z(\varphi)= Z^\prime\amalg Z\times\affl$.
The support $Z(\varphi)$ is a closed subscheme in $\aff^n_X$ finite over $X$ and so we can apply construction \ref{constr:Fr-GW} to $\Theta$.
Then $Q(\Theta)=(k[Z\times\affl],q)\in QCor(X,Y)$ for some invertible function $q\in k[Z\times\affl]^*$ and it follows from lemma \ref{lm:constQuad} that $Q(\Theta\circ i_0)\simeq Q(\Theta\circ i_1)$ and whence $GW(\Phi_0)=GW(\Theta\circ i_0)=GW(\Theta\circ i_1)=GW(\Phi_1)$. 


2)
Let $\Phi_k=(Z,\phi^k,g)$, $\phi^k=(\phi^k_i)$, $k=1,2$; let $\phi^2_1=s^2_1/x_0^{d_1+1}$, and $\phi^k_i=s^k_i/x_0^{d_i}$ for otherwise $k$ and $i$.
Consider the frame correspondence
$\Theta=(Z\times\affl,\mathbb A^n-Z^\prime, \lambda\varphi_1+(1-\lambda)\varphi_2,g)\in Fr_n(X\times\affl,Y)$,
where $\varphi = \lambda\varphi_1+(1-\lambda)\varphi_2$, $Z(\varphi)= Z^\prime\amalg Z\times\affl$.

Let $\Gamma$ be the graph of the regular map 
$(\varphi,pr)\colon \mathbb A^n\times X\times\affl\to \mathbb A^n\times X\times \affl$.
Then $\Gamma= \mathbb A^n\times\aff^n\times X\times \affl\cap \widetilde Z$, where 
$\widetilde Z=Z(\widetilde s)\subset \mathbb P^n\times\aff^n\times X\times \affl$,
$\widetilde s=(\widetilde s_i)$,
$\widetilde s_i= s_i -T_i t_{\infty}^{d_i} \in\Gamma(\mathbb P^n\times \aff^n \times X\times\affl ,\mathcal O(d_i))$,
$s_i= \lambda s^1+(1-\lambda) s_2\in \Gamma(\mathbb P^n\times X\times\affl,\mathcal O(d_i))$,
and $T_i$ denote coordinate functions on the multiplicand $\aff^n$ and $\lambda$ denotes the coordinate of the multiplicand $\affl$. 

\newcommand{\ttil}[1]{\widetilde{\widetilde #1}}
\newcommand{\til}{\widetilde}
\newcommand{\bbP}{\mathbb P}
\newcommand{\bbA}{\mathbb A}
By the same reason as in lemma \ref{lm:sectexist}
we can choose some sections
$s^\prime_i\in \Gamma(\bbP^n\times X\times\affl,\mathcal O(d_i))$,
$s^\prime_i\big|_{Z(\mathcal I(Z\times\affl)^2)} = s_i\big|_{Z(\mathcal I(Z\times\affl)^2)}$,
$s^\prime_i = x_1^{d_i}$,
$i=2,\dots n$.
Let's put $\til s^\prime_i = s^\prime_i -T_i x_0^{d_i}$, $i=2,\dots n$.

Then consider the closed subscheme
$\ttil Z = Z(\ttil s) \subset \mathbb P^n\times\aff^n\times X\times \affl\times (\aff-1)^{n-1}$,
$\ttil s=(\ttil s_i)$,
$\ttil s_i= \alpha_i \til s_i (1-\alpha_i) \til s^\prime_i  \in  
\Gamma(\bbP^n\times \aff^n \times X\times\affl \times (\aff-1)^n,\mathcal O(d))$, 
where $\alpha_i$ denote coordinates on the last multiplicand $(\aff-1)^n$.


Then $\ttil Z$ is a smooth scheme over $X$ and the projection 
$\ttil Z\to  \aff^n \times X\times\affl \times (\aff-1)^n$ is finite.
Applying the Duality theorem \ref{th:GrothendieckDualtiyThwithBaseCh}
we get the $k[\ttil Z]$-linear isomorphism
$$\ttil q^\prime\colon Hom_{k[\aff^n \times X\times\affl \times (\aff-1)^n]}(k[\ttil Z],k[\aff^n \times X\times\affl \times (\aff-1)^n])\simeq
\omega_{X\times\affl\times\aff^n}(\ttil Z) \otimes\omega(\aff^n)^{-1},$$
which can be considered as a morphism of coherent sheaves on $\ttil Z$ and let $\ttil q$ be the restriction of $\ttil q^\prime$ to the affine part $\ttil\Gamma \ttil Z\cap \aff^n\times \aff^n \times X\times\affl \times (\aff-1)^n$.
It is easy to see that $\ttil\Gamma$ is a graph of the morphism $\aff^n \times X\times\affl \times (\aff-1)^n\to \aff^n$ defined by the function $ $, hence the projection $\ttil\Gamma\to \aff^n \times X\times\affl \times (\aff-1)^n$ (here we skip the second multiplicand $\aff^n$) is isomorphism.

The fibre of $\ttil Z$ over $\aff^n\times X\times 0\times 0 $ is equal to the $X$-smooth closed subscheme $\til Z_1 = Z( s_i^1 + T_i x_0^{d_i})\subset \bbP^n\times \aff^n\times X\times 0\times 0\simeq\bbP^n\times \aff^n\times X$, and base change in the Duality theorem yields that
the fibre of $\ttil q$ is equal to the quadratic space
$\til q_1\colon Hom_{k[\aff^n\times X]}(k[\til Z_1,k[\aff^n\times X])\simeq\omega_{X}(\til Z_1)\otimes \omega(\aff^n)^{-1}$
using in the construction \ref{constr:Fr-GW} applying to $\Phi_1$.
Hence the fibre of $Q=(k[\ttil Z],\ttil q)$ over $ 0\times X\times 0\times 0$ is equal to $Q(\Phi_1)$. 
On the other side the $X$-smooth scheme $\til Z_2 = Z( s_1^2 + T_1 x_0^{d_1+1}, s_i^2 + T_i x_0^{d_i})\subset \bbP^n\times \aff^n\times X\times 1\times 0\simeq\bbP^n\times \aff^n\times X$ is the disjoint component of the fibre of $\ttil Z$ over $\aff^n\times X\times 1\times 0$ and 
by similarly to the above the fibre of $Q$ over $0\times X\times 1\times 0$ is equal to $Q(\Phi_2)$.

Thus the fibre of $Q$ over $0\times X\times \affl\times 0$ defines the quadratic space $(k[Z\times\affl],q)$ that is homotopy joining $Q(\Phi_1)$ and $Q(\Phi_2)$.
So the claim follows by lemma \ref{lm:constQuad}.

\end{proof}

Let's present the using duality theorem.
\begin{theorem}\label{th:GrothendieckDualtiyThwithBaseCh}


Suppose $f\colon Y\to X$ is finite Cohen-Macaulay morphism of smooth affine schemes over the base $S$;
then there is a $k[Y]$-linear isomorphism $q\colon Hom_{k[Y]}(k[Y],k[X])\simeq \omega(Y)\otimes \omega_S(X)^{-1}$
that is natural in respect to base changes, i.e.
for a diagram with Cartesian squares
\begin{gather}\xymatrix{
Y^\prime\ar[d]\ar[r]& X^\prime\ar[d]\ar[r]& S^\prime\ar[d]\\
Y\ar[r] & X\ar[r] & S
}
\end{gather}
we have 
$$
q^\prime= q\otimes_{k[S]} k[S^\prime] 
$$
\end{theorem}
\begin{proof}
The claim is 
a particular case of
the Duality theorem from \cite{Hart_ResDual} in combination with the Base Change theorem from \cite{Conrad_BCDTh} or \cite{Sastry_BCDThCoh-Mac}.
Another link is the proposition 2.1 in \cite{OP_WittPurity}.
\end{proof}

Using lemma \ref{lm:sectexist} and proposition \ref{prop:secthomot} we see that the construction \ref{constr:Fr-GW} define a map $Fr_n(X,Y)\to QCor(X,Y)$ for affine smooth schemes $X,Y$.
The pseudo-functoriality and the base change in the Duality theorem above 
yields that $Q(\Phi_1\circ\Phi_2)=Q(\Phi_1)\circ Q(\Phi_2)$, so we get a functor $Fr^{aff}_*\to QCor$.
Moreover this induce a functor $\mathbb ZFr^{aff}_*\to GWCor$,
since if the support $Z$ of a frame-correspondence $\Phi$ splits into disjoint union $Z=Z_1\amalg Z_2$ then for any $\Phi^\prime=(Z,s,g)\in S(\Phi)$, $\Phi_i^\prime=(Z_i,s,g)\in S((Z_i,\phi,g))$, $i=1,2$, and by construction \ref{constr:Fr-GW} $Q(\Phi^\prime)=Q(\Phi^\prime_1)\oplus Q(\Phi^\prime_2)$.

\begin{theorem}
There is a functor $Fr^{aff}_*\to QCor$ that takes $\Phi$ to $Q(\Phi)$, 
and  there is a functor $\mathbb ZFr^{aff}_*\to GWCor$ that takes $[\Phi]$ to $[Q(\Phi)]$.
\end{theorem}
%

%

\begin{corollary}
There are functors $Fr_*\to QCor_{nis}$, $\mathbb ZFr_*\to GWCor_{nis}$, where $QCor_{nis}(-,Y)$ and $GWCor_{nis}(-,Y)$ are Nisnevich sheafification of $QCor(-,Y)$ and $GWCor(-,Y)$.
\end{corollary}

\begin{remark}
We can shortify the proof of the second point of the lemma \ref{lm:sectexist}
applying the following variant of the duality theorem
to the scheme $\widetilde Z$ and open subscheme $\widetilde Z\cap \aff^n\times\aff^n\times X\times \affl$.

The duality theorem:
\emph{
Let $f\colon X\to Y$ be finite Cohen-Macaulay morphism,
and $i\colon U\hookrightarrow X$ a smooth open subscheme;
then there is an $\mathcal O(Y)$-linear isomorphism 
$\tau_{f,U}\colon \mathcal Hom_{\mathcal O(Y)}(f_*(\mathcal O(X)),\mathcal O(Y))\big|_U\simeq (f\circ i)_*\omega_S(U)$
that is natural in respect to base change and natural in respect shrinking of $U$.
I.e. 
for a diagram with Cartesian squares
\begin{gather}\xymatrix{
U^\prime\ar[r]\ar[d] &
X^\prime\ar[r]\ar[d] &
Y^\prime\ar[r]\ar[d]&
S^\prime\ar[d]\\
U \ar[r] &X\ar[r] & Y\ar[r]&S
}
\end{gather}
we have 
$$
q^\prime= q\otimes_{k[S]} k[S^\prime] ,
$$}
and if $U^\prime\subset U$ then $\tau_{f,U}\big|_{U^\prime}=\tau_{f,U}$.

Indeed, it is enough to the case of the schemes $X$ that are full intersection of sections in relative projective plain, i.e. $X=Z(s)\subset \mathbb P^n_Y$, $s=(s_1,\dots, s_n)$, $s_i\in \Gamma(\mathbb P_Y,\mathcal O(d_i))$.
In this case the Duality theorem above can be proved by using the explicit formula for dualisable complex 
given by Koszul complex $\bigwedge_{i=1,\dots,n} [\mathcal O\xrightarrow{\widetilde s_i}\mathcal O(d_i)] \in K^b(\mathbb P^n_{\aff^n})$. So the required base change property can be checked explicitly using base change for Koszul complex and the isomorphism with the canonical class follows from the splitting of the Koszul complex in this case. 
\end{remark}

\begin{remark}
It is possible to get the functor for all schemes using the Duality theorem in the following form:
For any scheme of finite type $S$ over the base field $k$ 
 there s a pseudo-functor $f^!$ from the category of morphisms of finite type of $S$-schemes 
to the triangulated (dg-categories) $f^!\colon Sch_{ft,ln}\to Tr\colon X\to D^-_{qc}(X)$ with isomorphism $f^!(X)\simeq \omega_S(X)$ for smooth $X$, and duality isomorphism $Hom_{D_{qc}(X)}(F,f^!(G))\simeq Hom_{D_{qc}(Y)}(f_*(F),G) $ for $F\in D^+_{qc}(X)$, and projective morphism $f$ that is compatible with base changes along arbitrary morphisms $S^\prime\to S$ and Cohen-Macauley $f$ and along open immersions $Y^\prime\to Y$ and arbitrary morphisms $f
$.

Let's briefly describe the construction:

Let $\Phi=(v\colon \mathcal V\to \aff^n_X,Z,\phi,g)$ be a frame-correspondence. 
Consider the morphism $f\colon\mathcal V\to \aff^n_X$ defined by $\phi$ and projection to $X$. 
Shrinking $\mathcal V$ we may assume that $f$ is quasi-finite and 
let $\mathcal V\xrightarrow{i} \overline{\mathcal V}\xrightarrow{\overline{f}} \mathbb A^n_X $ be factorisation of $f$ such that $i$ is an open immersion and $\overline f$ is finite, such factorisation exists by the Main Zarisky theorem.
Now applying the duality theorem to the finite morphism $\overline f$ we get an $\mathcal O(\overline{\mathcal V})$-linear isomorphism $\widetilde q\colon \mathcal Hom_{\mathcal O(\aff^n_X)}({\overline f}_*(\mathcal O(\overline{\mathcal V})) , \mathcal O*\aff^n_X ) = \omega^\circ_{\overline f}$, where $\omega^\circ_{\overline f}$ is dualizible complex.
Then using the isomorphism $\omega^\circ_{\overline f}\big|_{\mathcal V}\simeq \omega_{\mathcal V\to X}\otimes\omega^{-1}_{\aff^n}$ and trivialisation of $\omega_{\mathcal V\to X}$ induced by $v$ and trivialisation of $\omega_{\aff^n}$
we get an isomorphism $$\mathcal  Hom_{\mathcal O(\aff^n_X)}({\overline f}_*(\mathcal O(\overline{\mathcal V}))\big|_{\mathcal V}\simeq \mathcal O(\mathcal V).$$ Finally using base change along $0_X\to \aff^n_X$ we get $\mathcal O(Z)$-linear isomorphism $\mathcal Hom(p_*(\mathcal O(Z)),\mathcal O(X))\simeq \mathcal O(Z)$.

To prove the functoriality let's note that though $\overline f$ is not Cohen-Macaulay but it is Cohen-Macaulay over generic point of $\aff^n_X$ and so combining the base change theorem along open immersions (for arbitrary projective morphisms) with base change 
for Cohen-Macaulay morphisms (along an arbitrary morphism of schemes) we deduce that $\widetilde q$ satisfy compositions axiom. 
We leave details for further consideration.
\end{remark}


\section{Appendix: Spectral category and non-commutative category of $GW$-correspondences.}

We see that the definition \ref{def:catCohfcalP} (with the composition functor)  gives rise to the category of correspondences enriched over the exact (additive) categories with duality, and over dg-categories with duality.

\begin{definition}

Let $S$ is the noetherian base scheme of a finite dimension.
Denote by $Cat^{dual}_{Ch(S)}$ the category of dg-categories with duality (and duality preserving functors).

Denote by $\mathcal{P^D}(S)$ the category enriched over $Cat^{dual}_{Ch(S)}$  with objects being finite type schemes over $S$, and the morphism-category for $X,Y\in Sch_S$ being $(\cal P(X,Y),D_X)$ (see def. \ref{}).

Let $\mathcal{GW}_S$ be the category enriched over the spectra $\mathrm{SH}$ with objects finite type schemes over $S$ and morphism-spectra $\mathcal{GW}_S(X,Y)$ be the Hermitian K-theory spectra of $(D(\cal P(X,Y)),D_X)$

Consider the category $\mathcal{C}at^{dual}_{Ch(S)}$ enriched over $Cat^{dual}_{Ch(S)}$  with objects being dg-categories with duality $\mathcal X=(\cal P_{\mathcal X}, D_X)$ with smooth dg-category $\cal P_{\mathcal X}$, and for $\mathcal X,\mathcal Y\in \mathcal{C}at^{dual}_{Ch(S)}$ the morphism-category $\mathcal{C}at^{dual}_{Ch(S)}(\mathcal X,\mathcal Y)=Funct(\mathcal X,\mathcal Y)$ is a category
of functors $Funct(\mathcal X,\mathcal Y)$ equipped with the duality $F\mapsto D_{\mathcal Y}\circ F\circ D_{\mathcal X}$, where $D_X$ and $D_Y$ are the dualities on $\mathcal X,\mathcal Y$.

By the same way as above we can define the category $\mathcal{GW}_{nc}$ enriched over spectra form the category $\mathcal{C}at^{dual}_{Ch(S)}$.
\end{definition}
The definition of the category $\mathcal{GW}$ enriched over spectra actually is parallel to the definition of spectral category in \cite{GP-KmotAV} by Garkusha and Panin, where the case of usual K-theory is considered, though the intermediate step of the category enriched over the categories is not discussed explicitly. And following the technique of \cite{GP-KmotAV} one can defined the category of motives related to $\mathcal{GW}$.

In the same time we can consider the following definition is the encroachment of the category of non-commutative varieties (spaces) with the dualities on the dg-categories. (we follow the construction presented in \cite{MR-thesis-NM}).

\begin{definition}
Define the $(\infty,1)$-category $\mathcal{Dg}^{dual}_{idem}(S)$ as the localisation $\mathcal{Dg}^{dual} = N(Cat^{dual}_{Ch(S)})[W_{M}]$, $W_{M}$ is the class of morphisms which go to the Morita 
equivalences under the forgetful functor form $Cat^{dual}_{Ch(S)}$ to the category of dg-categories (without duality).

Denote by $\mathcal Dg^{dual}_{ft}(S)$ the subcategory of $\mathcal Dg^{dual}_{idem}(S)$ spanned by objects that goes to the dg-categories of finite type under the mentioned forgetful functor.

Denote by $\widetilde{NcS}^{dual}(S)$ the opposite category to $\mathcal Dg^{dual}_{ft}(S)$.
\end{definition}
Let's note that
the $(\infty,1)$-categories in the definition above can be equipped with symmetric monoidal structures, which we denote by the symbol $\otimes$.

\begin{remark}
In the definition above we need to use several different universes saying about additive categories with duality, the categories of additive categories with duality and the category of correspondences enriched over the last one. 
\end{remark}

Now we can apply the technique of \cite{MR-thesis-NM} to get the categories of motives form the categories of correspondences.


%
\begin{definition}
Define \[DM^{\mathcal{GW}}(S) = L_{nis}L_{\affl} \mathcal P(\mathcal{GW})[(\mathbb P^1,\infty)^{-1}],\
DM^{\mathcal{GW}}_{nc}(S) =L_{nc-nis}L_{\affl,nc} \mathcal P(\mathcal{GW}_{nc}),\]
where 
$\mathcal P(\mathcal C)$ denotes the category of functors on $\mathcal C$ with values in $\mathbf{SH}$, 
$L_{\affl}$, $L_{\affl,nc}$, $L_{nis}$, $L_{nc-nis}$ are the localisation with respect to the classes of morphisms 
$w_{\affl}$, $w_{nc\affl}$, $w_{nis}$, $w_{ncNis}$;
$w_{\affl}$ is the class of morphisms of the form $X\times\affl\to X$, $X\in Sm_S$, 
$w_{nc\affl}$ is the class of morphisms of the form $\mathcal X\otimes\affl\to \mathcal X$, $\mathcal X\in \widetilde{NcS}(S)$,
$w_{ncNis}$ of the form $U\coprod_{\pri U} \pri X\to X$ defined by Nisnevich squares $(U^\prime,\pri X)\to (U,X)$,
$w_{ncNis}$ defined in a similar way with respect to squares defined like as in def. 6.4.6, 6.4.7 with straightforward replacement of dg-categories with dg-categories with duality, and functors by the preserving duality functors.
\end{definition}

\begin{proposition}
The one see that there is a sequence of functors $\mathbf{SH}(k)\rightarrow DM^{\mathcal{GW}}(k)\to DM^{\mathcal{GW}}_{nc}(k)$.
$Hom_{DM^{\mathcal{GW}}_{nc}(k)}(X,pt[i])\simeq GW^i(X)$
\end{proposition}
\begin{proof}
The functor are defined by the universal properties and due that fact that $(D(\cal P(\mathbb P^1)),D_{\mathbb P^1},1)$ is invertible in $DM^{\mathcal{GW}}_{nc}(k)$ (because of full exceptional set of linear bundles on $\mathbb P^1$). 

The second claim  is because the functor $Hom_{Cat^{dual}_{Ch(S)}}(\mathcal X,pt)\simeq (\cal P(X),D_X)$ respects $w_{nc\affl}$ and $w_{ncNis}$.
\end{proof}

\begin{conjecture}
The canonical functor $\mathbf{SH}(k)_{\mathbb Q}\to DM^{\mathcal{GW}}_{nc}(k)_{\mathbb Q}$ is fully faithful.
\end{conjecture}
The conjecture is based on that as follows form the recent results by Garkusha and result by Bachmann and Fasel
$DM^{GW}(k)_{\mathbb Q}\simeq \mathbf{SH}(k)_{\mathbb Q}$.
So the full embedding form the conjecture above probable is decomposed as
$$\mathbf{SH}(k)_{\mathbb Q}\simeq DM^{GW}(k)_{\mathbb Q}\simeq DM^{\mathcal{GW}}(k)_{\mathbb Q}\hookrightarrow  DM^{\mathcal{GW}}_{nc}(k)_{\mathbb Q} .$$ 
Actually the second question should follow form the general equivalence relation between $(\infty,1)$-categories and cateogries enriched over $\mathbf{H_{\bullet}}$.
The question which looks being the most difficult (and doubtful) is the equivalence $DM^{GW}(k)_{\mathbb Q}\simeq DM^{\mathcal{GW}}(k)_{\mathbb Q}$.

Another point we'd like to discuss is the following.
In the definition of $\widetilde{NcS}(S)$ we replaced the objects .
Similar to that fact that 
The correspondences given by GW groups or an $SL$-orientable cohomology theory can be defined in the category $Sm_S$ and in the same time on the category of smooth varieties equipped with a line bundle, which has effect for the twisting of the cohomology groups. 
The corresponding categories of motives are equivalent with the identification of a pair $(X,L)$ and a pair of varieties $(T_L,T_L-0_X)$ (i.e. $Cone(T_L-0_X\to T_L)$, where  $T_L$ is the Tome space of the bundle $L$, and $0_X$ is the zero section.
In the same time in the context of non-commutative varieties, it is known that smooth proper non-commutative spaces satisfy duality theorem and so there is a canonical duality which we can use in the definition of $\widetilde{NcS}$ restricted to a smooth proper non-commutative varieties.
So it looks being natural to ask the question:
\begin{question}
To reconstruct the category $\widetilde{NcS}(S)$ or $DM^{\mathcal{GW}}_{nc}(k)$ without replacement of the objects of $NcS$ (which are dg-categories with out any additional structures) 
and with replacement morphisms only.
\end{question}

Let's give a version of the answer which is based on the equivalence of the information given by functor $f^*$ of the categories with dualities and the information given by the pair of adjunctions $f^*\dashv f^*$ and $f^!\dashv f_!$.
Define $Cat_{Ch(S)}^{\cdot -\cdot}$ as the category with objects being dg-categories and any morphism form $\mathcal A$ to $\mathcal B$ is a pair of dg-functors $F,G$ with natural equivalence $t\colon F\simeq G$.  
$NcS^{\cdot -\cdot}$ which is defined in a similar way to $NcS$ starting form $Cat_{Ch(S)}^{\cdot -\cdot}$.
\begin{conjecture}
The  $(\infty,1)$-category $DM^{\mathcal{GW}}_{nc}(k)$ is equivalent to $L_{nc-nis}L_{\affl,nc} \mathcal P(NcS^{\cdot -\cdot})$.
\end{conjecture}


\begin{thebibliography}{00}

\bibitem{AnGarPanin_CancellationThFrMot}
A. Ananyevskiy, G. Garkusha, I. Panin,
Cancellation theorem for framed motives of algebraic varieties,
arXiv:1601.06642
(Jan 2016).
\bibitem{ALP_WittSh-etsinvert}
A.~Ananyevskiy, M.~Levine, I.~Panin
Witt sheaves and the $\eta$-inverted sphere spectrum,
(Apr 2015), arXiv:1504.04860.

\bibitem{FB_EffSpMotCT}
T.~Bachmann, J.~Fasel,
On the effectivity of spectra representing motivic cohomology theories, arXiv:1710.00594.

\bibitem{Bal_DerWitt}  P.~Balmer, Witt groups Handbook of K-theory, vol. 2, Springer, Berlin (2005), 539-576.

\bibitem{CF_FinChWittCor}
B.~Calmes, J.~Fasel,
Finite Chow-Witt correspondences,
arXiv:1412.2989
\bibitem{ChepInjLocHIiWtrPreSh}
K. Chepurkin,	
Injectivity theorem for homotopy invariant presheaves with W-transfers, 
St. Petersburg Mathematical Journal, 2017, 28:2, 291–297.

\bibitem{AD_DMGWeff}
A.~Druzhinin, Effective Grothendieck-Witt motives and Witt-motives of smooth varieties, arXiv:1709.06273.

\bibitem{AD_StrHomInv}
A.~Druzninin,
Strictly homotopy invariance of Nisnevich sheaves with GW-transfers and Witt-transfers, arXiv:1709.05805.

\bibitem{AD_CancellGWCor}
A.~Druzninin, Cancellation theorem for Grothendieck-Witt-cor\-res\-pon\-den\-ces and Witt-correspondences. arXiv:1709.06543.

\bibitem{AD_WtrSh}
A.E.~Druzhinin, Preservation of homotopy invariance for presheaves with  $\mathrm{Witt}$-transfers under Nisnevich seafication, 
Journal of Mathematical Sciences (2015), Vol. 209, issue 4, pp 555-563 

\bibitem{FaselOstvaer_CancThMilnorWittCor}
J.~Fasel, P.~\O stvaer,
A cancellation theorem for Milnor-Witt correspondences,
arXiv:1708.06102.
(Aug 2017).

\bibitem{GG_RecRatStMot}
G.~Garkusha,
Reconstructing rational stable motivic homotopy theory, arXiv:1705.01635, 2017. 

\bibitem{GP-KmotAV}
G.~Garkusha, I.~Panin
K-motives of algebraic varieties (2011), arXiv:1108.0375.


\bibitem{GarkushaPanin_FrMot} G. Garkusha, I. Panin, Framed motives of algebraic varieties (after V. Voevodsky), arXiv:1409.4372

\bibitem{HoundMahe_ComConRAlgVar}
J.~Houdebine, L.~Mahé. 
Séparation des composantes connexes
réelles dans le cas des variétés projectives. In Real algebraic geometry and quadratic forms (Rennes, 1981), volume 959 of Lecture Notes in Math., pages 358–370. Springer-Verlag, Berlin, 1982.

\bibitem{Jacobson_RealChIWitt}
J.~Jacobson,
Real cohomology and the powers of the fundamental ideal in the Witt ring,
Ann. K-Theory
Volume 2, Number 3 (2017), 357-385.

\bibitem{Mahe_SignRCon}
L.~Mahé. Signatures et composantes connexes. Math. Ann., 260(2):191– 210, 1982.

\bibitem{MR-thesis-NM} Marco Robalo, ￼￼ Th ́eorie Homotopique Motivique des Espaces Noncommutatifs, thesis.

\bibitem{MVW_LectMotCohom}
Mazza, Voevodsky, Weibel, Lecture notes on motivic cohomology, Clay mathematics monographs, ISSN 1539-6061; volume 2.

\bibitem{Suslin-GraysonSpectralSeq}
A.~Suslin,
On the Grayson spectral sequence,
Proceedings of the Steklov Institute of Mathematics. 2003.

\bibitem{SV_Bloch-Kato}
A.~Suslin, V.~Voevodsky,
Bloch-Kato conjecture and motivic cohomology with finite coefficients,
The Arithmetic and Geometry of Algebraic Cycles, 548 (2000), 117--189.

\bibitem{Voevodsky_CancellationTh} 
V. Voevodsky, The Cancellation Theorem, Voevodsky, V., "Cancellation theorem", Doc. Math., pp. 671–685, 2010. 


\bibitem{VSF_CTMht_Ctpretr}
V.~Voevodsky, Cohomological theory of presheaves with transfers, 
(in Cycles, Transfers and Motivic Homology Theories. by Eric. M. Friedlander, and Andrei Suslin), 
Annals of Math. Studies, 1999.

\bibitem{VSF_CTMht_DM}
V.~Voevodsky, Triangulated categories of motives over a field, 
(in Cycles, Transfers and Motivic Homology Theories. by Eric. M. Friedlander, and Andrei Suslin), 
 Annals of Math. Studies, 1999.

\bibitem{W_MotComK} 
M. Walker,
Motivic Complexes and the K-Theory of Automorphisms, 
vol. 0205 (1997).

\bibitem{Hart_ResDual}
R.~Hartshorne,
Residues and duality,
Lecture Notes in Mathematics \textbf{20},
Springer-Verlag,
New York, 1966.


\bibitem{Conrad_BCDTh}
B.~Conrad,
Grothendieck Duality and Base Change.
Lecture Notes in Mathematics, 
Springer-Verlag, Berlin, 2000

\bibitem{Sastry_BCDThCoh-Mac}
P.~Sastry,
Base change and Grothendieck duality for Cohen-Macaulay maps,
Compositio Mathematica,
Volume 140, Issue 3 May 2004 , pp. 729-777

\bibitem{OP_WittPurity}
Ojanguren M., Panin I.
A Purity theorem for the Witt Group.
Ann. Sci. Ecole Norm. Sup. 
(4) 32 (1999) 
no.1, 71 86.

\end{thebibliography}

\end{document}